\documentclass[11pt]{article}
\usepackage{graphicx,wrapfig,lipsum}
\usepackage{amsmath,amssymb,graphicx} 
\usepackage{amsfonts}
\usepackage[margin=0.9in]{geometry} 
\usepackage{enumerate}
\usepackage[toc]{appendix}
\usepackage{hyperref}
\usepackage{fullpage}
\usepackage{tikz}
\usepackage{bbm}
\usepackage{comment}
\newtheorem{definition}{Definition}[section]
\newtheorem{theorem}{Theorem}
\newtheorem{lemma}{Lemma}
\newtheorem{remark}{Remark}
\usepackage{comment}
\usepackage{enumitem}
\usepackage[utf8]{inputenc}
\usepackage{xcolor}

\newenvironment{proof}{\paragraph{Proof:}}{\hfill$\square$}
\allowdisplaybreaks
\begin{document}
	
	
	\title{Scaling effects on the periodic homogenization  of a reaction-diffusion-convection problem posed in homogeneous domains connected by a thin composite layer}

\author{Vishnu Raveendran$^{a,*} $, Emilio N.M. Cirillo$^b$, Ida de Bonis $^{c}$, Adrian Muntean$^a$, \\
$^a$ Department of Mathematics and Computer Science, Karlstad University, Sweden\\
$^b$ Dipartimento di Scienze di Base e Applicate per l’Ingegneria, \\ Sapienza Universit`a di Roma, Italy\\
$^c$ Universit`a degli Studi “Giustino Fortunato”, Benevento, Italy.\\
* vishnu.raveendran@kau.se}

	\date{\today} 
	\maketitle
	
	\begin{abstract}\label{abstract}
	We study the question of periodic homogenization of a variably scaled reaction-diffusion problem with non-linear drift posed for a domain crossed by a flat composite thin layer. The structure of the non-linearity in the drift was obtained in earlier works as hydrodynamic limit of a totally asymmetric simple exclusion process (TASEP) process for a population of interacting particles crossing a domain with obstacle. 
		
		Using energy-type  estimates as well as concepts like thin-layer convergence and two-scale convergence, we derive  the homogenized evolution equation and the  corresponding effective model parameters for a regularized problem. Special attention is paid to the derivation of the effective transmission conditions across the separating limit interface in essentially two different situations:  (i) finitely thin layer and (ii) infinitely thin layer. 
		
		This study should be seen as a preliminary step needed for the investigation of averaging fast non-linear drifts across material interfaces -- a topic with direct applications in the design of thin composite materials meant to be impenetrable to high-velocity impacts.
	\end{abstract}
		\\
		{\bf MSC2020}: Reaction-Diffusion-Convection equation; homogenization; thin layer; dimension reduction; Galerkin method;  two scale convergence; effective transmission condition.
\maketitle

\section{Introduction}\label{In}

\par Reaction-diffusion equations posed  for thin layers endowed with periodic microstructures arise as mathematical models for a large number of real-world applications. Prominent examples refer, for instance,  to blood flow through the blood vessels (here one considers the blood vessel walls as thin membranes with periodic microstructures), membrane filtration (see \cite{herterich2017optimizing}),  passage of oxygen particles through paperboard or through some other paper-based packaging materials (see \cite{nyflott2015influence}), formation of fingers in smoldering combustion \cite{fatima2014homogenization},   heat and current flow through thin organic light-emitting diodes (OLEDs) mounted on glass substrates \cite{grigornikahom}.

\par In this paper, we study the effect of varying scalings on the periodic  homogenization  and eventual dimension reduction of a perforated thin layer hosting  diffusion, chemical reactions, and nonlinear drift\footnote{The drift term is here the gradient of  a bounded, possibly  discontinuous polynomial. To keep things simple, we  use a suitable mollification of the drift to gain extra regularity. The mollifier function has support within $\overline{B(0,\delta)}$, where $\delta>0$ is independent of $\varepsilon$. We choose our mollifier such that,  as $\delta \rightarrow 0$, the regularized drift  converges strongly to the original nonlinear drift  in $L^{p}(\mathbb{R}^{2})$ for all $p\in (0,\infty )$.} as derived earlier as mean-field limit for a totally asymmetric simple exclusion process (TASEP) on a lattice; see \cite{CIRILLO2016436}. As microscopic domain $\Omega_{\varepsilon}\subset \mathbb{R}^{2}$, we have  two regions $\Omega_{\mathcal{L}}^{\varepsilon}$ and  $\Omega_{\mathcal{R}}^{\varepsilon}$ glued together through standard transmission conditions  via a static flat thin layer $\Omega_{\mathcal{M}}^{\varepsilon}$ (see Fig \ref{micmod}). The thin layer $\Omega_{\mathcal{M}}^{\varepsilon}$ is made of an array of periodic microstructures, while the sets $\Omega_{\mathcal{L}}^{\varepsilon}$ and  $\Omega_{\mathcal{R}}^{\varepsilon}$ are in fact non-oscillating. To describe the internal structure of $\Omega_{\mathcal{M}}^{\varepsilon}$, we replicate a reference cell $Z$ (see Fig \ref{scell}), whose height is scaled by $\varepsilon$ and its width by $\kappa(\varepsilon$). In our case, the assumed periodicity acts only in vertical direction. In each of the regions $\Omega_{\mathcal{L}}^{\varepsilon}$ and  $\Omega_{\mathcal{R}}^{\varepsilon}$,  the coefficients of the partial differential equations are independent of  $\varepsilon$. Instead, within the set $\Omega_{\mathcal{M}}^{\varepsilon}$ the coefficients of the evolution equation are assumed to satisfy a variable scaling. To be specific, 
we consider that both the time derivative term and the production-by-reaction term are scaled by $\varepsilon^{\alpha}$, the diffusion coefficient is  scaled by $\varepsilon^{\beta}$, while the drift term is scaled by $\varepsilon^{\gamma}$. The boundary production terms at the oscillating boundaries are proportional to $\varepsilon^{\xi}$. Here $\alpha$, $\beta$, $\xi$, $\gamma\in \mathbb{R}$ are dimensionless parameters. It is worth noting that the factors $\varepsilon^{\alpha}$, $\varepsilon^{\beta}$, and $\varepsilon^{\gamma}$ are referred to in the chemical engineering literature as Damk\"ohler numbers, while  $\varepsilon^{\xi}$ resembles the Thiele modulus (also called surface Damk\"ohler number). They are all ratios of characteristic time scales of pairwise combinations of partial physical processes; see e.g. \cite{cussler2009diffusion}. For instance,  $\varepsilon^{\beta}$ is of order of $\mathcal{O}\left(\frac{t_{diff}}{t_{reac}}\right)$, where $t_{diff}$ and $t_{reac}$ are the characteristic time scales of diffusion, and respectively, of reaction. The boundary conditions are chosen such that they correspond to the original interacting particle systems scenario. Consequently, we take non-homogeneous Dirichlet boundary conditions on the vertical boundaries of $\Omega_{\varepsilon}$ and non-homogeneous Neumann boundary conditions on the rest of the boundaries. 

\par Our main goal is to study how the different choices of the parameters $\alpha, \beta, \gamma$, and $\xi$ affects the structure of the upscaled equations, i.e. when  $\varepsilon \rightarrow 0$. From the modeling point of view, the main  interest lies in learning which limit transmission conditions correspond to the cases: (i) the finitely thin layer (Fig. \ref{fig5}) and (ii) the infinitely thin layer (Fig. \ref{fig4}) and how does depend on the choice of the overall scaling. In this context,  we set for (i) $\kappa(\varepsilon)=\varepsilon$, while  for (ii)  we consider $\kappa(\varepsilon)$ to be a constant independent of $\varepsilon$. Other choices of scaling of the geometry are also possible, especially if we extend the current discussion from 2D to a scenario in 3D. However, we believe that we captured the main ones, especially from the application point of view. 	This study should be seen as a preliminary step needed for the investigation of averaging fast non-linear drifts across material interfaces -- a topic with direct applications in the design of thin composite materials meant to be impenetrable to high-velocity impacts. Most of the upscaled models receive a double-porosity type structure; see  \cite{arbogast1990derivation} for more in this direction.

The main tools used in this context to derive the wanted upscaled evolution equations and corresponding transmission condition for a large variety of  choices of scalings include the energy method (see the basic idea of playing with variable scalings in \cite{peter2008different} or in \cite{VoAnh1282488}) combined with classical two-scale convergence/compactness results (see \cite{lukkassen2002two}) and two-scale convergence/compactness for thin layers (see \cite{neuss2007effective}). The current main difficulties lie in deriving $\varepsilon$-independent estimates for  all scaling options so that passing to the homogenization limit becomes possible in each case, dealing with the the non-linearity of the drift, as well as varying  $\kappa(\varepsilon)$. In this context, we bring in rigorous mathematical analysis results complementing our formal asymptotic calculations reported in \cite{cirillo2020upscaling}. As future step, our investigation will attempt to deal with fast drifts, that is it will be about entering the regime of $\gamma<0$.

\par For a basic understanding of homogenization theory in the broader context of asymptotic analysis,  we refer the reader to the standard monographs \cite{donato1999homogenization}, \cite{papanicolauasymp}, \cite{mei2010homogenization}, and \cite{bakhvalovhomogenization}, e.g.  Classical two-scale convergence and compactness result can be found in \cite{nguetseng1989general} and \cite{allaire1992homogen} ; see also  \cite{lukkassen2002two}. The earliest result that we know regarding homogenization and dimension reduction for a thin layer  including also a drift with a Navier-Stokes-type nonlinearity is \cite{Marui2000TwoscaleCF}; see also \cite{rohde2020homogenization} for a more recent account. The simultaneous homogenization  and dimension reduction of reaction-diffusion equations with nonlinear reaction rates  posed in a thin heterogeneous layer have been carefully studied in \cite{neuss2007effective}. In {\em loc. cit.}, the authors introduced a number of new techniques to derive effective transmission conditions along the layer. Our work follows very much the spirit of this paper, as well as of the follow-up investigations for thin channels \cite{effectiveapratim} and \cite{gahn2020singular}. 
More research is available on the simultaneous homogenization and dimension reduction. We mention here but a few of them which we think are closer to our investigations. Linear reaction-diffusion-convection equations coupled with non-linear surface chemical reactions for infinitely thin layers were studied in \cite{fatima2014homogenization} in the context of smoldering combustion. In \cite{fabricius2020pressure}, the 
 authors studied pressure-driven Stokes flow through a infinitely thin layer. A  double porosity scenario with jumps at sharp heterogeneities was studied in \cite{bunoiu2020upscaling}. Further work related to homogenization of infinitely thin layers, sharp interfaces, and other geometric singularities can be found in \cite{haller2021pressure}, \cite{stelzig_2012}, and \cite{AMAR2019111562}.
 This list of potentially relevant references is not exhaustive.
 
It is worth mentioning that it is a challenge to approximate numerically the obtained upscaled models (compare e.g. \cite{starnoni2021modelling}). However, due to the scale separation between the microscopic and the  macroscopic characteristic length scales, high performance computing strategies are available to handle efficient approximations of such double-porosity like models (dimensionally-reduced or not). We refer the reader, for instance,  to \cite{richardson2021parallel} and references cited therein for a possible parallelization strategy. 

\par We organize our paper as follows: In section \ref{mm},  we describe the model problem, its variable scaling, and  introduce the boundary and initial conditions including the perfect transmission conditions. To work with a problem having homogeneous Dirichlet boundary condition on vertical boundaries, we use an affine  transformation of the original problem and obtain the transformed problem with homogeneous Dirichlet boundary. The downside of employing the transformation is that we lose the  perfect transmission condition on the boundaries between the bulk regions and thin layer. 
In section \ref{wsmp}, we prove the existence and uniqueness of $\varepsilon$-dependent weak solution to our microscopic problem via the Galerkin method (see e.g. the standard lines of arguments from p.314 in \cite{evans2010partial}). In section \ref{tsctm}, we prove $\varepsilon$-independent energy estimates for the solution of microscopic problem later and point out that we can use the well-established concept of two-scale convergence for thin layers to treat our infinitely thin layer case. By using energy-type estimates and  compactness results we derive the two-scale limit equations of the upscaled problem. 
In section \ref{macro}, we make  choices for $\alpha$, $\beta$, $\gamma$,  and $\xi$ that we deem as potentially relevant to derive the corresponding upscaled equations and effective coefficients. In the final section, we propose an approximation of solutions of the non-regularized upscaled problem by using a direct method inspired from \cite{Pokorny}.

	\section{Microscopic model}\label{mm}
	\subsection{Setting of the problem}\label{sop}
	In this section, we describe the microscopic reaction-diffusion-drift model we have in mind.  The geometry where our equations are posed is sketched in Fig. \ref{micmod}.
	
	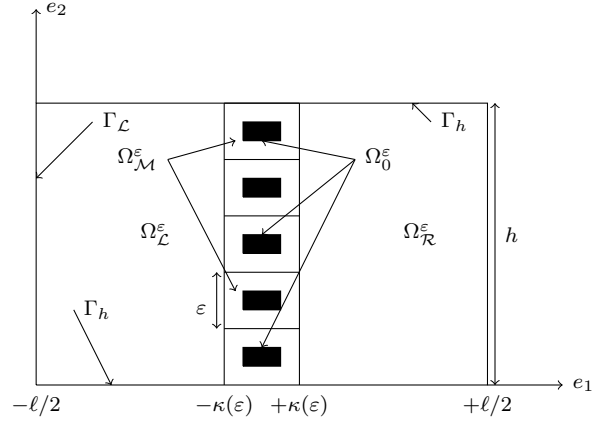
\begin{figure}[ht]
		\begin{center}
			\begin{tikzpicture}
			\draw (0,0) node [anchor=north] {{\scriptsize $-\ell/2$}} to (6,0) node [anchor=north] {{\scriptsize $+\ell/2$}}  to (6,3.75) to (0,3.75) to (0,0);
			\draw (2.5,0)node [anchor=north] {{\scriptsize $-\kappa(\varepsilon)$}}   to (2.5,3.75)  ;
			\draw (3.5,0)node [anchor=north] {{\scriptsize $+\kappa(\varepsilon)$}}to (3.5,3.75); 
			\draw (2.5,.75) to (3.5,.75);
			\draw(2.5,1.5) to (3.5,1.5);
			\draw(2.5,2.25) to (3.5,2.25);
			\draw(2.5,3) to (3.5,3);
			\draw(2.5,3.75) to (3.5,3.75);
			\draw [fill](2.75,.25) to (3.25,.25) to(3.25,.50) to (2.75,.50) to (2.75,.25);
			\draw [fill](2.75,1) to (3.25,1) to(3.25,1.25) to (2.75,1.25) to (2.75,1);
			\draw [fill](2.75,1.75) to (3.25,1.75) to(3.25,2) to (2.75,2) to (2.75,1.75);
			\draw [fill](2.75,2.5) to (3.25,2.5) to(3.25,2.75) to (2.75,2.75) to (2.75,2.5);
			\draw [fill](2.75,3.25) to (3.25,3.25) to(3.25,3.5) to (2.75,3.5) to (2.75,3.25);
			\draw [->](.75,3.5) node[anchor=west] {{\scriptsize $\Gamma_{\mathcal{L}} $}} to (0,2.75) ;
			\draw[->](0,3.75) to (0,5) node[anchor=west] {{\scriptsize $e_{2}$}};
			\draw[->](6,0) to (7,0) node[anchor=west] {{\scriptsize $e_{1}$}};
			\draw[->](.5,1) node[anchor=west] {{\scriptsize $\Gamma_{h}$}} to (1,0);
			\draw [<->] (6.1,0) to (6.1,3.75);
			\draw (6.1,2) node[anchor=west] {{\scriptsize $h$}};
			\draw (1.25,2) node[anchor=west] {{\scriptsize $\Omega_{\mathcal{L}}^{\varepsilon}$}};
			\draw (4.75,2) node[anchor=west] {{\scriptsize $\Omega_{\mathcal{R}}^{\varepsilon}$}};
			\draw[->](5.25,3.5) node[anchor=west] {{\scriptsize $\Gamma_{h}$}} to (5.00,3.75);
			\draw[->] (1.75,3) node[anchor=east] {{\scriptsize $\Omega_{\mathcal{M}}^{\varepsilon}$}} to (2.65,3.25);
			\draw[->] (1.75,3) to (2.65,1.25);
			\draw[->] (4.24,3) node[anchor=west] {{\scriptsize $\Omega_{0}^{\varepsilon}$}} to (3,3.25);
			\draw[->] (4.24,3)  to (3,.5);
			\draw[->] (4.24,3)  to (3,2);
			\draw [<->] (2.4,.75) to (2.4,1.5);
			\draw (2.4,1) node[anchor=east] {{\scriptsize $\varepsilon$}};
			\end{tikzpicture}
			\caption{Schematic representation of the microscopic model. }
			\label{micmod}
		\end{center}
	\end{figure}

Let $\varepsilon, \ell, \kappa (\varepsilon), \kappa, h, T>0$ with $ \frac{h}{\varepsilon}\in \mathbb{N}$, $\kappa(\varepsilon)=\varepsilon $ in the case of infinitely thin layer and $k(\varepsilon)=\kappa$ for the case of finitely thin layer. $\Omega$ be a two dimensional strip defined as $\Omega:=[-\ell /2,+\ell /2]\times [0,h]$.  Define $Y:=(-1,1)\times (0,1)$ and the standard cell $Z$ as $Y$ with an impenetrable compact rectangle called obstacle (denote as $Y_{0}$) with $Y_{0}=[a_{1},b_{1}]\times [a_{2},b_{2}] $ which  is placed in the center of the $Y$ (i.e $Z:=Y\backslash Y_{0}$ ).  Assume that $\partial Y_{0} $ is Lipchitz boundary and $\partial Y\cap {Y_{0}}=\emptyset $ (see Fig. \ref{scell}).

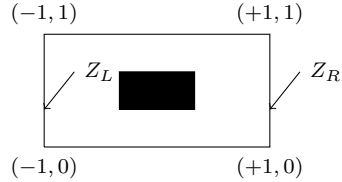
\begin{figure}[ht]
	\begin{center}
		\begin{tikzpicture}
		\draw (0,0) node [anchor=north] {{\scriptsize $(-1,0)$}} to (3,0) node [anchor=north] {{\scriptsize $(+1,0)$}}  to (3,1.5)node [anchor=south] {{\scriptsize $(+1,1)$}} to (0,1.5)node [anchor=south] {{\scriptsize $(-1,1)$}} to (0,0);
		\draw [fill] (1,.5) to (2,.5) to (2,1) to (1,1) to (1,.5);
		\draw[<-](0,0.5) to (0.4,1) node[anchor=west] {{\scriptsize $Z_{L}$}};
		\draw[<-](3,0.5) to (3.4,1) node[anchor=west] {{\scriptsize $Z_{R}$}};
		\end{tikzpicture}
		\caption{The standard cell $Z$ exhibiting a rectangular obstacle placed in the center.}
		\label{scell}
	\end{center}
\end{figure}

We define our microscopic domain $\Omega^{\varepsilon}\subset \Omega$ as 
\begin{equation}
\Omega^{\varepsilon}:=\Omega_{\mathcal{L}}^{\varepsilon}\cup \Omega_{R}^{\varepsilon}\cup \Omega_{\mathcal{M}}^{\varepsilon}\cup \mathcal{B_{L}}^{\varepsilon}\cup\mathcal{B_{R}}^{\varepsilon},
\end{equation}
where 

\begin{equation}\label{omg}
\begin{aligned}
\Omega_{\mathcal{L}}^{\varepsilon}:&=\left(-\ell /2,-\kappa(\varepsilon)\right)\times \left(0,h\right),\\
\Omega_{\mathcal{R}}^{\varepsilon}:&=\left(\kappa(\varepsilon),\ell /2\right)\times (0,h),\\
\Omega_{\mathcal{M}}^{\varepsilon}:&=((-\kappa(\varepsilon),\kappa(\varepsilon{}))\times (0,h))\backslash \Omega_{0}^{\varepsilon}
\end{aligned}
\end{equation}
where $e_{1}, e_{2}$ are standard unit vectors in $\mathbb{R}^{2}$, $k_{0}=\frac{h}{\varepsilon}$ and we denote the union of obstacles as $\Omega_{0}^{\varepsilon}$ 
 \begin{equation}
\Omega_{0}^{\varepsilon}:=\cup_{k=0}^{k_{0}}\left(ke_{2}+\left(\kappa(\varepsilon)[a_{1},b_{1}]\times \varepsilon [a_{2},b_{2}]\right)\right).
\end{equation}
We refer $\Omega_{\mathcal{M}}^{\varepsilon}$ as layer and $\Omega_{\mathcal{R}}^{\varepsilon}, \Omega_{\mathcal{L}}^{\varepsilon}$ as bulk region, and we cover boundary of $\Omega^{\varepsilon}$ by the following sets,
\begin{equation}\label{}
\begin{aligned}
\mathcal{B_{L}^{\varepsilon}}:=&\{-\kappa(\varepsilon)\}\times (0,h),\\
\mathcal{B_{R}^{\varepsilon}}:=&\{\kappa(\varepsilon)\}\times (0,h),\\
\Gamma_{\mathcal{L}}:=&\left\{-\frac{\ell}{2}\right\}\times [0,h],\\
\Gamma_{\mathcal{R}}:=&\left\{\frac{\ell}{2}\right\}\times [0,h],\\
\Gamma_{v}:=&\Gamma_{\mathcal{L}}\cup \Gamma_{\mathcal{R}},\\
\Gamma_{h}^{\varepsilon}:=&\left(\partial\Omega_{\mathcal{L}}^{\varepsilon} \cup \partial \Omega_{\mathcal{R}}^{\varepsilon}\right)\backslash \left(\mathcal{B_{L}}^{\varepsilon}\cup \mathcal{B_{R}^{\varepsilon}}\cup \Gamma_{v}\right),\\
\Gamma_{0}^{\varepsilon}:=&\partial \Omega_{\mathcal{M}}^{\varepsilon}\backslash \left(\mathcal{B_{L}}\cup \mathcal{B_{R}}\right).
\end{aligned}
\end{equation}
Note that $\partial\Omega^{\varepsilon}=\Gamma_{v}\cup \Gamma_{h}^{\varepsilon} \cup\Gamma_{0}^{\varepsilon} $.  The external unit normal vectors at  $\partial\Omega_{\mathcal{L}}^{\varepsilon}, \partial\Omega_{\mathcal{R}}^{\varepsilon}, \partial\Omega_{\mathcal{M}}^{\varepsilon}$ are denoted by $n_{l}, n_{r}, n_{m}^{\varepsilon}$ respectively.

\par  We consider the following reaction diffusion problem. Find $(u_{l}^{\varepsilon},u_{m}^{\varepsilon},u_{r}^{\varepsilon})$ satisfying the following equation
	\begin{equation}\label{rd1}
	\begin{aligned}
	\frac{\partial u_{l}^{\varepsilon}}{\partial t} +\mathrm{div}(-D_{L}\nabla u_{l}^{\varepsilon}+ B_{L}P_{\delta}(u_{l}^{\varepsilon}))&=f_{l} \;\;\;\; \mbox{on}  \ \Omega_{\mathcal{L}}^{\varepsilon} \times (0,T),\\
	\frac{\partial u_{r}^{\varepsilon}}{\partial t} +\mathrm{div}(-D_{R}\nabla u_{r}^{\varepsilon}+ B_{R}P_{\delta}(u_{r}^{\varepsilon}))&=f_{r} \;\;\;\; \mbox{on}  \ \Omega_{\mathcal{R}}^{\varepsilon} \times (0,T),\\
	\varepsilon^{\alpha}\frac{\partial u_{m}^{\varepsilon}}{\partial t} +\mathrm{div}(-\varepsilon^{\beta}D_{M}^{\varepsilon}\nabla u_{m}^{\varepsilon}+ \varepsilon^{\gamma}B_{M}^{\varepsilon}P_{\delta}(u_{m}^{\varepsilon}))&=\varepsilon^{\alpha}f_{m}^{\varepsilon} \;\;\;\; \mbox{on}  \ \Omega_{\mathcal{M}}^{\varepsilon} \times (0,T),
	\end{aligned}
	\end{equation}
	where $\alpha, \beta, \gamma \in \mathbb{R}$ and the parameter $\delta>0$ is fixed, $f_{l}:\Omega_{\mathcal{L}}^{\varepsilon} \rightarrow \mathbb{R}$, $f_{r}:\Omega_{\mathcal{R}}^{\varepsilon} \rightarrow \mathbb{R}$, $f_{m}^{\varepsilon}:\Omega_{\mathcal{M}}^{\varepsilon} \rightarrow \mathbb{R}$ are given functions , $D_{j}= \begin{bmatrix}
	d_{1}^{j} & 0 \\
	0 & d_{2}^{j}
	\end{bmatrix}$, 
	$B_{j}= \begin{bmatrix}
	b_{1}^{j}  \\
	b_{2}^{j}
	\end{bmatrix}$, 
	with $d_{1}^{j} ,d_{2}^{j},b_{1}^{j},b_{2}^{j}>0$ for $j\in \{L,R\}$, $D_{M}^{\varepsilon}(x_{1},x_{2})=D(x_{1}/\varepsilon,x_{2}/\varepsilon)$, $ (x_{1},x_{2})\in \mathbb{R}^{2}$, $D$ is a $2\times 2$ diagonal matrix with positive entries defined in standard cell $Z$ and 1-periodic, $B_{M}^{\varepsilon}=B(\frac{x_{1}}{\varepsilon},\frac{x_{2}}{\varepsilon})$, $B$ is a $2\times 1$ matrix with 1-periodic. $P_{\delta}(\cdot):\mathbb{R}\rightarrow \mathbb{R}$\footnote{	Notice that $P_{\delta}(r) \rightarrow {P}(r)$ in $L^{p}(\mathbb{R})$ as $\delta \rightarrow 0$ for all $p\in [1,\infty)$ (see p.717,  \cite{evans2010partial}).} is defined as 
	\begin{equation}\label{pd}
	P_{\delta}(r):=\int_{\mathbb{R}}\rho_{\delta}(y)P(r-y)dy,
	\end{equation}
	where $\rho_{\delta}(x):=\frac{1}{\delta^{n_{0}}}\rho(x/\delta)$,for $x\in \mathbb{R}$, $n_{0}\in \mathbb{N}$, $\rho$ is a mollifier, for instance we take
	\begin{equation}\label{pd1}
	\rho(x):=\begin{cases}
	Ce^{\left(\frac{1}{|x|^{2}-1}\right)} & |x|<1\\
	0 & |x|\geq 1,
	\end{cases}
	\end{equation}
     where the constant $C>0$ selected so that $\int_{\mathbb{R}^{n}}\rho dx =1$
	and
	\begin{equation}\label{pd2}
	P(r)=\left\{
	\begin{array}{ll}
	a_{0}+a_{1}r+\cdots+a_{m}r^{m} &\mbox{for } r\in [0,1] \\
	0& \mbox{ otherwise,}\\
	\end{array} 
	\right. 
	\end{equation}
	with $a_{k}\in \mathbb{R} \mbox{ for }k \in \mathbb{N}$.
\par We endow \eqref{rd1} with the following boundary and initial conditions
	\begin{equation}\label{bc1}
	\begin{aligned}
	u_{l}^{\varepsilon}&=U_{L}   \mbox{       on      } \Gamma_{\mathcal{L}}\times (0,T),\\
	u_{r}^{\varepsilon}&=U_{R}   \mbox{       on      } \Gamma_{\mathcal{R}}\times (0,T),
	\end{aligned}
	\end{equation}
	\begin{equation}
	\begin{aligned}
	(-\varepsilon^{\beta}D_{M}^{\varepsilon}\nabla u_{m}^{\varepsilon}+ \varepsilon^{\gamma}B_{M}^{\varepsilon}P_{\delta}(u_{m}^{\varepsilon}))\cdot n_{m}^{\varepsilon}&= \varepsilon ^{\xi} g_{0} ^{\varepsilon}\mbox{       on      } \Gamma_{0}^{\varepsilon}\times (0,T),\\
	(-D_{L}\nabla u_{l}^{\varepsilon}+ B_{L}P_{\delta}(u_{l}^{\varepsilon}))\cdot n_{l}&=  g_{l} \mbox{       on      } \left(\Gamma_{h}\cap \partial\Omega_{\mathcal{L}}^{\varepsilon}\right)\times [0,T],\\
	(-D_{R}\nabla u_{r}^{\varepsilon}+ B_{R}P_{\delta}(u_{l}^{\varepsilon}))\cdot n_{r}&=  g_{r} \mbox{       on      } \left(\Gamma_{h}\cap \partial\Omega_{\mathcal{R}}^{\varepsilon}\right)\times (0,T),
	\end{aligned}
	\end{equation}
	\begin{equation}
	\begin{aligned}
	u_{l}^{\varepsilon}(0,x)&=h_{l}^{\varepsilon}(x)  \mbox{      for all       } x\in \overline{\Omega}_{\mathcal{L}}^{\varepsilon},\\
	u_{r}^{\varepsilon}(0,x)&=h_{r}^{\varepsilon}(x)  \mbox{      for all       } x\in \overline{\Omega}_{\mathcal{R}}^{\varepsilon},\\
	u_{m}^{\varepsilon}(0,x)&=h_{m}^{\varepsilon}(x)  \mbox{      for all       } x\in \overline{\Omega}_{\mathcal{M}}^{\varepsilon},\\
	\end{aligned}
	\end{equation} 
	where $\xi>0$ is fixed real number, $U_{L}:\Gamma_{\mathcal{L}}\times [0,T]\rightarrow \mathbb{R}, U_{R}:\Gamma_{\mathcal{R}}\times [0,T]\rightarrow \mathbb{R}$, $ g_{l}:\left(\Gamma_{h}\cap \partial\Omega_{\mathcal{L}}^{\varepsilon}\right)\times [0,T]\rightarrow \mathbb{R}$, $ g_{r}: \left(\Gamma_{h}\cap \partial\Omega_{\mathcal{R}}^{\varepsilon}\right)\times [0,T]\rightarrow \mathbb{R}$, $ g_{0}^{\varepsilon}:\Gamma_{0}^{\varepsilon}\times [0,T]\rightarrow \mathbb{R}$, $h_{l}^{\varepsilon}:\overline{\Omega_{\mathcal{L}}^{\varepsilon}}\rightarrow\mathbb{R}$, $h_{r}^{\varepsilon}:\overline{\Omega_{\mathcal{R}}^{\varepsilon}}\rightarrow\mathbb{R}$, $h_{m}^{\varepsilon}:\overline{\Omega_{\mathcal{M}}^{\varepsilon}}\rightarrow\mathbb{R}$  are given functions so that $U_{L}\in L^{\infty}(0,T;H^{\frac{1}{2}}(\Gamma_{L})), U_{R} \in L^{\infty}(0,T;H^{\frac{1}{2}}(\Gamma_{R}))$. On surface $\mathcal{B_{L}^{\varepsilon}}$ and $\mathcal{B_{R}^{\varepsilon}}$ we assume the perfect transmission condition
	
		\begin{equation}\label{tc1}
	\begin{aligned}
	u_{l}^{\varepsilon}&=u_{M}^{\varepsilon}   \mbox{       on      }\mathcal{B}_{\mathcal{L}}^{\varepsilon},\\
	u_{r}^{\varepsilon}&=u_{M}^{\varepsilon}     \mbox{       on      } \mathcal{B}_{\mathcal{R}}^{\varepsilon},\\
	(-\varepsilon^{\beta}D_{M}^{\varepsilon}\nabla u_{m}^{\varepsilon}+ \varepsilon^{\gamma}B_{M}^{\varepsilon}P_{\delta}(u_{m}^{\varepsilon}))\cdot n_{m}^{\varepsilon}&=(-D_{L}\nabla u_{l}^{\varepsilon}+ B_{L}P_{\delta}(u_{l}^{\varepsilon}))\cdot n_{l}\mbox{       on      } \mathcal{B_{L}^{\varepsilon}},\\
		(-\varepsilon^{\beta}D_{M}^{\varepsilon}\nabla u_{m}^{\varepsilon}+ \varepsilon^{\gamma}B_{M}^{\varepsilon}P_{\delta}(u_{m}^{\varepsilon}))\cdot n_{m}^{\varepsilon}&=(-D_{R}\nabla u_{r}^{\varepsilon}+ B_{R}P_{\delta}(u_{l}^{\varepsilon}))\cdot n_{r}\mbox{       on      }\mathcal{B_{R}^{\varepsilon}}.\\
	\end{aligned}
	\end{equation} 

\subsection{Transformation of problem}\label{top}
To obtain the homogenous Dirichlet boundary condition for the problem \eqref{rd1}, we use the following transformation 
\begin{equation}\label{tr}
v_{i}^{\varepsilon}:=u_{i}^{\varepsilon}+\frac{1}{2} (x_{1}-1)U_{L}-\frac{1}{2} (x_{1}+1)U_{R}, 
\end{equation} 
for $i\in \{l,m,r\}$.
Let us denote $u_{b}(x,t):=\frac{1}{2} (x_{1}-1)U_{L}-\frac{1}{2} (x_{1}+1)U_{R},$ with $x=(x_{1},x_{2} ) \in \mathbb{R}^{2}$ and $t\in [0,T]$. So,
\begin{equation}\label{t}
u_{i}^{\varepsilon}=v_{i}^{\varepsilon}-u_{b}.
\end{equation}
Inserting \eqref{t} into \eqref{rd1} gives the transformed reaction diffusion equation
	\begin{equation}\label{rd2}
\begin{aligned}
\frac{\partial v_{l}^{\varepsilon}}{\partial t} +\mathrm{div}(-D_{L}\nabla v_{l}^{\varepsilon}+ B_{L}P_{\delta}(v_{l}^{\varepsilon}-u_{b}))&=f_{b_{l}}  &\mbox{on}  \ \Omega_{\mathcal{L}}^{\varepsilon} \times (0,T),\\
\frac{\partial v_{r}^{\varepsilon}}{\partial t} +\mathrm{div}(-D_{R}\nabla v_{r}^{\varepsilon}+ B_{R}P_{\delta}(v_{r}^{\varepsilon}-u_{b}))&=f_{b_{r}}&\mbox{on}  \ \Omega_{\mathcal{R}}^{\varepsilon} \times (0,T),\\
\varepsilon^{\alpha}\frac{\partial v_{m}^{\varepsilon}}{\partial t} +\mathrm{div}(-\varepsilon^{\beta}D_{M}^{\varepsilon}\nabla v_{m}^{\varepsilon}+ \varepsilon^{\gamma}B_{M}^{\varepsilon}P_{\delta}(v_{m}^{\varepsilon}-u_{b}))&=\varepsilon^{\alpha}f_{a_{m}}^{\varepsilon}&\\
&\hspace{1cm}+\varepsilon^{\beta}f_{b_{m}} &\mbox{on}  \ \Omega_{\mathcal{M}}^{\varepsilon} \times (0,T),
\end{aligned}
\end{equation}
endowed with the following boundary and initial conditions,

	\begin{equation}\label{bc2}
\begin{aligned}
v_{l}^{\varepsilon}&=0   \mbox{       on      } \Gamma_{\mathcal{L}}\times (0,T),\\
v_{r}^{\varepsilon}&=0  \mbox{       on      } \Gamma_{\mathcal{R}}\times(0,T),\\
(-\varepsilon^{\beta}D_{M}^{\varepsilon}\nabla v_{m}^{\varepsilon}+ \varepsilon^{\gamma}B_{M}^{\varepsilon}P_{\delta}(v_{m}^{\varepsilon}-u_{b}))\cdot n_{m}^{\varepsilon}&= \varepsilon ^{\xi} g_{0} ^{\varepsilon}+\varepsilon^{\beta}g_{b_{0}}\mbox{       on      } \Gamma_{0}^{\varepsilon}\times (0,T),\\
(-D_{L}\nabla v_{l}^{\varepsilon}+ B_{L}P_{\delta}(v_{l}^{\varepsilon}-u_{b}))\cdot n_{l}&=  g_{b{_{l}}} \mbox{       on      } \left(\Gamma_{h}^{\varepsilon}\cap \partial\Omega_{\mathcal{L}}^{\varepsilon}\right)\times (0,T),\\
(-D_{R}\nabla v_{r}^{\varepsilon}+ B_{R}P_{\delta}(v_{l}^{\varepsilon}-u_{b}))\cdot n_{r}&=  g_{b_{{r}}} \mbox{       on      } \left(\Gamma_{h}^{\varepsilon}\cap \partial\Omega_{\mathcal{R}}^{\varepsilon}\right)\times (0,T),\\
v_{l}^{\varepsilon}(0,x)&=h_{b_{{l}}}^{\varepsilon}(x)  \mbox{      for all       } x\in \overline{\Omega_{\mathcal{L}}^{\varepsilon}},\\
v_{r}^{\varepsilon}(0,x)&=h_{b_{r}}^{\varepsilon}(x)  \mbox{      for all       } x\in \overline{\Omega_{\mathcal{R}}^{\varepsilon}},\\
v_{m}^{\varepsilon}(0,x)&=h_{b_{m}}^{\varepsilon}(x)  \mbox{      for all       } x\in \overline{\Omega_{\mathcal{M}}^{\varepsilon}},\\
\end{aligned}
\end{equation} 
 where $f_{b_{l}}:=\partial_{t} u_{b}-div(D_{L}\nabla u_{b})+f_{l}$, $f_{b_{r}}:=\partial_{t} u_{b}-div(D_{R}\nabla u_{b})+f_{r}$, $f_{a_{m}}^{\varepsilon}:=\partial_{t} u_{b}+f_{{m}}^{\varepsilon}$, $f_{b_{m}}^{\varepsilon}:=-\mathrm{div}(D_{M}\nabla u_{b})$, $ g_{b_{0}}^{\varepsilon} :=-D_{M}^{\varepsilon}\nabla u_{b}$, $g_{b_{l}} =D_{L}\nabla u_{b}$, $g_{b_{r}} =D_{R}\nabla u_{b}$ and $h_{b_{i}}^{\varepsilon} =h_{i}^{\varepsilon}-u_{b}$ for every $i\in \{l,m,r\}$ and with the following transmission condition
 \
 	\begin{equation}\label{tc2}
 \begin{aligned}
 v_{l}^{\varepsilon}&=v_{m}^{\varepsilon}   \mbox{       on      }\mathcal{B}_{\mathcal{L}}^{\varepsilon},\\
 v_{r}^{\varepsilon}&=v_{m}^{\varepsilon}     \mbox{       on      } \mathcal{B}_{\mathcal{R}}^{\varepsilon},\\
 (-\varepsilon^{\beta}D_{M}^{\varepsilon}\nabla (v_{m}^{\varepsilon}-u_{b})+ \varepsilon^{\gamma}B_{M}^{\varepsilon}P_{\delta}(v_{m}^{\varepsilon}-u_{b}))\cdot n_{m}^{\varepsilon}&\\
 =(-D_{L}\nabla (v_{l}^{\varepsilon}-u_{b}&)+ B_{L}P_{\delta}(v_{l}^{\varepsilon}-u_{b}))\cdot n_{l}\mbox{       on      } \mathcal{B_{L}^{\varepsilon}},\\
 (-\varepsilon^{\beta}D_{M}^{\varepsilon}\nabla (v_{m}^{\varepsilon}-u_{b})+ \varepsilon^{\gamma}B_{M}^{\varepsilon}P_{\delta}(v_{m}^{\varepsilon}-{u_{b}}))\cdot n_{m}^{\varepsilon}&\\
 =(-D_{R}\nabla (v_{r}^{\varepsilon}-u_{b}&)+ B_{R}P_{\delta}(v_{l}^{\varepsilon}-u_{b}))\cdot n_{r}\mbox{       on      } \mathcal{B_{R}^{\varepsilon}}.\\
 \end{aligned}
 \end{equation} 
 \subsection{Assumptions on data}\label{aop}
 From now on $C$ denotes a positive real number possibly changing from line to line. When necessary, we will write explicit the parameters on which it will depend.

 \begin{enumerate}[label=({A}{{\arabic*}})]
 	\item\label{assump1}
 (Ellipticity condition) For every $i\in \{L,M,R\}$ with $j(L)=\mathcal{L}, j(R)=\mathcal{R} , j(M)=\mathcal{M}$ and for every $\eta\in \mathbb{R}^{2}$ there exist a $\theta>0$ such that,
 
 \begin{equation}\label{el}
 \theta \|\eta \|^{2} \leq
 \eta^{t}D_{i}\eta,
 \end{equation}
 and \begin{equation}\label{dinf}
 D_{i}\in L^{\infty}(\Omega_{j(i)}^{\varepsilon};\mathbb{R}^{4}).
 \end{equation}
 	\item \label{assump2}
 	 Concerning the drift coefficient, we assume  
 	 \begin{equation}\label{}
 	 \mathrm{div}B_{i} \in L^{\infty}(\Omega_{j(i)}^{\varepsilon}),
 	 \end{equation} and 
 \begin{equation}\label{a22}
 B_{i}\in L^{\infty}(\Omega_{j(i)}^{\varepsilon};\mathbb{R}^{2})
 \end{equation} for $i\in \{L,M,R\}$ with $j(L)=\mathcal{L}, j(R)=\mathcal{R} , j(M)=\mathcal{M}$.
 	 \item \label{assump3}
 For the reaction rate, we assume 
 	 $f_{b_{l}},\partial_{t}f_{b_{l}}\in L^{2}(0,T;L^{2}(\Omega_{\mathcal{L}}^{\varepsilon}))$,\\  $f_{b_{r}},\partial_{t}f_{b_{r}}\in L^{2}(0,T;L^{2}(\Omega_{\mathcal{R}}^{\varepsilon}))$,  $f_{b_{m}}^{\varepsilon},\partial_{t}f_{b_{m}}^{\varepsilon}\in L^{2}(0,T;L^{2}(\Omega_{\mathcal{M}}^{\varepsilon}))$ and \begin{equation}\label{}
 	 \varepsilon^{\alpha}\|f_{a_{m}}^{\varepsilon}\|_{L^{2}(0,T;L^{2}(\Omega_{\mathcal{M}}^{\varepsilon})}\leq C,
 	 \end{equation}
 	 for $a.e.$ $t\in (0,T).$ Together we assume there exist $f_{a_{0}}\in L^{2}((0,T)\times \Sigma\times Z)$ such that
 	 	\begin{equation}\label{tsf}
 	 	f_{a_{m}}^{\varepsilon}\overset{2-s}{\rightharpoonup}f_{a_{0}}.
 	 	\end{equation}
 	 \item  \label{assump4}
 	 $g_{b_{l}},\partial_{t}g_{b_{l}} \in L^{\infty}(0,T;L^{2}(\Gamma_{h}^{\varepsilon}\cap \partial\Omega_{\mathcal{L}}^{\varepsilon}))$,  $g_{b_{r}},\partial_{t}g_{b_{r}} \in L^{\infty}(0,T;L^{2}(\Gamma_{h}^{\varepsilon}\cap \partial\Omega_{\mathcal{R}}^{\varepsilon}))$, \\ $g_{b_{0}}^{\varepsilon},\partial_{t}g_{b_{0}}^{\varepsilon} \in L^{\infty}(0,T;L^{2}(\Gamma_{0}^{\varepsilon}))$, $g_{{0}}^{\varepsilon},\partial_{t}g_{{0}}^{\varepsilon} \in L^{\infty}(0,T;L^{2}(\Gamma_{0}^{\varepsilon}))$ and 
 	 
 	 \begin{equation}\label{}
 	 \varepsilon^{\xi-\frac{1}{2}}\|g_{0}^{\varepsilon}\|_{L^{2}(\Gamma_{0}^{\varepsilon})}^{2}\leq C,
 	 \end{equation}
 	  \begin{equation}\label{}
 	 \varepsilon^{\beta-\frac{1}{2}}\|g_{b_{0}}^{\varepsilon}\|_{L^{2}(\Gamma_{0}^{\varepsilon})}^{2}\leq C,
  \end{equation}
 	 
  for $a.e.$ $t\in (0,T)$. Together we assume there exist $g_{0}\in L^{2}((0,T)\times \Sigma \times \partial Y_{0})$ such that
 	 \begin{equation}\label{tsg}
 	 g_{0}^{\varepsilon}\overset{2-s}{\rightharpoonup}g_{0}.
 	 \end{equation}
 	 \item \label{assump5}
 	 For initial conditions, we assume  $h_{b_{l}}^{\varepsilon}\in H^{1}(\Omega_{\mathcal{L}}^{\varepsilon})$, $h_{b_{r}}^{\varepsilon}\in H^{1}(\Omega_{\mathcal{R}}^{\varepsilon})$, $h_{b_{m}}^{\varepsilon}\in H^{1}(\Omega_{\mathcal{M}}^{\varepsilon})$  with
 	 \begin{equation}\label{}
 	 \|h_{b_{l}}^{\varepsilon}\|_{L^{2}(\Omega_{\mathcal{L}}^{\varepsilon})}^{2}+\|h_{b_{r}}^{\varepsilon}\|_{L^{2}(\Omega_{\mathcal{R}}^{\varepsilon})}^{2}+\varepsilon^{\alpha}\|h_{b_{m}}^{\varepsilon}\|_{L^{2}(\Omega_{\mathcal{M}}^{\varepsilon})}^{2}\leq C,
 	 \end{equation} \\
 	and 
 		\begin{align}
 	\mathbbm{1}_{\Omega_{\mathcal{L}}^{\varepsilon}}h_{b_{l}}^{\varepsilon}&\rightarrow h_{b_{l}}^{0}\hspace{2cm}\mbox{on}\hspace{1cm}L^{2}((0,T)\times\Omega_{\mathcal{L}}),\label{A51}\\
 	\mathbbm{1}_{\Omega_{\mathcal{R}}^{\varepsilon}}h_{b_{r}}^{\varepsilon}&\rightarrow h_{b_{r}}^{0}\hspace{2cm}\mbox{on}\hspace{1cm}L^{2}((0,T)\times\Omega_{\mathcal{R}}),\label{A52}\\
 	h_{b_{m}}^{\varepsilon}&\overset{2-s}{\rightharpoonup} h_{b_{r}}^{0}\label{A53}
 	\end{align}
 	 \item \label{assump6}
 	  On parameter $\alpha, \beta, \gamma, \xi$ we assume, $\beta, \gamma \geq 0,$ $ \gamma\geq \beta$, $\beta\leq \xi +\frac{1}{2},$ $ \alpha+\frac{1}{2} \leq \beta,$ $ \alpha+\frac{1}{2} \leq\xi.$
 	  \item \label{assump7}
 	  $\partial_{t}u_{b}\in L^{2}(0,T;H^{1}(\Omega^{\varepsilon}))$.
 \end{enumerate}
 Note that we choose $U_{L}$ and $U_{R}$ defined in \eqref{bc1} according to satisfy \ref{assump1}-\ref{assump7}.
 \subsection{Weak formulation}\label{weakf}
 In this section we  propose the weak formulation of problem \eqref{rd2}. We use the following definitions 
 \begin{align}
 H^{1}(\Omega_{\mathcal{L}}^{\varepsilon};\Gamma_{\mathcal{L}}):&=\{u\in H^{1}(\Omega_{\mathcal{L}}^{\varepsilon}):u=0 \mbox{  on  } \Gamma_{\mathcal{L}}\},\\
H^{1}(\Omega_{\mathcal{R}}^{\varepsilon};\Gamma_{\mathcal{R}}):&=\{u\in H^{1}(\Omega_{\mathcal{R}}^{\varepsilon}):u=0 \mbox{  on  } \Gamma_{\mathcal{R}}\},\\
V_{\varepsilon}:&=\{(u_{l}^{\varepsilon},u_{m}^{\varepsilon},u_{r}^{\varepsilon})\in H(\Omega_{\mathcal{L}}^{\varepsilon};\Gamma_{\mathcal{L}})\times H^{1}(\Omega_{\mathcal{M}}^{\varepsilon})\times H(\Omega_{\mathcal{R}}^{\varepsilon};\Gamma_{\mathcal{R}}):\\
 &\hspace{2cm}u_{l}^{\varepsilon}=u_{m}^{\varepsilon} \mbox{  on  } \mathcal{B_{L}^{\varepsilon}},u_{r}^{\varepsilon}=u_{m}^{\varepsilon}\mbox{  on  } \mathcal{B_{R}^{\varepsilon}} \},\\
v^{\varepsilon}&:=(v_{l}^{\varepsilon},v_{m}^{\varepsilon},v_{r}^{\varepsilon}).
 \end{align}
 
 \begin{definition}\label{def1}

 The weak formulation of the problem \eqref{rd2}-\eqref{tc2} is to find $$(v_{l}^{\varepsilon},v_{m}^{\varepsilon},v_{r}^{\varepsilon})\in L^{2}(0,T;V_{\varepsilon})\cap H^{1}(0,T;L^{2}(\Omega_{\mathcal{L}}^{\varepsilon})\times L^{2}(\Omega_{\mathcal{M}}^{\varepsilon})\times L^{2}(\Omega_{\mathcal{R}}^{\varepsilon}) )$$

 such that $(v_{l}^{\varepsilon},v_{m}^{\varepsilon},v_{r}^{\varepsilon}) $ satisfies
 
 	\begin{multline}\label{var1}
 \int_{\Omega_{\mathcal{L}}^{\varepsilon}} \partial_{t} v_{l}^{\varepsilon}\phi_{1}dx
 +\int_{\Omega_{\mathcal{L}}^{\varepsilon}}D_{L}\nabla v_{l} ^{\varepsilon}\nabla\phi_{1} dx
 -\int_{\Omega_{\mathcal{L}}^{\varepsilon}}B_{L}P_{\delta}(v_{l}^{\varepsilon}-u_{b})\nabla \phi_{1} dx\\
 = \int_{\Omega_{\mathcal{L}}^{\varepsilon}}f_{b_{l}}\phi_{1} dx 
 -\int_{\Gamma_{h}\cap \partial\Omega_{\mathcal{L}}^{\varepsilon}}g_{b_{l}} \phi_{1} d\sigma+\int_{\mathcal{B_{L}^{\varepsilon}}}D_{L}\nabla u_{b}\cdot n_{l}\phi_{1}d\sigma \\\
 +\int_{\mathcal{B_{L}^{\varepsilon}}} (-\varepsilon^{\beta}D_{M}^{\varepsilon}\nabla (v_{m}^{\varepsilon}-u_{b})+ \varepsilon^{\gamma}B_{M}^{\varepsilon}P_{\delta}(v_{m}^{\varepsilon}-u_{b}))\cdot n_{m}^{\varepsilon}\phi_{1}d\sigma,
 \end{multline}

  	\begin{multline}\label{var2}
  \int_{\Omega_{\mathcal{R}}^{\varepsilon}} \partial_{t} v_{r}^{\varepsilon}\phi_{3}dx
 +\int_{\Omega_{\mathcal{R}}^{\varepsilon}}D_{R}\nabla v_{r} ^{\varepsilon}\nabla\phi_{3} dx
 -\int_{\Omega_{\mathcal{R}}^{\varepsilon}}B_{R}P_{\delta}(v_{r}^{\varepsilon}-u_{b})\nabla \phi_{3} dx\\
 = \int_{\Omega_{\mathcal{R}}^{\varepsilon}}f_{b_{r}}\phi_{1} dx 
 -\int_{\Gamma_{h}\cap \partial\Omega_{\mathcal{R}}^{\varepsilon}}g_{b_{r}} \phi_{3} d\sigma+\int_{\mathcal{B_{R}^{\varepsilon}}}D_{R}\nabla u_{b}\cdot n_{r}\phi_{3}d\sigma \\\
 +\int_{\mathcal{B_{R}^{\varepsilon}}} (-\varepsilon^{\beta}D_{M}^{\varepsilon}\nabla (v_{m}^{\varepsilon}-u_{b})+ \varepsilon^{\gamma}B_{M}^{\varepsilon}P_{\delta}(v_{m}^{\varepsilon}-u_{b}))\cdot n_{m}^{\varepsilon}\phi_{1}d\sigma,
 \end{multline}
 
   	\begin{multline}\label{var3}
\varepsilon^{\alpha}  \int_{\Omega_{\mathcal{M}}^{\varepsilon}} \partial_{t} v_{m}^{\varepsilon}\phi_{2}dx
+ \varepsilon^{\beta}\int_{\Omega_{\mathcal{M}}^{\varepsilon}}D_{M}^{\varepsilon}\nabla v_{m} ^{\varepsilon}\nabla\phi_{2} dx
 -\varepsilon^{\gamma}\int_{\Omega_{\mathcal{M}}^{\varepsilon}}B_{M}^{\varepsilon}P_{\delta}(v_{m}^{\varepsilon}-u_{b})\nabla \phi_{2} dx\\
 = \varepsilon^{\alpha}\int_{\Omega_{\mathcal{M}}^{\varepsilon}}f_{a_{m}}^{\varepsilon}\phi_{2} dx +\varepsilon^{\beta}\int_{\Omega_{\mathcal{M}}^{\varepsilon}}f_{b_{m}}^{\varepsilon}\phi_{2} dx
 -\varepsilon^{\xi}\int_{\Gamma_{0}^{\varepsilon}}g_{0}^{\varepsilon} \phi_{2} d\sigma-\varepsilon^{\beta}\int_{\Gamma_{0}^{\varepsilon}}g_{b_{0}}^{\varepsilon} \phi_{2} d\sigma\\
 +\varepsilon^{\beta}\int_{\mathcal{B_{L}^{\varepsilon}}}D_{M}\nabla u_{b}\cdot n_{m}^{\varepsilon}\phi_{2}d\sigma 
 -\int_{\mathcal{B_{L}^{\varepsilon}}} (-\varepsilon^{\beta}D_{M}^{\varepsilon}\nabla (v_{m}^{\varepsilon}-u_{b})+ \varepsilon^{\gamma}B_{M}^{\varepsilon}P_{\delta}(v_{m}^{\varepsilon}-u_{b}))\cdot n_{m}^{\varepsilon}\phi_{2}d\sigma\\
 +\varepsilon^{\beta}\int_{\mathcal{B_{R}^{\varepsilon}}}D_{M}\nabla u_{b}\cdot n_{m}^{\varepsilon}\phi_{2}d\sigma 
 -\int_{\mathcal{B_{R}^{\varepsilon}}} (-\varepsilon^{\beta}D_{M}^{\varepsilon}\nabla (v_{m}^{\varepsilon}-u_{b})+ \varepsilon^{\gamma}B_{M}^{\varepsilon}P_{\delta}(v_{m}^{\varepsilon}-u_{b}))\cdot n_{m}^{\varepsilon}\phi_{2}d\sigma,
 \end{multline}
 for every $\phi:=(\phi_{1},\phi_{2},\phi_{3})\in V_{\varepsilon}$ and $a.e.$  $t\in (0,T)$.
\end{definition}
 \par   Now using, $n_{l}=-n_{m}^{\varepsilon}$ on $\mathcal{B_{L}^{\varepsilon}}$, $n_{r}=-n_{m}^{\varepsilon}$ on $\mathcal{B_{R}^{\varepsilon}}$, $\phi_{1}=\phi_{2}$ on $\mathcal{B_{L}^{\varepsilon}}$, $\phi_{3}=\phi_{2}$ on $\mathcal{B_{L}^{\varepsilon}}$ and adding the equations \eqref{var1}-\eqref{var3}, we  find it is useful to rewrite Definition \ref{def1} as to find 
$$(v_{l}^{\varepsilon},v_{m}^{\varepsilon},v_{r}^{\varepsilon})\in  L^{2}(0,T;V_{\varepsilon})\cap H^{1}(0,T;L^{2}(\Omega_{\mathcal{L}}^{\varepsilon})\times L^{2}(\Omega_{\mathcal{M}}^{\varepsilon})\times L^{2}(\Omega_{\mathcal{R}}^{\varepsilon}) )$$
  such that
  
 	\begin{multline}\label{wf}
 	  \int_{\Omega_{\mathcal{L}}^{\varepsilon}} \partial_{t} v_{l}^{\varepsilon}\phi_{1}dx+
 	  \int_{\Omega_{\mathcal{R}}^{\varepsilon}} \partial_{t} v_{r}^{\varepsilon}\phi_{3}dx+
 	 \varepsilon^{\alpha}   \int_{\Omega_{\mathcal{M}}^{\varepsilon}} \partial_{t} v_{m}^{\varepsilon}\phi_{2}dx
\\
+ \int_{\Omega_{\mathcal{L}}^{\varepsilon}}D_{L}\nabla v_{l} ^{\varepsilon}\nabla\phi_{1} dx+ \int_{\Omega_{\mathcal{R}}^{\varepsilon}}D_{R}\nabla v_{r} ^{\varepsilon}\nabla\phi_{3} dx+ \varepsilon^{\beta}\int_{\Omega_{\mathcal{M}}^{\varepsilon}}D_{M}^{\varepsilon}\nabla v_{m} ^{\varepsilon}\nabla\phi_{2} dx\\-\int_{\Omega_{\mathcal{L}}^{\varepsilon}}B_{L}P_{\delta}(v_{l}^{\varepsilon}-u_{b})\nabla \phi_{1} dx
-\int_{\Omega_{\mathcal{R}}^{\varepsilon}}B_{R}P_{\delta}(v_{r}^{\varepsilon}-u_{b})\nabla \phi_{3} dx\\
-\varepsilon^{\gamma}\int_{\Omega_{\mathcal{M}}^{\varepsilon}}B_{M}^{\varepsilon}P_{\delta}(v_{m}^{\varepsilon}-u_{b})\nabla \phi_{2} dx\\
 =\int_{\Omega_{\mathcal{L}}^{\varepsilon}}f_{b_{l}}\phi_{1} dx  +\int_{\Omega_{\mathcal{R}}^{\varepsilon}}f_{b_{r}}\phi_{3} dx +\varepsilon^{\alpha}\int_{\Omega_{\mathcal{M}}^{\varepsilon}}f_{a_{m}}^{\varepsilon}\phi_{2} dx +\varepsilon^{\beta}\int_{\Omega_{\mathcal{M}}^{\varepsilon}}f_{b_{m}}^{\varepsilon}\phi_{2} dx\\
 -\int_{\Gamma_{h}^{\varepsilon}\cap \partial\Omega_{\mathcal{L}}^{\varepsilon}}g_{b_{l}} \phi_{1} d\sigma-\int_{\Gamma_{h}^{\varepsilon}\cap \partial\Omega_{\mathcal{R}}^{\varepsilon}}g_{b_{r}} \phi_{3} d\sigma-\varepsilon^{\xi}\int_{\Gamma_{0}^{\varepsilon}}g_{0}^{\varepsilon} \phi_{2} d\sigma-\varepsilon^{\beta}\int_{\Gamma_{0}^{\varepsilon}}g_{b_{0}}^{\varepsilon} \phi_{2} d\sigma\\
 +\int_{\mathcal{B_{L}^{\varepsilon}}}(D_{L}-\varepsilon^{\beta}D_{M})\nabla u_{b}\cdot n_{l}\phi_{1}d\sigma+\int_{\mathcal{B_{R}^{\varepsilon}}}(D_{R}-\varepsilon^{\beta}D_{M})\nabla u_{b}\cdot n_{r}\phi_{3}d\sigma 
 \end{multline}
 
 for every $\phi\in V_{\varepsilon}$ and $a.e.$  $t\in (0,T)$ with the initial condition \eqref{bc2}, namely
 \begin{equation}\label{ic}
 \begin{aligned}
 v_{l}^{\varepsilon}(0,x)&=h_{b_{{l}}}^{\varepsilon}(x)  \mbox{      for all       } x\in {\overline{\Omega}_{\mathcal{L}}^{\varepsilon}},\\
 v_{r}^{\varepsilon}(0,x)&=h_{b_{r}}^{\varepsilon}(x)  \mbox{      for all       } x\in {\overline{\Omega}_{\mathcal{R}}^{\varepsilon}},\\
 v_{m}^{\varepsilon}(0,x)&=h_{b_{m}}^{\varepsilon}(x)  \mbox{      for all       } x\in {\overline{\Omega}_{\mathcal{M}}^{\varepsilon}}.
 \end{aligned}
 \end{equation}
We denote the $\varepsilon$ dependent problem \eqref{wf} together with \eqref{ic} as $\left({P_{\varepsilon}}\right)$ problem.

	\section{Weak solvability of the microscopic problem}\label{wsmp}
	The following theorem establishes  existence and uniqueness of the weak solution of the problem \eqref{rd1}. 
	\begin{theorem}\label{T1}
		Under the assumption \ref{assump1}-\ref{assump6} there exists a unique solution of \eqref{wf}.

	\end{theorem}
	\begin{proof}
	The proof follows via Galerkin method similar to that used in \cite{muntean2010multiscale}.
	\end{proof}
	\section{Two-scale convergence for thin membrane}\label{tsctm}	
	In this section, we first prove $\varepsilon$ independent \emph{a priori} energy estimates for the solution $(v_{l}^{\varepsilon},v_{m}^{\varepsilon},v_{r}^{\varepsilon})$ of \eqref{wf}. Then we will use these estimates to get the  two-scale limit of $(v_{l}^{\varepsilon},v_{m}^{\varepsilon},v_{r}^{\varepsilon})$  as $\varepsilon \rightarrow 0$.
	\subsection{\emph{ A priori }estimates} \label{ape}
	\begin{lemma}\label{L1}
There exists $C>0$ independent of $\varepsilon$, such that
\begin{equation}\label{l1}
\|v_{l}^{\varepsilon}\|_{L^{2}(\Gamma_{h}^{\varepsilon}\cap \partial\Omega_{\mathcal{L}}^{\varepsilon})}^{2}\leq C\|\nabla v_{l}^{\varepsilon}\|_{L^{2}(\Omega_{\mathcal{L}}^{\varepsilon})}^{2},
\end{equation}  
\begin{equation}\label{l2}
\|v_{l}^{\varepsilon}\|_{L^{2}(\Omega_{\mathcal{L}}^{\varepsilon})}^{2}\leq C\|\nabla v_{l}^{\varepsilon}\|_{L^{2}(\Omega_{\mathcal{L}}^{\varepsilon})}^{2}.
\end{equation}
\begin{equation}\label{l3}
\|v_{r}^{\varepsilon}\|_{L^{2}(\Gamma_{h}^{\varepsilon}\cap \partial\Omega_{\mathcal{R}}^{\varepsilon})}^{2}\leq C\|\nabla v_{r}^{\varepsilon}\|_{L^{2}(\Omega_{\mathcal{R}}^{\varepsilon})}^{2},
\end{equation}  
\begin{equation}\label{l4}
\|v_{r}^{\varepsilon}\|_{L^{2}(\Omega_{\mathcal{R}}^{\varepsilon})}^{2}\leq C\|\nabla v_{r}^{\varepsilon}\|_{L^{2}(\Omega_{\mathcal{R}}^{\varepsilon})}^{2}.
\end{equation}
	\end{lemma}
\begin{proof}
	The proof of \eqref{l1} is a simple application of Theorem 3.3 from \cite{han1992best}. For each $\varepsilon>0$, we can calculate the best trace constant $C(\varepsilon)$ exactly for $\Omega_{\mathcal{L}}^{\varepsilon}$ (see similar case in Example 4.3 of \cite{han1992best}). Then we bound the trace constants $C(\varepsilon)$ by a general constant $C>0$.
	\par The proof of \eqref{l2} is  application of Proposition 2.1.1 from \cite{nazarov2015exact}.\\
\end{proof}
\begin{remark}\label{remark1}The inequalities similar to \eqref{l1}-\eqref{l4} are not valid in the case of $v_{m}^{\varepsilon}$. Since in the case of $\Omega_{\mathcal{L}}^{\varepsilon}$ or $\Omega_{\mathcal{R}}^{\varepsilon}$ we have nested set $i.e$ for every $\varepsilon_{1}\leq \varepsilon_{2}$ we have $\Omega_{\mathcal{L}}^{\varepsilon_{2}}\subset \Omega_{\mathcal{L}}^{\varepsilon_{1}}$ and $\Omega_{\mathcal{R}}^{\varepsilon_{2}}\subset \Omega_{\mathcal{R}}^{\varepsilon_{1}}$. But in the case of $v_{m}^{\varepsilon}$, for $\varepsilon_{1}\neq\varepsilon_{2}$ we can not say either $\Omega_{\mathcal{M}}^{\varepsilon_{1}}\subset \Omega_{\mathcal{M}}^{\varepsilon_{2}}$ or $\Omega_{\mathcal{M}}^{\varepsilon_{2}}\subset \Omega_{\mathcal{M}}^{\varepsilon_{1}}$.
\end{remark}
 \begin{lemma}\label{L2}
There exists $C>0$ independent of $\varepsilon$, such that
\begin{equation}\label{l5}
\|v_{l}^{\varepsilon}\|_{L^{2}(\mathcal{B_{L}^{\varepsilon}})}^{2}\leq C\|\nabla v_{l}^{\varepsilon}\|_{L^{2}(\Omega_{\mathcal{L}}^{\varepsilon})}^{2},
\end{equation}  

\begin{equation}\label{l7}
\|v_{r}^{\varepsilon}\|_{L^{2}(\mathcal{B_{R}^{\varepsilon}})}^{2}\leq C\|\nabla v_{r}^{\varepsilon}\|_{L^{2}(\Omega_{\mathcal{R}}^{\varepsilon})}^{2}.
\end{equation}  
\begin{proof}
	Proof follows same lines of proof of Lemma \ref{L1}.
\end{proof}
\end{lemma}
Next we state Lemma 3.1 from \cite{effectiveapratim}.
	\begin{lemma}\label{LN1}
For all $v_{m}^{\varepsilon} \in H^{1}(\Omega_{\mathcal{M}}^{\varepsilon})$, there exist $C>0$ satisfying the following inequality
\begin{equation}\label{}
\|v_{m}^{\varepsilon}\|_{L^{2}(\Gamma_{0}^{\varepsilon})}\leq C\left(\varepsilon^{\frac{-1}{2}}\| v_{m}^{\varepsilon}\|_{L^{2}(\Omega_{\mathcal{M}}^{\varepsilon})}+\varepsilon^{\frac{1}{2}}\|\nabla  v_{m}^{\varepsilon}\|_{L^{2}(\Omega_{\mathcal{M}}^{\varepsilon})}\right)
\end{equation}
	\end{lemma}
\begin{proof}
For proof we refer Lemma 3.1 of  \cite{effectiveapratim}.
\end{proof}
	\begin{theorem}\label{T3}
	The weak solution $v^{\varepsilon}$ to the problem \eqref{wf}  satisfies the following energy estimates with $C>0$
	
		\begin{equation}\label{e1}
		\|v_{l}^{\varepsilon}\|_{L^{2}(\Omega_{\mathcal{L}}^{\varepsilon})}^{2}+			\|v_{r}^{\varepsilon}\|_{L^{2}(\Omega_{\mathcal{R}}^{\varepsilon})}^{2}+\varepsilon^{\alpha}\|v_{m}^{\varepsilon}\|_{L^{2}(\Omega_{\mathcal{M}}^{\varepsilon})}^{2}\leq C,
		\end{equation}
		\begin{equation}\label{e2}
		\|\nabla v_{l}^{\varepsilon}\|_{L^{2}(0,T;{L^{2}(\Omega_{\mathcal{L}}^{\varepsilon})})}^{2}+	\|\nabla v_{r}^{\varepsilon}\|_{L^{2}(0,T;{L^{2}(\Omega_{\mathcal{R}}^{\varepsilon})})}^{2}+\varepsilon^{\beta}\|\nabla v_{m}^{\varepsilon}\|_{L^{2}(0,T;{L^{2}(\Omega_{\mathcal{M}}^{\varepsilon})})}^{2}\leq C,
		\end{equation}
			\begin{equation}\label{e4}
		\|P_{\delta}(v_{l}^{\varepsilon}-u_{b})\|_{L^{2}(\Omega_{\mathcal{L}}^{\varepsilon})}^{2}+			\|P_{\delta}(v_{r}^{\varepsilon}-u_{b})\|_{L^{2}(\Omega_{\mathcal{R}}^{\varepsilon})}^{2}+\varepsilon^{\gamma}\|P_{\delta}(v_{m}^{\varepsilon}-u_{b})\|_{L^{2}(\Omega_{\mathcal{M}}^{\varepsilon})}^{2}\leq C,
		\end{equation}
\begin{equation}\label{e3}
	\|\partial_{t}v_{l}^{\varepsilon}\|_{L^{2}(0,T;L^{2}(\Omega_{\mathcal{L}}^{\varepsilon}))}^{2}+\|\partial_{t}v_{r}^{\varepsilon}\|_{L^{2}(0,T;L^{2}(\Omega_{\mathcal{R}}^{\varepsilon}))}^{2}+\varepsilon^{\alpha}\|\partial_{t}v_{m}^{\varepsilon}\|_{L^{2}(0,T;L^{2}(\Omega_{\mathcal{M}}^{\varepsilon}))}^{2}
	\leq C.
		\end{equation}

	\end{theorem}
\begin{proof}
	We prove the estimate \eqref{e1} by choosing the test function $\phi =(v_{l}^{\varepsilon},v_{m}^{\varepsilon},v_{r}^{\varepsilon})$ in the weak formulation \eqref{wf}, we get
		\begin{multline}\label{wf1}
		 \int_{\Omega_{\mathcal{L}}^{\varepsilon}} \partial_{t} v_{l}^{\varepsilon} v_{l}^{\varepsilon}dx+
	 \int_{\Omega_{\mathcal{R}}^{\varepsilon}} \partial_{t} v_{r}^{\varepsilon} v_{r}^{\varepsilon}dx+
		\varepsilon^{\alpha}  \int_{\Omega_{\mathcal{M}}^{\varepsilon}} \partial_{t}  v_{m}^{\varepsilon}\phi_{2}dx+\\
	\int_{\Omega_{\mathcal{L}}^{\varepsilon}}D_{L}\nabla v_{l} ^{\varepsilon}\nabla v_{l}^{\varepsilon} dx+ \int_{\Omega_{\mathcal{R}}^{\varepsilon}}D_{R}\nabla v_{r} ^{\varepsilon}\nabla v_{r}^{\varepsilon} dx+ \varepsilon^{\beta}\int_{\Omega_{\mathcal{M}}^{\varepsilon}}D_{M}^{\varepsilon}\nabla v_{m} ^{\varepsilon}\nabla v_{m}^{\varepsilon} dx\\
	-\int_{\Omega_{\mathcal{L}}^{\varepsilon}}B_{L}P_{\delta}(v_{l}^{\varepsilon}-u_{b})\nabla v_{l}^{\varepsilon} dx
	-\int_{\Omega_{\mathcal{R}}^{\varepsilon}}B_{R}P_{\delta}(v_{r}^{\varepsilon}-u_{b})\nabla v_{r}^{\varepsilon} dx\\
	-\varepsilon^{\gamma}\int_{\Omega_{\mathcal{M}}^{\varepsilon}}B_{M}^{\varepsilon}P_{\delta}(v_{m}^{\varepsilon}-u_{b})\nabla v_{m}^{\varepsilon} dx\\
	=\int_{\Omega_{\mathcal{L}}^{\varepsilon}}f_{b_{l}}v_{l}^{\varepsilon} dx  +\int_{\Omega_{\mathcal{R}}^{\varepsilon}}f_{b_{r}}v_{r}^{\varepsilon}dx +\varepsilon^{\alpha}\int_{\Omega_{\mathcal{M}}^{\varepsilon}}f_{a_{m}}^{\varepsilon}v_{m}^{\varepsilon} dx +\varepsilon^{\beta}\int_{\Omega_{\mathcal{M}}^{\varepsilon}}f_{b_{m}}^{\varepsilon}v_{m}^{\varepsilon}dx\\
	-\int_{\Gamma_{h}^{\varepsilon}\cap \partial\Omega_{\mathcal{L}}^{\varepsilon}}g_{b_{l}}v_{l}^{\varepsilon} d\sigma-\int_{\Gamma_{h}^{\varepsilon}\cap \partial\Omega_{\mathcal{R}}^{\varepsilon}}g_{b_{r}} v_{r}^{\varepsilon} d\sigma-\varepsilon^{\xi}\int_{\Gamma_{0}^{\varepsilon}}g_{0}^{\varepsilon} v_{m}^{\varepsilon} d\sigma-\varepsilon^{\beta}\int_{\Gamma_{0}^{\varepsilon}}g_{b_{0}}^{\varepsilon} v_{m}^{\varepsilon} d\sigma\\
	+\int_{\mathcal{B_{L}^{\varepsilon}}}(D_{L}-\varepsilon^{\beta}D_{M})\nabla u_{b}\cdot n_{l}v_{l}^{\varepsilon}d\sigma+\int_{\mathcal{B_{R}^{\varepsilon}}}(D_{R}-\varepsilon^{\beta}D_{M})\nabla u_{b}\cdot n_{r}v_{r}^{\varepsilon}d\sigma .
	\end{multline}
	Using  \eqref{pd}-\eqref{pd2} and \eqref{a22}, we obtain
	\begin{align}
	 \int_{\Omega_{\mathcal{L}}^{\varepsilon}}B_{L}P_{\delta}(v_{l}^{\varepsilon}-u_{b})\nabla v_{l}^{\varepsilon}dx\leq C	\int_{\Omega_{\mathcal{L}}^{\varepsilon}}|\nabla v_{l}^{\varepsilon}| dx \label{bes1},\\
	\int_{\Omega_{\mathcal{R}}^{\varepsilon}}B_{R}P_{\delta}(v_{r}^{\varepsilon}-u_{b})\nabla v_{r}^{\varepsilon}dx\leq C	\int_{\Omega_{\mathcal{R}}^{\varepsilon}}|\nabla v_{r}^{\varepsilon} |dx\label{bes2},\\
	\int_{\Omega_{\mathcal{M}}^{\varepsilon}}B_{M}^{\varepsilon}P_{\delta}(v_{m}^{\varepsilon}-u_{b})\nabla v_{m}^{\varepsilon}dx\leq C	\int_{\Omega_{\mathcal{M}}^{\varepsilon}}|\nabla v_{m}^{\varepsilon} |dx\label{bes3},
	\end{align}
	where the constant $C$ in \eqref{bes1} depends on $\|B_{L}\|_{L^{\infty}(\Omega_{\mathcal{L}}^{\varepsilon};\mathbb{R}^{2})}$ and $\|P_{\delta}(u)\|_{L^{\infty}(\mathbb{R})}$, similarly $C$ in \eqref{bes2} depend on $\|B_{L}\|_{L^{\infty}(\Omega_{\mathcal{R}}^{\varepsilon};\mathbb{R}^{2})}$ and $\|P_{\delta}(u)\|_{L^{\infty}(\mathbb{R})}$, $C$ in \eqref{bes3} depends on $\|B_{L}\|_{L^{\infty}(\Omega_{\mathcal{M}}^{\varepsilon};\mathbb{R}^{2})}$ and $\|P_{\delta}(u)\|_{L^{\infty}(\mathbb{R})}$.\\
Now using \eqref{el} together with the estimates \eqref{bes1}-\eqref{bes3}, we obtain
\begin{multline}\label{ee0}
\frac{1}{2}\frac{d}{dt}\|v_{l}^{\varepsilon}\|_{L^{2}(\Omega_{\mathcal{L}}^{\varepsilon})}^{2}+\frac{1}{2}\frac{d}{dt}\|v_{r}^{\varepsilon}\|_{L^{2}(\Omega_{\mathcal{R}}^{\varepsilon})}^{2}+\varepsilon^{\alpha}\frac{1}{2}\frac{d}{dt}\|v_{m}^{\varepsilon}\|_{L^{2}(\Omega_{\mathcal{M}}^{\varepsilon})}^{2}\\
+\theta \int_{\Omega_{\mathcal{L}}^{\varepsilon}}|\nabla v_{l}^{\varepsilon}|^{2}dx+\theta \int_{\Omega_{\mathcal{R}}^{\varepsilon}}|\nabla v_{r}^{\varepsilon}|^{2}dx+\varepsilon^{\beta}\theta \int_{\Omega_{\mathcal{M}}^{\varepsilon}}|\nabla v_{m}^{\varepsilon}|^{2}dx\\
\leq C\int_{\Omega_{\mathcal{L}}^{\varepsilon}}|\nabla v_{l}^{\varepsilon}|dx+C\int_{\Omega_{\mathcal{R}}^{\varepsilon}}|\nabla v_{r}^{\varepsilon}|dx+C\varepsilon^{\gamma}\int_{\Omega_{\mathcal{M}}^{\varepsilon}}|\nabla v_{m}^{\varepsilon}|dx\\
+\int_{\Omega_{\mathcal{L}}^{\varepsilon}}|f_{b_{l}}v_{l}^{\varepsilon}|dx+\int_{\Omega_{\mathcal{R}}^{\varepsilon}}|f_{b_{r}}v_{r}^{\varepsilon}|dx+\varepsilon^{\alpha}\int_{\Omega_{\mathcal{M}}^{\varepsilon}}|f_{a_{m}} v_{m}^{\varepsilon}| dx+\varepsilon^{\beta}\int_{\Omega_{\mathcal{M}}^{\varepsilon}}|f_{b_{m}}^{\varepsilon} v_{m}^{\varepsilon}| dx\\
+\int_{\Gamma_{h}^{\varepsilon}\cap \partial\Omega_{\mathcal{L}}^{\varepsilon}}|g_{b_{l}}||v_{l}^{\varepsilon}| d\sigma+\int_{\Gamma_{h}^{\varepsilon}\cap \partial\Omega_{\mathcal{R}}^{\varepsilon}}|g_{b_{r}} ||v_{r}^{\varepsilon}| d\sigma+\varepsilon^{\xi}\int_{\Gamma_{0}^{\varepsilon}}|g_{0}^{\varepsilon} || v_{m}^{\varepsilon}| d\sigma+\varepsilon^{\beta}\int_{\Gamma_{0}^{\varepsilon}}|g_{b_{0}}^{\varepsilon} || v_{m}^{\varepsilon}| d\sigma\\
+ \int_{\mathcal{B_{L}^{\varepsilon}}}|(D_{L}-\varepsilon^{\beta}D_{M})\nabla u_{b}| |v_{l}^{\varepsilon}|d\sigma+\int_{\mathcal{B_{R}^{\varepsilon}}}|(D_{R}-\varepsilon^{\beta}D_{M})\nabla u_{b}||v_{r}^{\varepsilon}|d\sigma.
\end{multline}

Now, by applying Young's Inequality and using the assumption \ref{assump6}, for any $\zeta>0$ we have
\begin{equation}\label{ee1}
\int_{\Omega_{\mathcal{L}}^{\varepsilon}}|\nabla v_{l}^{\varepsilon}|dx\leq \zeta \int_{\Omega_{\mathcal{L}}^{\varepsilon}}|\nabla v_{l}^{\varepsilon}|^{2}dx+ C(\zeta),
\end{equation}
\begin{equation}\label{ee1x}
\int_{\Omega_{\mathcal{R}}^{\varepsilon}}|\nabla v_{r}^{\varepsilon}|dx\leq \zeta \int_{\Omega_{\mathcal{R}}^{\varepsilon}}|\nabla v_{r}^{\varepsilon}|^{2}dx+ C(\zeta),
\end{equation}
\begin{equation}\label{ee2}
\begin{aligned}
\varepsilon^{\gamma}\int_{\Omega_{\mathcal{M}}^
{\varepsilon}}|\nabla v_{m}^{\varepsilon}|dx&\leq \varepsilon^{\gamma}\zeta \int_{\Omega_{\mathcal{M}}^{\varepsilon}}|\nabla v_{m}^{\varepsilon}|^{2}dx+ \varepsilon^{\gamma} C(\zeta)|\Omega_{\mathcal{M}}^{\varepsilon}|^{\frac{1}{2}}\\
&\leq \varepsilon^{\gamma}\zeta \int_{\Omega_{\mathcal{M}}^{\varepsilon}}|\nabla v_{m}^{\varepsilon}|^{2}dx+\varepsilon^{\gamma+\frac{1}{2}}\sqrt{2h}C(\zeta)\\
&\leq  \varepsilon^{\gamma}\zeta \int_{\Omega_{\mathcal{M}}^{\varepsilon}}|\nabla v_{m}^{\varepsilon}|^{2}dx+C(\zeta),
\end{aligned}
\end{equation}
\begin{equation}\label{ee3}
\int_{\Omega_{\mathcal{L}}^{\varepsilon}}|f_{b_{l}}v_{l}^{\varepsilon}|dx\leq\frac{1}{2} \int_{\Omega_{\mathcal{L}}^{\varepsilon}}|f_{b_{l}}|^{2}dx+\frac{1}{2}\int_{\Omega_{\mathcal{L}}^{\varepsilon}}{|v_{l}^{\varepsilon}}|^{2}dx,
\end{equation}
\begin{equation}\label{ee3x}
\int_{\Omega_{\mathcal{R}}^{\varepsilon}}|f_{b_{r}}v_{r}^{\varepsilon}|dx\leq\frac{1}{2} \int_{\Omega_{\mathcal{R}}^{\varepsilon}}|f_{b_{r}}|^{2}dx+\frac{1}{2}\int_{\Omega_{\mathcal{R}}^{\varepsilon}}{|v_{r}^{\varepsilon}}|^{2}dx,
\end{equation}
\begin{equation}\label{nee4}
\varepsilon^{\alpha}\int_{\Omega_{\mathcal{M}}^{\varepsilon}}|f_{a_{m}}v_{m}^{\varepsilon}|dx\leq \frac{1}{2}\varepsilon^{\alpha}\int_{\Omega_{\mathcal{M}}^{\varepsilon}}|f_{a_{m}}|^{2}dx+\frac{1}{2}\varepsilon^{\alpha}\int_{\Omega_{\mathcal{M}}}{|v_{m}^{\varepsilon}}|^{2}dx,
\end{equation}
\begin{equation}\label{ee4}
\begin{aligned}
\varepsilon^{\beta}\int_{\Omega_{\mathcal{M}}^{\varepsilon}}|f_{b_{m}}^{\varepsilon}v_{m}^{\varepsilon}|dx&\leq\frac{1}{2} \varepsilon^{\beta}\int_{\Omega_{\mathcal{M}}^{\varepsilon}}|f_{b_{m}}^{\varepsilon}|^{2}dx+\frac{1}{2}\varepsilon^{\beta}\int_{\Omega_{\mathcal{M}}^{\varepsilon}}{|v_{m}^{\varepsilon}}|^{2}dx,\\
&\leq\frac{1}{2}\varepsilon^{\beta}\int_{\Omega_{\mathcal{M}}^{\varepsilon}}|f_{b_{m}}^{\varepsilon}|^{2}dx +\frac{1}{2}\varepsilon^{\alpha}\int_{\Omega_{\mathcal{M}}^{\varepsilon}}{|v_{m}^{\varepsilon}}|^{2}dx.
\end{aligned}
\end{equation}
Using Cauchy Schwarz's inequality, \eqref{l1} and Young's Inequality, we have
\begin{equation}\label{ee5}
\begin{aligned}
\int_{\Gamma_{h}^\varepsilon\cap \partial\Omega_{\mathcal{L}}^{\varepsilon}}|g_{b_{l}}||v_{l}^{\varepsilon}| d\sigma &\leq \|g_{b_{l}}\|_{L^{2}(\Gamma_{h}^{\varepsilon}\cap \partial\Omega_{\mathcal{L}}^{\varepsilon})}\|v_{l}^{\varepsilon}\|_{L^{2}(\Gamma_{h}^{\varepsilon}\cap \partial\Omega_{\mathcal{L}}^{\varepsilon})}\\
&\leq C \|g_{b_{l}}\|_{L^{2}(\Gamma_{h}^{\varepsilon}\cap \partial\Omega_{\mathcal{L}}^{\varepsilon})}\|\nabla v_{l}^{\varepsilon}\|_{L^{2}(\Omega_{\mathcal{L}}^{\varepsilon})}\\
&\leq C(\zeta) \|g_{b_{l}}\|_{L^{2}(\Gamma_{h}^{\varepsilon}\cap \partial\Omega_{\mathcal{L}}^{\varepsilon})}^{2}+C\zeta \|\nabla v_{l}^{\varepsilon}\|_{L^{2}(\Omega_{\mathcal{L}}^{\varepsilon})}^{2}.
\end{aligned}
\end{equation}
 By similar argument as for \eqref{ee5}, we get
\begin{equation}\label{}
\int_{\Gamma_{h}^{\varepsilon}\cap \partial\Omega_{\mathcal{R}}^{\varepsilon}}|g_{b_{r}}||v_{r}^{\varepsilon}| d\sigma \leq C(\zeta) \|g_{b_{r}}\|_{L^{2}(\Gamma_{h}^{\varepsilon}\cap \partial\Omega_{\mathcal{R}}^{\varepsilon})}^{2}+C\zeta \|\nabla v_{r}^{\varepsilon}\|_{L^{2}(\Omega_{\mathcal{R}}^{\varepsilon})}^{2}.
\end{equation}
Using Cauchy Schwarz's inequality, Lemma \ref{LN1}, Young's inequality and \ref{assump6}, we get
\begin{equation}\label{nee6}
\begin{aligned}
\varepsilon^{\beta}\int_{\Gamma_{0}^{\varepsilon}}|g_{b_{0}}^{\varepsilon} || v_{m}^{\varepsilon}| d\sigma&\leq \varepsilon^{\beta} \|g_{b_{0}}^{\varepsilon}\|_{L^{2}(\Gamma_{0}^{\varepsilon})}\|v_{m}^{\varepsilon}\|_{L^{2}(\Gamma_{0}^{\varepsilon})}\\
&\leq C\varepsilon^{\beta-\frac{1}{2}}\|g_{b_{0}}^{\varepsilon}\|_{L^{2}(\Gamma_{0}^{\varepsilon})}\| v_{m}^{\varepsilon}\|_{L^{2}(\Omega_{\mathcal{M}}^{\varepsilon})}\\
&\hspace{1cm}+C\varepsilon^{\beta+\frac{1}{2}}\|g_{b_{0}}^{\varepsilon}\|_{L^{2}(\Gamma_{0}^{\varepsilon})}\|\nabla v_{m}^{\varepsilon}\|_{L^{2}(\Omega_{\mathcal{M}}^{\varepsilon})}\\
&\leq C\varepsilon^{\beta-\frac{1}{2}}\|g_{b_{0}}^{\varepsilon}\|_{L^{2}(\Gamma_{0}^{\varepsilon})}^{2}\\
&\hspace{2cm}+ \varepsilon^{\beta-\frac{1}{2}}C\| v_{m}^{\varepsilon}\|_{L^{2}(\Omega_{\mathcal{M}}^{\varepsilon})}^{2}+C(\zeta)\varepsilon^{\beta+\frac{1}{2}}\|g_{b_{0}}^{\varepsilon}\|_{L^{2}(\Gamma_{0}^{\varepsilon})}^{2}\\
&\hspace{3cm}+\zeta\varepsilon^{\beta+\frac{1}{2}} \|\nabla v_{m}^{\varepsilon}\|_{L^{2}(\Omega_{\mathcal{M}}^{\varepsilon})}^{2} \\
& \leq C\varepsilon^{\beta-\frac{1}{2}}\|g_{b_{0}}^{\varepsilon}\|_{L^{2}(\Gamma_{0}^{{\varepsilon}})}^{2}+ \varepsilon^{\alpha}C\| v_{m}^{\varepsilon}\|_{L^{2}(\Omega_{\mathcal{M}}^{\varepsilon})}^{2}\\
&\hspace{1cm}+C(\zeta)\varepsilon^{\beta+\frac{1}{2}}\|g_{b_{0}}^{\varepsilon}\|_{L^{2}(\Gamma_{0}^{\varepsilon})}^{2}+\zeta\varepsilon^{\beta} \|\nabla v_{m}^{\varepsilon}\|_{L^{2}(\Omega_{\mathcal{M}}^{\varepsilon})}^{2},
\end{aligned}
\end{equation}

similarly,
\begin{equation}\label{ee6}
\begin{aligned}
\varepsilon^{\xi}\int_{\Gamma_{0}^{\varepsilon}}|g_{0}^{\varepsilon} || v_{m}^{\varepsilon}| d\sigma&\leq \varepsilon^{\xi-\frac{1}{2}}C\|g_{0}^{\varepsilon}\|_{L^{2}(\Gamma_{0}^{\varepsilon})}^{2}+\varepsilon^{\alpha}C\| v_{m}^{\varepsilon}\|_{L^{2}(\Omega_{\mathcal{M}}^{\varepsilon})}^{2}\\
&\hspace{2cm}+C(\zeta)\varepsilon^{\xi+\frac{1}{2}}\|g_{0}^{\varepsilon}\|_{L^{2}(\Gamma_{0}^{\varepsilon})}^{2}+\zeta\varepsilon^{\beta} \|\nabla v_{m}^{\varepsilon}\|_{L^{2}(\Omega_{\mathcal{M}}^{\varepsilon})}^{2}
\end{aligned}
\end{equation}
Using \eqref{dinf}, \ref{assump6}, Cauchy Schwarz's inequality, the trace inequality and Young's Inequality, we have
\begin{equation}\label{ee6n}
\begin{aligned}
 \int_{\mathcal{B_{L}^{\varepsilon}}}|(D_{L}-\varepsilon^{\beta}D_{M})\nabla u_{b}| |v_{l}^{\varepsilon}|d\sigma&\leq C \int_{\mathcal{B_{L}^{\varepsilon}}}|\nabla u_{b}| |v_{l}^{\varepsilon}|d\sigma\\
 &\leq C\|\nabla u_{b}\|_{L^{2}(\mathcal{B_{L}})}\|v_{l}^{\varepsilon}\|_{L^{2}(\mathcal{B_{L}})}\\
 &\leq C\|\nabla u_{b}\|_{L^{2}(\mathcal{B_{L}})}\|\nabla v_{l}^{\varepsilon}\|_{L^{2}(\Omega_{\mathcal{L}}^{\varepsilon})}\\
 &\leq C(\zeta)+\zeta\|\nabla v_{l}^{\varepsilon}\|_{L^{2}({\Omega_{\mathcal{L}}^{\varepsilon}})}^{2},
\end{aligned}
\end{equation}

similarly,
\begin{equation}\label{ee6x}
\int_{\mathcal{B_{R}^{\varepsilon}}}|(D_{R}-\varepsilon^{\beta}D_{M})\nabla u_{b}||v_{r}^{\varepsilon}|d\sigma\leq C(\zeta)+\zeta\|\nabla v_{r}^{\varepsilon}\|_{L^{2}({\Omega_{\mathcal{R}}^{\varepsilon}})}^{2}.
\end{equation}

Now, we use assumption \ref{assump6}, \eqref{ee1}-\eqref{ee6x} in \eqref{ee0}, and for $\zeta$ small enough, we control the gradient terms appearing on the right-hand side using the analogous nonnegative terms on the left-hand side, thus we get
\begin{multline}\label{ee8}
\frac{1}{2}\frac{d}{dt}\|v_{l}^{\varepsilon}\|_{L^{2}(\Omega_{\mathcal{L}}^{\varepsilon})}^{2}+\frac{1}{2}\frac{d}{dt}\|v_{r}^{\varepsilon}\|_{L^{2}(\Omega_{\mathcal{R}}^{\varepsilon})}^{2}+\varepsilon^{\alpha}\frac{1}{2}\frac{d}{dt}\|v_{m}^{\varepsilon}\|_{L^{2}(\Omega_{\mathcal{M}}^{\varepsilon})}^{2}\\
\leq C\left(1+\|f_{b_{l}}\|_{L^{2}(\Omega_{\mathcal{L}}^{\varepsilon})}^{2}+\|f_{b_{r}}\|_{L^{2}(\Omega_{\mathcal{R}}^{\varepsilon})}^{2}+\varepsilon ^{\alpha}\|f_{a_{m}}^{\varepsilon}\|_{L^{2}(\Omega_{\mathcal{M}}^{\varepsilon})}^{2}+\varepsilon ^{\beta}\|f_{b_{m}}^{\varepsilon}\|_{L^{2}(\Omega_{\mathcal{M}}^{\varepsilon})}^{2}\right.\\
+ \|g_{b_{l}}\|_{L^{2}(\Gamma_{h}^{\varepsilon}\cap \partial\Omega_{\mathcal{L}}^{\varepsilon})}^{2}+ \|g_{b_{r}}\|_{L^{2}(\Gamma_{h}^{\varepsilon}\cap \partial\Omega_{\mathcal{R}}^{\varepsilon})}^{2}+\varepsilon^{\beta-\frac{1}{2}}\|g_{b_{0}}^{\varepsilon}\|_{L^{2}(\Gamma_{0}^{\varepsilon})}^{2}+\varepsilon^{\xi-\frac{1}{2}}C\|g_{0}^{\varepsilon}\|_{L^{2}(\Gamma_{0}^{\varepsilon})}^{2}\\
\left. +\|v_{l}^{\varepsilon}\|_{L^{2}({\Omega_{\mathcal{L}}^{\varepsilon}})}^{2}+\|v_{r}^{\varepsilon}\|_{L^{2}({\Omega_{\mathcal{R}}^{\varepsilon}})}^{2}+\varepsilon^{\alpha} \| v_{m}^{\varepsilon}\|_{L^{2}(\Omega_{\mathcal{M}}^{\varepsilon})}^{2}\right).
\end{multline}
By using  \ref{assump3}, \ref{assump4} in \eqref{ee8} and using Gronwall's inequality with \ref{assump5}, we obtain
\begin{equation}\label{}
	\|v_{l}^{\varepsilon}\|_{L^{2}(\Omega_{\mathcal{L}}^{\varepsilon})}^{2}+			\|v_{r}^{\varepsilon}\|_{L^{2}(\Omega_{\mathcal{R}}^{\varepsilon})}^{2}+\varepsilon^{\alpha}\|v_{m}^{\varepsilon}\|_{L^{2}(\Omega_{\mathcal{M}}^{\varepsilon})}^{2}\leq C.
\end{equation}
Hence we have finished the proof of the estimate \eqref{e1}.
\par From \eqref{ee0}-\eqref{ee6x} and \ref{assump6}, we have

\begin{multline}\label{ee9}
\frac{1}{2}\frac{d}{dt}\|v_{l}^{\varepsilon}\|_{L^{2}(\Omega_{\mathcal{L}}^{\varepsilon})}^{2}+\frac{1}{2}\frac{d}{dt}\|v_{r}^{\varepsilon}\|_{L^{2}(\Omega_{\mathcal{R}}^{\varepsilon})}^{2}+\varepsilon^{\alpha}\frac{1}{2}\frac{d}{dt}\|v_{m}^{\varepsilon}\|_{L^{2}(\Omega_{\mathcal{M}}^{\varepsilon})}^{2}\\
+\theta \int_{\Omega_{\mathcal{L}}^{\varepsilon}}|\nabla v_{l}^{\varepsilon}|^{2}dx+\theta \int_{\Omega_{\mathcal{R}}^{\varepsilon}}|\nabla v_{r}^{\varepsilon}|^{2}dx+\varepsilon^{\beta}\theta \int_{\Omega_{\mathcal{M}}^{\varepsilon}}|\nabla v_{m}|^{2}dx\\
\leq C(C(\zeta)+\|f_{b_{l}}\|_{L^{2}(\Omega_{\mathcal{L}}^{\varepsilon})}^{2}+\|f_{b_{r}}\|_{L^{2}(\Omega_{\mathcal{R}}^{\varepsilon})}^{2}+\varepsilon ^{\alpha}\|f_{a_{m}}^{\varepsilon}\|_{L^{2}(\Omega_{\mathcal{M}}^{\varepsilon})}^{2}+\varepsilon ^{\beta}\|f_{b_{m}}^{\varepsilon}\|_{L^{2}(\Omega_{\mathcal{M}}^{\varepsilon})}^{2}\\
+C(\zeta) \|g_{b_{l}}\|_{L^{2}(\Gamma_{h}^{\varepsilon}\cap \partial\Omega_{\mathcal{L}})}^{2}+ C(\zeta) \|g_{b_{r}}\|_{L^{2}(\Gamma_{h}^{\varepsilon}\cap \partial\Omega_{\mathcal{R}})}^{2}+\varepsilon^{\beta-\frac{1}{2}}\|g_{b_{0}}^{\varepsilon}\|_{L^{2}(\Gamma_{0}^{\varepsilon})}^{2}\\
+C(\zeta)\varepsilon^{\beta+\frac{1}{2}}\|g_{b_{0}}\|_{L^{2}(\Gamma_{0}^{\varepsilon})}^{2}+\varepsilon^{\xi-\frac{1}{2}}\|g_{0}^{\varepsilon}\|_{L^{2}(\Gamma_{0}^{\varepsilon})}^{2}+C(\zeta)\varepsilon^{\xi+\frac{1}{2}}\|g_{0}^{\varepsilon}\|_{L^{2}(\Gamma_{0}^{\varepsilon})}^{2}\\
+	\|v_{l}^{\varepsilon}\|_{L^{2}(\Omega_{\mathcal{L}}^{\varepsilon})}^{2}+			\|v_{r}^{\varepsilon}\|_{L^{2}(\Omega_{\mathcal{R}}^{\varepsilon})}^{2}+\varepsilon^{\alpha}\|v_{m}^{\varepsilon}\|_{L^{2}(\Omega_{\mathcal{M}}^{\varepsilon})}^{2}\\
+\zeta \int_{\Omega_{\mathcal{L}}^{\varepsilon}}|\nabla v_{l}^{\varepsilon}|^{2}dx+\zeta \int_{\Omega_{\mathcal{R}}^{\varepsilon}}|\nabla v_{r}^{\varepsilon}|^{2}dx+\varepsilon^{\beta}\zeta \int_{\Omega_{\mathcal{M}}^{\varepsilon}}|\nabla v_{m}|^{2}dx).
\end{multline} 
Integrating \eqref{ee9}  from 0 to $T$ and using \eqref{e1} together with \ref{assump3}-\ref{assump6} and for $\zeta$ small enough, we obtain
\begin{equation}\label{}
	\|\nabla v_{l}^{\varepsilon}\|_{L^{2}(0,T;{L^{2}(\Omega_{\mathcal{L}}^{\varepsilon})})}^{2}+	\|\nabla v_{r}^{\varepsilon}\|_{L^{2}(0,T;{L^{2}(\Omega_{\mathcal{R}}^{\varepsilon})})}^{2}+\varepsilon^{\beta}\|\nabla v_{m}^{\varepsilon}\|_{L^{2}(0,T;{L^{2}(\Omega_{\mathcal{M}}^{\varepsilon})})}^{2}\leq C.
\end{equation}
Hence we proved estimate \eqref{e2}.
\par By using the structure of $P_{\delta}$ we know that there exists $k\in \mathbb{R}$ such that $P_{\delta}(k)=0$. Now using Mean Value Theorem and \eqref{e1} we have
\begin{equation}\label{t2l1}
\begin{aligned}
\|P_{\delta}(v_{l}^{\varepsilon}-u_{b})\|_{L^{2}(\Omega_{\mathcal{L}}^{\varepsilon})}^{2}&=	\|P_{\delta}(v_{l}^{\varepsilon}-u_{b})-P_{\delta}(k)\|_{L^{2}(\Omega_{\mathcal{L}}^{\varepsilon})}^{2}\\
&\leq C	\|v_{l}^{\varepsilon}-u_{b}-k\|_{L^{2}(\Omega_{\mathcal{L}}^{\varepsilon})}^{2}\\
&\leq C \left(\|v_{l}^{\varepsilon}\|_{L^{2}(\Omega_{\mathcal{L}}^{\varepsilon})}^{2}+\|u_{b}-k\|_{L^{2}(\Omega_{\mathcal{L}}^{\varepsilon})}^{2}\right).\\
&\leq C.
\end{aligned}
\end{equation}
Similarly we can prove
\begin{equation}\label{t2l2}
\|P_{\delta}(v_{r}^{\varepsilon}-u_{b})\|_{L^{2}(\Omega_{\mathcal{R}}^{\varepsilon})}^{2}\leq C.
\end{equation}
We have 

\begin{equation}\label{t2l3}
\begin{aligned}
\varepsilon^{\gamma}\|P_{\delta}(v_{m}^{\varepsilon}-u_{b})\|_{L^{2}(\Omega_{\mathcal{M}}^{\varepsilon})}^{2}&=\varepsilon^{\gamma}	\|P_{\delta}(v_{m}^{\varepsilon}-u_{b})-P_{\delta}(k)\|_{L^{2}(\Omega_{\mathcal{M}}^{\varepsilon})}^{2}\\
&\leq \varepsilon ^{\gamma}C	\|v_{m}^{\varepsilon}-u_{b}-k\|_{L^{2}(\Omega_{\mathcal{M}}^{\varepsilon})}^{2}\\
&\leq \varepsilon^{\gamma}C \left(\|v_{m}^{\varepsilon}\|_{L^{2}(\Omega_{\mathcal{M}}^{\varepsilon})}^{2}+\|u_{b}-k\|_{L^{2}(\Omega_{\mathcal{M}}^{\varepsilon})}^{2}\right).\\
&\leq C\varepsilon^{\alpha}\|v_{m}^{\varepsilon}\|_{L^{2}(\Omega_{\mathcal{M}}^{\varepsilon})}^{2}+C\varepsilon^{\gamma}|\Omega_{\mathcal{M}}^{\varepsilon}|\\
&\leq C+C\varepsilon^{\gamma}2\varepsilon h\\
&\leq C+C\varepsilon^{\gamma+1}\\
&\leq C,
\end{aligned}
\end{equation}
where we used  $\varepsilon^{\gamma}\leq \varepsilon^{\alpha}$ and  \eqref{e1}.
 Now combining \eqref{t2l1}, \eqref{t2l2} and \eqref{t2l3} we obtain \eqref{e4}.\\ \begin{remark}
 
  from \eqref{t2l3} we can see that we can relax the assumption on $\gamma$ stated in \ref{assump6} by $\gamma\in [-1,\infty)$ for the estimate \eqref{e4} 
\end{remark}

\par To prove the estimate \eqref{e3}, we approximate the weak solution of \eqref{wf} using Galerkin scheme and take $(\partial_{t}v_{l}^{\varepsilon},\partial_{t}v_{m}^{\varepsilon},\partial_{t}v_{r}^{\varepsilon})$ as test function.
\par Let $\{w_{l,k}^{\varepsilon}\}_{k=1}^{\infty}$  be an orthogonal basis of $H^{1}(\Omega_{\mathcal{L}}^{\varepsilon};\Gamma_{\mathcal{L}})$ and orthonormal basis of $L^{2}(\Omega_{\mathcal{L}}^{\varepsilon})$. Similarly, we define $\{w_{r,k}^{\varepsilon}\}_{k=1}^{\infty}$ and $\{w_{m,k}^{\varepsilon}\}_{k=1}^{\infty}$ as basis of $H^{1}(\Omega_{\mathcal{R}}^{\varepsilon};\Gamma_{\mathcal{R}})$ and $H^{1}(\Omega_{\mathcal{M}}^{\varepsilon})$ respectively. Then the Galerkin scheme problem will be to find $d_{l,1}^{\varepsilon},d_{l,2}^{\varepsilon},\cdots,d_{l,n}^{\varepsilon},d_{m,1}^{\varepsilon},d_{m,2}^{\varepsilon}$\\$,\cdots,d_{m,n}^{\varepsilon},d_{r,1}^{\varepsilon},d_{r,2}^{\varepsilon},\cdots,d_{r,n}^{\varepsilon}:[0,1]\rightarrow \mathbb{R}$ such that,
\begin{equation}\label{ge1}
\begin{aligned}
v_{l,n}^{\varepsilon}&=\sum_{k=0}^{n}d_{l,k}^{\varepsilon}w_{l,k}^{\varepsilon}\\
v_{m,n}^{\varepsilon}&=\sum_{k=0}^{n}d_{m,k}^{\varepsilon}w_{m,k}^{\varepsilon}\\
v_{r,n}^{\varepsilon}&=\sum_{k=0}^{n}d_{r,k}^{\varepsilon}w_{r,k}^{\varepsilon}
\end{aligned}
\end{equation}
satisfying the weak formulation corresponding to \eqref{wf} with the initial condition \eqref{ic}. The existence and uniqueness of such $d_{l,1}^{\varepsilon},d_{l,2}^{\varepsilon},\cdots,d_{l,n}^{\varepsilon}$ $,d_{m,1}^{\varepsilon},d_{m,2}^{\varepsilon},\cdots,$ $d_{m,n}^{\varepsilon},d_{r,1}^{\varepsilon},d_{r,2}^{\varepsilon}$ $,\cdots,d_{r,n}^{\varepsilon}$ follows by direct application of Picard-Lindelöf Theorm (see \cite{coddington1955theory}). Moreover $d_{l,1}^{\varepsilon},d_{l,2}^{\varepsilon},\cdots,d_{l,n}^{\varepsilon}$$,d_{m,1}^{\varepsilon},d_{m,2}^{\varepsilon},\cdots,d_{m,n}^{\varepsilon},\\ d_{r,1}^{\varepsilon},d_{r,2}^{\varepsilon},\cdots,d_{r,n}^{\varepsilon}$ lie in $C^{1}(0,T)\cap C[0,T]$ as a consequence of Picard-Lindelöf Theorem . So, we have 

	\begin{multline}\label{ge2}
 \int_{\Omega_{\mathcal{L}}^{\varepsilon}} \partial_{t} v_{l,n}^{\varepsilon}w_{l,k}^{\varepsilon}dx+
 \int_{\Omega_{\mathcal{R}}^{\varepsilon}} \partial_{t} v_{r,n}^{\varepsilon}w_{r,k}^{\varepsilon}dx+
\varepsilon^{\alpha}  \int_{\Omega_{\mathcal{M}}^{\varepsilon}} \partial_{t} v_{m,n}^{\varepsilon}w_{m,k}^{\varepsilon}dx
+\\
\int_{\Omega_{\mathcal{L}}^{\varepsilon}}D_{L}\nabla v_{l,n} ^{\varepsilon}\nabla w_{l,k}^{\varepsilon} dx+ \int_{\Omega_{\mathcal{R}}^{\varepsilon}}D_{R}\nabla v_{r,n} ^{\varepsilon}\nabla w_{r,k}^{\varepsilon} dx+ \varepsilon^{\beta}\int_{\Omega_{\mathcal{M}}^{\varepsilon}}D_{M}^{\varepsilon}\nabla v_{m,n} ^{\varepsilon}\nabla w_{m,k}^{\varepsilon} dx\\
-\int_{\Omega_{\mathcal{L}}^{\varepsilon}}B_{L}P_{\delta}(v_{l,n}^{\varepsilon}-u_{b})\nabla w_{l,k}^{\varepsilon} dx
-\int_{\Omega_{\mathcal{R}}^{\varepsilon}}B_{R}P_{\delta}(v_{r,n}^{\varepsilon}-u_{b})\nabla w_{r,k}^{\varepsilon} dx\\
-\varepsilon^{\gamma}\int_{\Omega_{\mathcal{M}}^{\varepsilon}}B_{M}^{\varepsilon}P_{\delta}(v_{m,n}^{\varepsilon}-u_{b})\nabla w_{m,k}^{\varepsilon} dx\\
=\int_{\Omega_{\mathcal{L}}^{\varepsilon}}f_{b_{l}}w_{l,k}^{\varepsilon} dx  +\int_{\Omega_{\mathcal{R}}^{\varepsilon}}f_{b_{r}}w_{r,k}^{\varepsilon}dx +\varepsilon^{\alpha}\int_{\Omega_{\mathcal{M}}^{\varepsilon}}f_{a_{m}}^{\varepsilon}w_{m,k}^{\varepsilon} dx +\varepsilon^{\beta}\int_{\Omega_{\mathcal{M}}^{\varepsilon}}f_{b_{m}}^{\varepsilon}w_{m,k}^{\varepsilon}dx\\
-\int_{\Gamma_{h}^{\varepsilon}\cap \partial\Omega_{\mathcal{L}}^{\varepsilon}}g_{b_{l}} w_{l,k}^{\varepsilon} d\sigma-\int_{\Gamma_{h}^{\varepsilon}\cap \partial\Omega_{\mathcal{R}}^{\varepsilon}}g_{b_{r}} w_{r,k}^{\varepsilon} d\sigma-\varepsilon^{\xi}\int_{\Gamma_{0}^{\varepsilon}}g_{0}^{\varepsilon} w_{m,k}^{\varepsilon} d\sigma-\varepsilon^{\beta}\int_{\Gamma_{0}^{\varepsilon}}g_{b_{0}}^{\varepsilon} w_{m,k}^{\varepsilon}d\sigma\\
+\int_{\mathcal{B_{L}^{\varepsilon}}}(D_{L}-\varepsilon^{\beta}D_{M})\nabla u_{b}\cdot n_{l}w_{l,k}^{\varepsilon}d\sigma+\int_{\mathcal{B_{R}^{\varepsilon}}}(D_{R}-\varepsilon^{\beta}D_{M})\nabla u_{b}\cdot n_{l}w_{r,k}^{\varepsilon}d\sigma ,
\end{multline}
for all $k\in \{1,2,\cdots,n\}$.\\
Now, multiplying \eqref{ge2} with $\mathbbm{1}_{\Omega_{\mathcal{L}}^{\varepsilon}}\partial_{t}d_{l,k}^{\varepsilon}, \mathbbm{1}_{\Omega_{\mathcal{R}}^{\varepsilon}}\partial_{t}d_{r,k}^{\varepsilon},  \mathbbm{1}_{\Omega_{\mathcal{M}}^{\varepsilon}}\partial_{t}d_{m,k}^{\varepsilon}$ for $k\in \{1,2,\cdots,n\}$ and summing over $k=1,2,\cdots,n$, we get

\begin{multline}\label{ge3}
 \int_{\Omega_{\mathcal{L}}^{\varepsilon}} \partial_{t} v_{l,n}^{\varepsilon} \partial_{t}v_{l,n}^{\varepsilon}dx+
\int_{\Omega_{\mathcal{R}}^{\varepsilon}} \partial_{t} v_{r,n}^{\varepsilon} \partial_{t}v_{r,n}^{\varepsilon}dx+
\varepsilon^{\alpha}  \int_{\Omega_{\mathcal{M}}^{\varepsilon}} \partial_{t}v_{m,n}^{\varepsilon} \partial_{t}v_{m,n}^{\varepsilon}dx+\\
\int_{\Omega_{\mathcal{L}}^{\varepsilon}}D_{L}\nabla v_{l,n} ^{\varepsilon}\nabla \partial_{t}v_{l,n}^{\varepsilon} dx+ \int_{\Omega_{\mathcal{R}}^{\varepsilon}}D_{R}\nabla v_{r,n} ^{\varepsilon}\nabla \partial_{t}v_{r,n}^{\varepsilon} dx+ \varepsilon^{\beta}\int_{\Omega_{\mathcal{M}}^{\varepsilon}}D_{M}^{\varepsilon}\nabla v_{m,n} ^{\varepsilon}\nabla \partial_{t}v_{m,n}^{\varepsilon} dx\\
-\int_{\Omega_{\mathcal{L}}^{\varepsilon}}B_{L}P_{\delta}(v_{l,n}^{\varepsilon}-u_{b})\nabla \partial_{t}v_{l,n}^{\varepsilon} dx
-\int_{\Omega_{\mathcal{R}}^{\varepsilon}}B_{R}P_{\delta}(v_{r,n}^{\varepsilon}-u_{b})\nabla \partial_{t}v_{r,n}^{\varepsilon} dx\\
-\varepsilon^{\gamma}\int_{\Omega_{\mathcal{M}}^{\varepsilon}}B_{M}^{\varepsilon}P_{\delta}(v_{m,n}^{\varepsilon}-u_{b})\nabla\partial_{t} v_{m,n}^{\varepsilon} dx\\
=\int_{\Omega_{\mathcal{L}}^{\varepsilon}}f_{b_{l}}\partial_{t}v_{l,n}^{\varepsilon} dx  +\int_{\Omega_{\mathcal{R}}^{\varepsilon}}f_{b_{r}}\partial_{t}v_{r,n}^{\varepsilon}dx +\varepsilon^{\alpha}\int_{\Omega_{\mathcal{M}}^{\varepsilon}}f_{a_{m}}^{\varepsilon}\partial_{t}v_{m,n}^{\varepsilon} dx \\
+\varepsilon^{\beta}\int_{\Omega_{\mathcal{M}}^{\varepsilon}}f_{b_{m}}^{\varepsilon}\partial_{t}v_{m,n}^{\varepsilon}dx\\
-\int_{\Gamma_{h}^{\varepsilon}\cap \partial\Omega_{\mathcal{L}}^{\varepsilon}}g_{b_{l}}\partial_{t}v_{l,n}^{\varepsilon} d\sigma-\int_{\Gamma_{h}^{\varepsilon}\cap \partial\Omega_{\mathcal{R}}^{\varepsilon}}g_{b_{r}} \partial_{t}v_{r,n}^{\varepsilon} d\sigma-\varepsilon^{\xi}\int_{\Gamma_{0}^{\varepsilon}}g_{0}^{\varepsilon} \partial_{t}v_{m,n}^{\varepsilon} d\sigma\\
-\varepsilon^{\beta}\int_{\Gamma_{0}^{\varepsilon}}g_{b_{0}}^{\varepsilon} \partial_{t}v_{m,n}^{\varepsilon} d\sigma\\
+\int_{\mathcal{B_{L}^{\varepsilon}}}(D_{L}-\varepsilon^{\beta}D_{M})\nabla u_{b}\cdot n_{l}\partial_{t}v_{l,n}^{\varepsilon}d\sigma+\int_{\mathcal{B_{R}^{\varepsilon}}}(D_{R}-\varepsilon^{\beta}D_{M})\nabla u_{b}\cdot n_{r}\partial_{t}v_{r,n}^{\varepsilon}d\sigma .
\end{multline}
To derive the desired estimates, we use the following identities,
\begin{equation}\label{ge4}
\frac{1}{2}\frac{d}{dt}\int_{\Omega_{\mathcal{L}}^{\varepsilon}}D_{L}\nabla v_{l,n} ^{\varepsilon}\nabla v_{l,n}^{\varepsilon} dx=\int_{\Omega_{\mathcal{L}}^{\varepsilon}}D_{L}\nabla v_{l,n} ^{\varepsilon}\nabla \partial_{t}v_{l,n}^{\varepsilon} dx,
\end{equation}
\begin{equation}\label{ge5}
\frac{1}{2}\frac{d}{dt}\int_{\Omega_{\mathcal{R}}^{\varepsilon}}D_{R}\nabla v_{r,n} ^{\varepsilon}\nabla v_{r,n}^{\varepsilon} dx=\int_{\Omega_{\mathcal{R}}^{\varepsilon}}D_{R}\nabla v_{r,n} ^{\varepsilon}\nabla \partial_{t}v_{r,n}^{\varepsilon} dx,
\end{equation}
\begin{equation}\label{ge6}
\frac{\varepsilon^{\beta}}{2}\frac{d}{dt}\int_{\Omega_{\mathcal{M}}^{\varepsilon}}D_{M}^{\varepsilon}\nabla v_{m,n} ^{\varepsilon}\nabla v_{m,n}^{\varepsilon} dx\\
=\varepsilon^{\beta}\int_{\Omega_{\mathcal{M}}^{\varepsilon}}D_{M}^{\varepsilon}\nabla v_{m,n} ^{\varepsilon}\nabla \partial_{t}v_{m}^{\varepsilon} dx,
\end{equation}
\begin{multline}\label{ge7}
\int_{\Omega_{\mathcal{L}}^{\varepsilon}}B_{L}P_{\delta}(v_{l,n}^{\varepsilon}-u_{b})\nabla \partial_{t}v_{l,n}^{\varepsilon} dx=\frac{d}{dt}\int_{\Omega_{\mathcal{L}}^{\varepsilon}}B_{L}P_{\delta}(v_{l,n}^{\varepsilon}-u_{b})\nabla v_{l,n}^{\varepsilon} dx\\-\int_{\Omega_{\mathcal{L}}^{\varepsilon}}B_{L}P_{\delta}^{\prime}(v_{l,n}^{\varepsilon}-u_{b})\partial_{t}(v_{l,n}^{\varepsilon}-u_{b})\nabla v_{l,n}^{\varepsilon} dx,
\end{multline}

\begin{multline}\label{ge8}
\int_{\Omega_{\mathcal{R}}^{\varepsilon}}B_{R}P_{\delta}(v_{r,n}^{\varepsilon}-u_{b})\nabla \partial_{t}v_{r,n}^{\varepsilon} dx=\frac{d}{dt}\int_{\Omega_{\mathcal{R}}^{\varepsilon}}B_{R}P_{\delta}(v_{r,n}^{\varepsilon}-u_{b})\nabla v_{r,n}^{\varepsilon} dx\\-\int_{\Omega_{\mathcal{R}}^{\varepsilon}}B_{R}P_{\delta}^{\prime}(v_{r,n}^{\varepsilon}-u_{b})\partial_{t}(v_{r,n}^{\varepsilon}-u_{b})\nabla v_{r,n}^{\varepsilon} dx,
\end{multline}

\begin{multline}\label{}
\varepsilon^{\gamma}\int_{\Omega_{\mathcal{M}}^{\varepsilon}}B_{M}P_{\delta}(v_{m,n}^{\varepsilon}-u_{b})\nabla \partial_{t}v_{m,n}^{\varepsilon} dx=\varepsilon^{\gamma}\frac{d}{dt}\int_{\Omega_{\mathcal{M}}^{\varepsilon}}B_{M}P_{\delta}(v_{m,n}^{\varepsilon}-u_{b})\nabla v_{m,n}^{\varepsilon} dx\\-\varepsilon^{\gamma}\int_{\Omega_{\mathcal{M}}^{\varepsilon}}B_{M}P_{\delta}^{\prime}(v_{m,n}^{\varepsilon}-u_{b})\partial_{t}(v_{m,n}^{\varepsilon}-u_{b})\nabla v_{m,n}^{\varepsilon} dx,
\end{multline}
\begin{equation}
\int_{\Gamma_{h}^{\varepsilon}\cap \partial\Omega_{\mathcal{L}}^{\varepsilon}}g_{b_{l}}\partial_{t}v_{l,n}^{\varepsilon} d\sigma=\frac{d}{dt}\int_{\Gamma_{h}^{\varepsilon}\cap \partial\Omega_{\mathcal{L}}^{\varepsilon}}g_{b_{l}}v_{l,n}^{\varepsilon} d\sigma-\int_{\Gamma_{h}^{\varepsilon}\cap \partial\Omega_{\mathcal{L}}^{\varepsilon}}\partial_{t}g_{b_{l}}v_{l,n}^{\varepsilon} d\sigma
\end{equation}
\begin{equation}\label{}
\int_{\Gamma_{h}^{\varepsilon}\cap \partial\Omega_{\mathcal{R}}^{\varepsilon}}g_{b_{r}}\partial_{t}v_{r,n}^{\varepsilon} d\sigma=\frac{d}{dt}\int_{\Gamma_{h}^{\varepsilon}\cap \partial\Omega_{\mathcal{R}}^{\varepsilon}}g_{b_{r}}v_{r,n}^{\varepsilon} d\sigma-\int_{\Gamma_{h}^{\varepsilon}\cap \partial\Omega_{\mathcal{R}}^{\varepsilon}}\partial_{t}g_{b_{r}}v_{r,n}^{\varepsilon} d\sigma
\end{equation}
\begin{equation}\label{}
\varepsilon^{\xi}\int_{\Gamma_{0}^{\varepsilon}}g_{0}^{\varepsilon} \partial_{t}v_{m,n}^{\varepsilon} d\sigma=\varepsilon^{\xi}\frac{d}{dt}\int_{\Gamma_{0}^{\varepsilon}}g_{0}^{\varepsilon} v_{m,n}^{\varepsilon} d\sigma-\varepsilon^{\xi}\int_{\Gamma_{0}^{\varepsilon}}\partial_{t}g_{0}^{\varepsilon} v_{m,n}^{\varepsilon} d\sigma
\end{equation}
\begin{equation}\label{}
\varepsilon^{\beta}\int_{\Gamma_{0}^{\varepsilon}}g_{b_{0}}^{\varepsilon} \partial_{t}v_{m,n}^{\varepsilon} d\sigma=\frac{d}{dt}\varepsilon^{\beta}\int_{\Gamma_{0}^{\varepsilon}}g_{b_{0}}^{\varepsilon} v_{m,n}^{\varepsilon} d\sigma-\varepsilon^{\beta}\int_{\Gamma_{0}^{\varepsilon}}\partial_{t}g_{b_{0}}^{\varepsilon} v_{m,n}^{\varepsilon} d\sigma.
\end{equation}

We obtain

\begin{multline}\label{ge9}
\|\partial_{t}v_{l,n}^{\varepsilon}\|_{L^{2}(\Omega_{\mathcal{L}}^{\varepsilon})}^{2}+\|\partial_{t}v_{r,n}^{\varepsilon}\|_{L^{2}(\Omega_{\mathcal{R}}^{\varepsilon})}^{2}+\varepsilon^{\alpha}\|\partial_{t}v_{m,n}^{\varepsilon}\|_{L^{2}(\Omega_{\mathcal{M}}^{\varepsilon})}^{2}+\\
\frac{1}{2}\frac{d}{dt}\int_{\Omega_{\mathcal{L}}^{\varepsilon}}D_{L}\nabla v_{l,n} ^{\varepsilon}\nabla v_{l,n}^{\varepsilon} dx+\frac{1}{2}\frac{d}{dt}\int_{\Omega_{\mathcal{R}}^{\varepsilon}}D_{R}\nabla v_{r,n} ^{\varepsilon}\nabla v_{r,n}^{\varepsilon} dx\\
+\frac{\varepsilon^{\beta}}{2}\frac{d}{dt}\int_{\Omega_{\mathcal{M}}^{\varepsilon}}D_{M}^{\varepsilon}\nabla v_{m,n} ^{\varepsilon}\nabla v_{m,n}^{\varepsilon} dx\\
=\frac{d}{dt}\int_{\Omega_{\mathcal{L}}^{\varepsilon}}B_{L}P_{\delta}(v_{l,n}^{\varepsilon}-u_{b})\nabla v_{l,n}^{\varepsilon} dx+\frac{d}{dt}\int_{\Omega_{\mathcal{R}}^{\varepsilon}}B_{R}P_{\delta}(v_{r,n}^{\varepsilon}-u_{b})\nabla v_{r,n}^{\varepsilon} dx\\
+\varepsilon^{\gamma}\frac{d}{dt}\int_{\Omega_{\mathcal{M}}^{\varepsilon}}B_{M}P_{\delta}(v_{m,n}^{\varepsilon}-u_{b})\nabla v_{m,n}^{\varepsilon} dx\\-\int_{\Omega_{\mathcal{L}}^{\varepsilon}}B_{L}P_{\delta}^{\prime}(v_{l,n}^{\varepsilon}-u_{b})\partial_{t}(v_{l,n}^{\varepsilon}-u_{b})\nabla v_{l,n}^{\varepsilon} dx\\
-\int_{\Omega_{\mathcal{R}}^{\varepsilon}}B_{R}P_{\delta}^{\prime}(v_{r,n}^{\varepsilon}-u_{b})\partial_{t}(v_{r,n}^{\varepsilon}-u_{b})\nabla v_{r,n}^{\varepsilon} dx\\
-\varepsilon^{\gamma}\int_{\Omega_{\mathcal{M}}^{\varepsilon}}B_{M}P_{\delta}^{\prime}(v_{m,n}^{\varepsilon}-u_{b})\partial_{t}(v_{m,n}^{\varepsilon}-u_{b})\nabla v_{m,n}^{\varepsilon} dx\\
+\int_{\Omega_{\mathcal{L}}^{\varepsilon}}f_{b_{l}}\partial_{t}v_{l,n}^{\varepsilon} dx  +\int_{\Omega_{\mathcal{R}}^{\varepsilon}}f_{b_{r}}\partial_{t}v_{r,n}^{\varepsilon}dx +\varepsilon^{\alpha}\int_{\Omega_{\mathcal{M}}^{\varepsilon}}f_{a_{m}}^{\varepsilon}\partial_{t}v_{m,n}^{\varepsilon} dx \\
+\varepsilon^{\beta}\int_{\Omega_{\mathcal{M}}^{\varepsilon}}f_{b_{m}}^{\varepsilon}\partial_{t}v_{m,n}^{\varepsilon}dx
-\frac{d}{dt}\int_{\Gamma_{h}^{\varepsilon}\cap \partial\Omega_{\mathcal{L}}^{\varepsilon}}g_{b_{l}}v_{l,n}^{\varepsilon} d\sigma-\frac{d}{dt}\int_{\Gamma_{h}^{\varepsilon}\cap \partial\Omega_{\mathcal{R}}^{\varepsilon}}g_{b_{r}}v_{r,n}^{\varepsilon} d\sigma\\
-\varepsilon^{\xi}\frac{d}{dt}\int_{\Gamma_{0}^{\varepsilon}}g_{0}^{\varepsilon} v_{m,n}^{\varepsilon} d\sigma
-\frac{d}{dt}\varepsilon^{\beta}\int_{\Gamma_{0}^{\varepsilon}}g_{b_{0}}^{\varepsilon} v_{m,n}^{\varepsilon} d\sigma
+\int_{\Gamma_{h}^{\varepsilon}\cap \partial\Omega_{\mathcal{L}}^{\varepsilon}}\partial_{t}g_{b_{l}}v_{l,n}^{\varepsilon} d\sigma\\
+\int_{\Gamma_{h}^{\varepsilon}\cap \partial\Omega_{\mathcal{R}}^{\varepsilon}}\partial_{t}g_{b_{r}}v_{r,n}^{\varepsilon} d\sigma
+\varepsilon^{\xi}\int_{\Gamma_{0}^{\varepsilon}}\partial_{t}g_{0}^{\varepsilon} v_{m,n}^{\varepsilon} d\sigma+\varepsilon^{\beta}\int_{\Gamma_{0}^{\varepsilon}}\partial_{t}g_{b_{0}}^{\varepsilon} v_{m,n}^{\varepsilon} d\sigma\\
+\int_{\mathcal{B_{L}^{\varepsilon}}}(D_{L}-\varepsilon^{\beta}D_{M})\nabla u_{b}\cdot n_{l}\partial_{t}v_{l,n}^{\varepsilon}d\sigma+\int_{\mathcal{B_{R}^{\varepsilon}}}(D_{R}-\varepsilon^{\beta}D_{M})\nabla u_{b}\cdot n_{r}\partial_{t}v_{r,n}^{\varepsilon}d\sigma. 
\end{multline}
Now, integrating \eqref{ge9} from $0$ to $t$, we get

\begin{multline}\label{ge10}
\|\partial_{t}v_{l,n}^{\varepsilon}\|_{L^{2}(0,t;L^{2}(\Omega_{\mathcal{L}}^{\varepsilon}))}^{2}+\|\partial_{t}v_{r,n}^{\varepsilon}\|_{L^{2}(0,t;L^{2}(\Omega_{\mathcal{R}}^{\varepsilon}))}^{2}+\varepsilon^{\alpha}\|\partial_{t}v_{m,n}^{\varepsilon}\|_{L^{2}(0,t;L^{2}(\Omega_{\mathcal{M}}^{\varepsilon}))}^{2}\\+
\frac{1}{2}\int_{\Omega_{\mathcal{L}}^{\varepsilon}}D_{L}\nabla v_{l,n} ^{\varepsilon}\nabla v_{l,n}^{\varepsilon} dx+\frac{1}{2}\int_{\Omega_{\mathcal{R}}^{\varepsilon}}D_{R}\nabla v_{r,n} ^{\varepsilon}\nabla v_{r,n}^{\varepsilon} dx\\
+\frac{\varepsilon^{\beta}}{2}\int_{\Omega_{\mathcal{M}}^{\varepsilon}}D_{M}^{\varepsilon}\nabla v_{m,n} ^{\varepsilon}\nabla v_{m,n}^{\varepsilon} dx\\
=\int_{\Omega_{\mathcal{L}}^{\varepsilon}}B_{L}P_{\delta}(v_{l,n}^{\varepsilon}-u_{b})\nabla v_{l,n}^{\varepsilon} dx+\int_{\Omega_{\mathcal{R}}^{\varepsilon}}B_{R}P_{\delta}(v_{r,n}^{\varepsilon}-u_{b})\nabla v_{r,n}^{\varepsilon} dx\\
+\varepsilon^{\gamma}\int_{\Omega_{\mathcal{M}}^{\varepsilon}}B_{M}P_{\delta}(v_{m,n}^{\varepsilon}-u_{b})\nabla v_{m,n}^{\varepsilon} dx-\int_{0}^{t}\int_{\Omega_{\mathcal{L}}^{\varepsilon}}B_{L}P_{\delta}^{\prime}(v_{l,n}^{\varepsilon}-u_{b})\partial_{t}(v_{l,n}^{\varepsilon}-u_{b})\nabla v_{l,n}^{\varepsilon} dxdt\\
-\int_{0}^{t}\int_{\Omega_{\mathcal{R}}^{\varepsilon}}B_{R}P_{\delta}^{\prime}(v_{r,n}^{\varepsilon}-u_{b})\partial_{t}(v_{r,n}^{\varepsilon}-u_{b})\nabla v_{r,n}^{\varepsilon} dxdt
\\
-\varepsilon^{\gamma}\int_{0}^{t}\int_{\Omega_{\mathcal{M}}^{\varepsilon}}B_{M}P_{\delta}^{\prime}(v_{m,n}^{\varepsilon}-u_{b})\partial_{t}(v_{m,n}^{\varepsilon}-u_{b})\nabla v_{m,n}^{\varepsilon} dxdt\\
+\int_{0}^{t}\int_{\Omega_{\mathcal{L}}^{\varepsilon}}f_{b_{l}}\partial_{t}v_{l,n}^{\varepsilon} dx dt +\int_{0}^{t}\int_{\Omega_{\mathcal{R}}^{\varepsilon}}f_{b_{r}}\partial_{t}v_{r,n}^{\varepsilon}dxdt +\varepsilon^{\alpha}\int_{0}^{t}\int_{\Omega_{\mathcal{M}}^{\varepsilon}}f_{a_{m}}^{\varepsilon}\partial_{t}v_{m,n}^{\varepsilon} dxdt \\
+\varepsilon^{\beta}\int_{0}^{t}\int_{\Omega_{\mathcal{M}}^{\varepsilon}}f_{b_{m}}^{\varepsilon}\partial_{t}v_{m,n}^{\varepsilon}dxdt\\
-\int_{\Gamma_{h}^{\varepsilon}\cap \partial\Omega_{\mathcal{L}}^{\varepsilon}}g_{b_{l}}v_{l,n}^{\varepsilon} d\sigma-\int_{\Gamma_{h}^{\varepsilon}\cap \partial\Omega_{\mathcal{R}}^{\varepsilon}}g_{b_{r}}v_{r,n}^{\varepsilon} d\sigma-\varepsilon^{\xi}\int_{\Gamma_{0}^{\varepsilon}}g_{0}^{\varepsilon} v_{m,n}^{\varepsilon} d\sigma-\varepsilon^{\beta}\int_{\Gamma_{0}^{\varepsilon}}g_{b_{0}}^{\varepsilon} v_{m,n}^{\varepsilon} d\sigma\\
+\int_{0}^{t}\int_{\Gamma_{h}^{\varepsilon}\cap \partial\Omega_{\mathcal{L}}^{\varepsilon}}\partial_{t}g_{b_{l}}v_{l,n}^{\varepsilon} d\sigma dt+\int_{0}^{t}\int_{\Gamma_{h}^{\varepsilon}\cap \partial\Omega_{\mathcal{R}}^{\varepsilon}}\partial_{t}g_{b_{r}}v_{r,n}^{\varepsilon} d\sigma dt\\
+\varepsilon^{\xi}\int_{0}^{t}\int_{\Gamma_{0}^{\varepsilon}}\partial_{t}g_{0}^{\varepsilon} v_{m,n}^{\varepsilon} d\sigma dt+\varepsilon^{\beta}\int_{0}^{t}\int_{\Gamma_{0}^{\varepsilon}}\partial_{t}g_{b_{0}}^{\varepsilon} v_{m,n}^{\varepsilon} d\sigma dt\\
+\int_{0}^{t}\int_{\mathcal{B_{L}^{\varepsilon}}}(D_{L}-\varepsilon^{\beta}D_{M})\nabla u_{b}\cdot n_{l}\partial_{t}v_{l,n}^{\varepsilon}d\sigma dt+\int_{0}^{t}\int_{\mathcal{B_{R}^{\varepsilon}}}(D_{R}-\varepsilon^{\beta}D_{M})\nabla u_{b}\cdot n_{r}\partial_{t}v_{r,n}^{\varepsilon}d\sigma dt\\
+\frac{1}{2}\int_{\Omega_{\mathcal{L}}^{\varepsilon}}D_{L}\nabla v_{l,n}^{\varepsilon}(0) \nabla v_{l,,n}^{\varepsilon}(0) dx+\frac{1}{2}\int_{\Omega_{\mathcal{R}}^{\varepsilon}}D_{R}\nabla v_{r} ^{\varepsilon}(0)\nabla v_{r,n}^{\varepsilon}(0) dx\\
+\frac{\varepsilon^{\beta}}{2}\int_{\Omega_{\mathcal{M}}^{\varepsilon}}D_{M}^{\varepsilon}\nabla v_{m,n}^{\varepsilon}(0) \nabla v_{m,n}^{\varepsilon}(0) dx
-\int_{\Omega_{\mathcal{L}}^{\varepsilon}}B_{L}P_{\delta}(v_{l,n}^{\varepsilon}(0)-u_{b}(0))\nabla v_{l,n}^{\varepsilon}(0) dx\\
-\int_{\Omega_{\mathcal{R}}^{\varepsilon}}B_{R}P_{\delta}(v_{r,n}^{\varepsilon}(0)-u_{b}(0))\nabla v_{r,n}^{\varepsilon}(0) dx
-\varepsilon^{\gamma}\int_{\Omega_{\mathcal{M}}^{\varepsilon}}B_{M}P_{\delta}(v_{m,n}^{\varepsilon}(0)-u_{b}(0))\nabla v_{m,n}^{\varepsilon}(0) dx\\
+\int_{\Gamma_{h}^{\varepsilon}\cap \partial\Omega_{\mathcal{L}}^{\varepsilon}}g_{b_{l}}(0)v_{l,n}^{\varepsilon}(0) d\sigma+\int_{\Gamma_{h}^{\varepsilon}\cap \partial\Omega_{\mathcal{R}}^{\varepsilon}}g_{b_{r}}(0)v_{r,n}^{\varepsilon}(0) d\sigma+\varepsilon^{\xi}\int_{\Gamma_{0}^{\varepsilon}}g_{0}^{\varepsilon}(0) v_{m,n}^{\varepsilon}(0) d\sigma\\
+\varepsilon^{\beta}\int_{\Gamma_{0}^{\varepsilon}}g_{b_{0}}^{\varepsilon}(0) v_{m,n}^{\varepsilon} (0)d\sigma.
\end{multline}

By \eqref{a22}, Cauchy-Schwarz's inequality,  \eqref{e4} and Young's inequality, we obtain
\begin{equation}\label{ge11}
\begin{aligned}
\int_{\Omega_{\mathcal{L}}^{\varepsilon}}|B_{L}P_{\delta}(v_{l,n}^{\varepsilon}-u_{b})\nabla v_{l,n}^{\varepsilon} |dx&\leq C\int_{\Omega_{\mathcal{L}}^{\varepsilon}}|P_{\delta}(v_{l,n}^{\varepsilon}-u_{b})||\nabla v_{l,n}^{\varepsilon} |dx\\
&\leq C\|P_{\delta}(v_{l,n}^{\varepsilon}-u_{b})\|_{L^{2}(\Omega_{\mathcal{L}^{\varepsilon}})}\|\nabla v_{l,n}^{\varepsilon}\|_{L^{2}(\Omega_{\mathcal{L}^{\varepsilon}})}\\
&\leq C\|P_{\delta}(v_{l}^{\varepsilon}-u_{b})\|_{L^{2}(\Omega_{\mathcal{L}^{\varepsilon}})}^{2}+C\|\nabla v_{l,n}^{\varepsilon}\|_{L^{2}(\Omega_{\mathcal{L}^{\varepsilon}})}^{2}\\
&\leq C+C\|\nabla v_{l,n}^{\varepsilon}\|_{L^{2}(\Omega_{\mathcal{L}^{\varepsilon}})}^{2}.
\end{aligned}
\end{equation}
Similarly it holds
\begin{equation}\label{ge12}
\int_{\Omega_{\mathcal{R}}^{\varepsilon}}|B_{L}P_{\delta}(v_{r,n}^{\varepsilon}-u_{b})\nabla v_{r,n}^{\varepsilon} |dx\leq C+C\|\nabla v_{r,n}^{\varepsilon}\|_{L^{2}(\Omega_{\mathcal{R}^{\varepsilon}})}^{2},
\end{equation}
and
\begin{equation}\label{ge13}
\varepsilon^{\gamma}\int_{\Omega_{\mathcal{M}}^{\varepsilon}}|B_{M}P_{\delta}(v_{m,n}^{\varepsilon}-u_{b})\nabla v_{m,n}^{\varepsilon} |dx\leq C+C\varepsilon^{\beta}\|\nabla v_{m,n}^{\varepsilon}\|_{L^{2}(\Omega_{\mathcal{M}}^{\varepsilon})}^{2},
\end{equation}
for \eqref{ge13} we used the assumption \ref{assump6}.
Using Cauchy Schwarz's inequality, for $\eta>0$ and $ C(\eta)>0$ 

\begin{equation}\label{ge14}
\int_{0}^{t}\int_{\Omega_{\mathcal{L}}^{\varepsilon}}|f_{b_{l}}\partial_{t}v_{l,n}^{\varepsilon} |dxdt \leq C(\eta) \|f_{b_{l}}\|_{L^{2}(0,t;L^{2}(\Omega_{\mathcal{L}}^{\varepsilon}))}^{2}+\eta \|\partial_{t}v_{l,n}^{\varepsilon}\|_{L^{2}(0,t;L^{2}(\Omega_{\mathcal{L}}^{\varepsilon}))}^{2},
\end{equation}

\begin{equation}\label{ge15}
\int_{0}^{t}\int_{\Omega_{\mathcal{R}}^{\varepsilon}}|f_{b_{r}}\partial_{t}v_{r,n}^{\varepsilon} |dxdt \leq C(\eta) \|f_{b_{r}}\|_{L^{2}(0,t;L^{2}(\Omega_{\mathcal{R}}^{\varepsilon}))}^{2}+\eta \|\partial_{t}v_{r,n}^{\varepsilon}\|_{L^{2}(0,t;L^{2}(\Omega_{\mathcal{R}}^{\varepsilon}))}^{2},
\end{equation}

\begin{equation}\label{ge16}
\begin{aligned}
\varepsilon^{\alpha}\int_{0}^{t}\int_{\Omega_{\mathcal{M}}^{\varepsilon}}|f_{a_{m}}\partial_{t}v_{m,n}^{\varepsilon} |dxdt& \leq \varepsilon^{\alpha}C(\eta) \|f_{a_{m}}\|_{L^{2}(0,t;L^{2}(\Omega_{\mathcal{M}}^{\varepsilon}))}^{2}\\
&\hspace{2cm}+\varepsilon^{\alpha}\eta \|\partial_{t}v_{m,n}^{\varepsilon}\|_{L^{2}(0,t;L^{2}(\Omega_{\mathcal{M}}^{\varepsilon}))}^{2},
\end{aligned}
\end{equation}

\begin{equation}\label{ge17}
\varepsilon^{\beta}\int_{0}^{t}\int_{\Omega_{\mathcal{M}}^{\varepsilon}}|f_{b_{m}}^{\varepsilon}\partial_{t}v_{m}^{\varepsilon} |dxdt \leq \varepsilon^{\beta}C(\eta) \|f_{b_{m}}^{\varepsilon}\|_{L^{2}(0,t;L^{2}(\Omega_{\mathcal{M}}^{\varepsilon}))}^{2}+\varepsilon^{\alpha}\eta \|\partial_{t}v_{m,n}^{\varepsilon}\|_{L^{2}(0,t;L^{2}(\Omega_{\mathcal{M}}^{\varepsilon}))}^{2}.
\end{equation}

Using the regularity and the structure of $P_{\delta}(\cdot)$ together with Cauchy-Schwarz's inequality, Young's inequality, \eqref{a22}, \eqref{e2} and \ref{assump7}, we see that 

\begin{equation}\label{ge18}
\begin{aligned}
\int_{0}^{t}\int_{\Omega_{\mathcal{L}}^{\varepsilon}}|B_{L}P_{\delta}^{\prime}(v_{l,n}^{\varepsilon}-u_{b})\partial_{t}&(v_{l,n}^{\varepsilon}-u_{b})\nabla v_{l,n}^{\varepsilon} |dxdt\\
&\leq C\int_{0}^{t}\int_{\Omega_{\mathcal{L}}^{\varepsilon}}|\partial_{t}(v_{l,n}^{\varepsilon}-u_{b})\nabla v_{l,n}^{\varepsilon} |dxdt\\
&\leq C \|\partial_{t}v_{l,n}^{\varepsilon}\|_{L^{2}(0,t;L^{2}(\Omega_{\mathcal{L}}^{\varepsilon}))}\|\nabla v_{l,n}^{\varepsilon}\|_{L^{2}(0,t;L^{2}(\Omega_{\mathcal{L}}^{\varepsilon}))}\\
&\hspace{1cm}+C \|\partial_{t}u_{b}\|_{L^{2}(0,t;L^{2}(\Omega_{\mathcal{L}}^{\varepsilon}))}\|\nabla v_{l,n}^{\varepsilon}\|_{L^{2}(0,t;L^{2}(\Omega_{\mathcal{L}}^{\varepsilon}))}\\
&\leq C\eta \|\partial_{t}v_{l,n}^{\varepsilon}\|_{L^{2}(0,t;L^{2}(\Omega_{\mathcal{L}}^{\varepsilon}))}^{2}+C(\eta) \|\nabla v_{l,n}^{\varepsilon}\|_{L^{2}(0,t;L^{2}(\Omega_{\mathcal{L}}^{\varepsilon}))}^{2}\\
&\hspace{1cm }+C+C \|\nabla v_{l,n}^{\varepsilon}\|_{L^{2}(0,t;L^{2}(\Omega_{\mathcal{L}}^{\varepsilon}))}^{2}\\
&\leq \eta C\|\partial_{t}v_{l,n}^{\varepsilon}\|_{L^{2}(0,t;L^{2}(\Omega_{\mathcal{L}}^{\varepsilon}))}^{2}+C(\eta).
\end{aligned}
\end{equation}
Similarly, it holds
\begin{equation}\label{ge19}
\begin{aligned}
\int_{0}^{t}\int_{\Omega_{\mathcal{R}}^{\varepsilon}}|B_{R}P_{\delta}^{\prime}(v_{r,n}^{\varepsilon}-u_{b})\partial_{t}(v_{r,n}^{\varepsilon}-u_{b})\nabla v_{r,n}^{\varepsilon}| &dxdt\\
&\leq C(\eta)+\eta C\|\partial_{t}v_{r,n}^{\varepsilon}\|_{L^{2}(0,t;L^{2}(\Omega_{\mathcal{R}}^{\varepsilon}))}^{2},
\end{aligned}
\end{equation}
\begin{equation}\label{ge20}
\begin{aligned}
\varepsilon^{\gamma}\int_{0}^{t}\int_{\Omega_{\mathcal{M}}^{\varepsilon}}|B_{M}P_{\delta}^{\prime}(v_{m,n}^{\varepsilon}-u_{b})\partial_{t}(v_{m,n}^{\varepsilon}&-u_{b})\nabla v_{m,n}^{\varepsilon} |dxdt\\
&\leq C(\eta)+\eta C\varepsilon^{\gamma}\|\partial_{t}v_{m,n}^{\varepsilon}\|_{L^{2}(0,t;L^{2}(\Omega_{\mathcal{M}}^{\varepsilon}))}^{2}.
\end{aligned}
\end{equation}
Using (A4), \eqref{l1} and Young's inequality, we get
\begin{equation}\label{ge21}
\begin{aligned}
\int_{\Gamma_{h}^{\varepsilon}\cap \partial\Omega_{\mathcal{L}}^{\varepsilon}}|g_{b_{l}}v_{l,n}^{\varepsilon} |d\sigma&\leq \|g_{b_{l}}\|_{L^{2}(\Gamma_{h}^{\varepsilon}\cap \partial\Omega_{\mathcal{L}}^{\varepsilon})}\| v_{l,n}^{\varepsilon}\|_{L^{2}(\Gamma_{h}\cap \partial\Omega_{\mathcal{L}}^{\varepsilon})}\\
&\leq C\|g_{b_{l}}\|_{L^{2}(\Gamma_{h}^{\varepsilon}\cap \partial\Omega_{\mathcal{L}}^{\varepsilon})}\|\nabla v_{l,n}^{\varepsilon}\|_{L^{2}(\Omega_{\mathcal{L}}^{\varepsilon})}\\
&\leq C(t)+C\|\nabla v_{l,n}^{\varepsilon}\|_{L^{2}(\Omega_{\mathcal{L}}^{\varepsilon})}^{2}.
\end{aligned}
\end{equation}
Similarly, we have 

\begin{equation}\label{ge22}
\begin{aligned}
\int_{\Gamma_{h}^{\varepsilon}\cap \partial\Omega_{\mathcal{R}}^{\varepsilon}}|g_{b_{r}}v_{r,n}^{\varepsilon} |d\sigma\leq C(t)+C\|\nabla v_{r,n}^{\varepsilon}\|_{L^{2}(\Omega_{\mathcal{R}}^{\varepsilon})}^{2}.
\end{aligned}
\end{equation}

Using similar arguments as those used in the proof of \eqref{nee6}, namely,  Cauchy Schwarz's inequality, Lemma \ref{LN1}, Young's inequality and (A6), we get
\begin{equation}\label{ge23}
\begin{aligned}
\varepsilon^{\xi}\int_{\Gamma_{0}^{\varepsilon}}|g_{0}^{\varepsilon}  v_{m,n}^{\varepsilon}| d\sigma&\leq \varepsilon^{\xi-\frac{1}{2}}C\|g_{0}^{\varepsilon}\|_{L^{2}(\Gamma_{0}^{\varepsilon})}^{2}+\varepsilon^{\alpha}C\| v_{m,n}^{\varepsilon}\|_{L^{2}(\Omega_{\mathcal{M}}^{\varepsilon})}^{2}\\
&+C\varepsilon^{\xi+\frac{1}{2}}\|g_{0}^{\varepsilon}\|_{L^{2}(\Gamma_{0}^{\varepsilon})}^{2}+\varepsilon^{\beta} \|\nabla v_{m,n}^{\varepsilon}\|_{L^{2}(\Omega_{\mathcal{M}}^{\varepsilon})}^{2}\\
&\leq C(t)+\varepsilon^{\beta} \|\nabla v_{m,n}^{\varepsilon}\|_{L^{2}(\Omega_{\mathcal{M}}^{\varepsilon})}^{2}.
\end{aligned}
\end{equation}
\\
Similarly, we get
\begin{equation}\label{ge24}
\varepsilon^{\beta}\int_{\Gamma_{0}^{\varepsilon}}|g_{b_{0}} v_{m,n}^{\varepsilon} |d\sigma\leq C(t)+\varepsilon^{\beta} \|\nabla v_{m,n}^{\varepsilon}\|_{L^{2}(\Omega_{\mathcal{M}}^{\varepsilon})}^{2}.
\end{equation}

By using similar arguments as in  \eqref{ge21}, \eqref{ge22}, \eqref{ge23} and \eqref{ge24} together with \eqref{e2}, we get

\begin{equation}\label{ge25}
\int_{0}^{t}\int_{\Gamma_{h}^{\varepsilon}\cap \partial\Omega_{\mathcal{L}}^{\varepsilon}}|\partial_{t}g_{b_{l}}v_{l,n}^{\varepsilon} |d\sigma dt\leq C,
\end{equation}

\begin{equation}\label{ge26}
\int_{0}^{t}\int_{\Gamma_{h}^{\varepsilon}\cap \partial\Omega_{\mathcal{R}}^{\varepsilon}}|\partial_{t}g_{b_{r}}v_{r,n}^{\varepsilon} |d\sigma dt\leq C,
\end{equation}

\begin{equation}\label{ge27}
\varepsilon^{\xi}\int_{0}^{t}\int_{\Gamma_{0}^{\varepsilon}}|\partial_{t}g_{0}^{\varepsilon} v_{m,n}^{\varepsilon} |d\sigma dt\leq C,
\end{equation}

\begin{equation}\label{ge28}
\varepsilon^{\beta}\int_{0}^{t}\int_{\Gamma_{0}^{\varepsilon}}|\partial_{t}g_{b_{0}}^{\varepsilon} v_{m}^{\varepsilon} |d\sigma dt\leq C.
\end{equation}
Furthermore, we have
\begin{equation}\label{ge29}
\begin{aligned}
\int_{0}^{t}\int_{\mathcal{B_{L}^{\varepsilon}}}(D_{L}-\varepsilon^{\beta}D_{M})\nabla u_{b}\cdot n_{l}\partial_{t}v_{l,n}^{\varepsilon}d\sigma dt&=\int_{\mathcal{B_{L}^{\varepsilon}}}(D_{L}-\varepsilon^{\beta}D_{M})\nabla u_{b}\cdot n_{l}v_{l,n}^{\varepsilon}d\sigma\\
&-\int_{0}^{t}\int_{\mathcal{B_{L}^{\varepsilon}}}(D_{L}-\varepsilon^{\beta}D_{M})\nabla \partial_{t}u_{b}\cdot n_{l}v_{l,n}^{\varepsilon}d\sigma dt\\
&-\int_{\mathcal{B_{L}^{\varepsilon}}}(D_{L}-\varepsilon^{\beta}D_{M})\nabla u_{b}(0)\cdot n_{l}v_{l,n}^{\varepsilon}(0)d\sigma,
\end{aligned}
\end{equation}
and
\begin{equation}\label{ge30}
\begin{aligned}
\int_{0}^{t}\int_{\mathcal{B_{R}^{\varepsilon}}}(D_{R}-\varepsilon^{\beta}D_{M})\nabla u_{b}\cdot n_{r}\partial_{t}v_{r,n}^{\varepsilon}&d\sigma dt=\int_{\mathcal{B_{R}^{\varepsilon}}}(D_{R}-\varepsilon^{\beta}D_{M})\nabla u_{b}\cdot n_{r}v_{r,n}^{\varepsilon}d\sigma\\
&-\int_{0}^{t}\int_{\mathcal{B_{R}^{\varepsilon}}}(D_{R}-\varepsilon^{\beta}D_{M})\nabla \partial_{t}u_{b}\cdot n_{r}v_{r,n}^{\varepsilon}d\sigma dt\\
&-\int_{\mathcal{B_{R}^{\varepsilon}}}(D_{R}-\varepsilon^{\beta}D_{M})\nabla u_{b}(0)\cdot n_{r}v_{r,n}^{\varepsilon}(0)d\sigma.
\end{aligned}
\end{equation}
Now, using \eqref{dinf}, Cauchy-Schwarz's inequality, Young's inequality \ref{assump6} and \eqref{l5}, we get
\begin{equation}\label{ge31}
\begin{aligned}
\int_{\mathcal{B_{L}^{\varepsilon}}}|(D_{L}-\varepsilon^{\beta}D_{M})\nabla u_{b}\cdot n_{l}v_{l,n}^{\varepsilon}|d\sigma&\leq C\int_{\mathcal{B_{L}^{\varepsilon}}}|\nabla u_{b}||v_{l,n}^{\varepsilon}|d\sigma\\
&\leq C(t)+C \|\nabla v_{l,n}^{\varepsilon}\|_{L^{2}(\Omega_{\mathcal{L}}^{\varepsilon})}^{2}
\end{aligned}
\end{equation}

\begin{equation}\label{ge32}
\begin{aligned}
\int_{\mathcal{B_{R}^{\varepsilon}}}|(D_{R}-\varepsilon^{\beta}D_{M})\nabla u_{b}\cdot n_{r}v_{r,n}^{\varepsilon}|d\sigma\leq C(t)+C \|\nabla v_{r,n}^{\varepsilon}\|_{L^{2}(\Omega_{\mathcal{R}}^{\varepsilon})}^{2}
\end{aligned}
\end{equation}
\begin{equation}\label{ge34}
\begin{aligned}
\int_{0}^{t}\int_{\mathcal{B_{L}^{\varepsilon}}}|(D_{L}-\varepsilon^{\beta}D_{M})\nabla \partial_{t}&u_{b}\cdot n_{l}v_{l,n}^{\varepsilon}|d\sigma dt\\
&\leq C\|\nabla\partial_{t}u_{b}\|_{L^{2}(0,t;L^{2}(\Omega_{\mathcal{L}}^{\varepsilon}))}\|\nabla v_{l,n}^{\varepsilon}\|_{L^{2}(0,t;L^{2}(\Omega_{\mathcal{L}}^{\varepsilon}))}\\
&\leq C\|\nabla\partial_{t}u_{b}\|_{L^{2}(0,t;L^{2}(\Omega_{\mathcal{L}}^{\varepsilon}))}^{2}+C\|\nabla v_{l,n}^{\varepsilon}\|_{L^{2}(0,t;L^{2}(\Omega_{\mathcal{L}}^{\varepsilon}))}^{2}\\
&\leq C,
\end{aligned}
\end{equation}
and
\begin{equation}\label{gee34}
\int_{0}^{t}\int_{\mathcal{B_{R}^{\varepsilon}}}|(D_{R}-\varepsilon^{\beta}D_{M})\nabla \partial_{t}u_{b}\cdot n_{r}v_{l,n}^{\varepsilon}|d\sigma dt\leq C.
\end{equation}
Using \ref{assump4} leads to
\begin{equation}\label{gee35}
\begin{aligned}
\int_{\mathcal{B_{R}^{\varepsilon}}}|(D_{R}-\varepsilon^{\beta}D_{M})\nabla u_{b}(0)\cdot n_{r}v_{r,n}^{\varepsilon}(0)|d\sigma&\leq C\int_{\mathcal{B_{R}^{\varepsilon}}}|\nabla u_{b}(0)||h_{b_{r},n}^{\varepsilon}|d\sigma\\
&\leq C +C\|\nabla h_{b_{r},n}^{\varepsilon}\|_{L^{2}(\Omega_{\mathcal{R}}^{\varepsilon})}^{2}\\
&\leq C,
\end{aligned}
\end{equation}
where $h_{b_{r},n}^{\varepsilon}:=\sum_{k=0}^{n}d_{r,k}^{\varepsilon}(0)w_{r,k}^{\varepsilon}$ and by using $d_{r,k}^{\varepsilon}(0)=\int_{\Omega_{\mathcal{L}}^{\varepsilon}}h_{b_{r},n}^{\varepsilon}w_{l,k}^{\varepsilon}dx $, we get \\$\| h_{b_{r},n}^{\varepsilon}\|_{H^{1}(\Omega_{\mathcal{R}}^{\varepsilon})}^{2}\leq \| h_{b_{r}}^{\varepsilon}\|_{H^{1}(\Omega_{\mathcal{R}}^{\varepsilon})}^{2}$,
\begin{equation}\label{ge35}
\begin{aligned}
\int_{\mathcal{B_{L}^{\varepsilon}}}|(D_{L}-\varepsilon^{\beta}D_{M})\nabla u_{b}(0)\cdot n_{l}v_{m,n}^{\varepsilon}(0)|d\sigma&\leq C.
\end{aligned}
\end{equation}

Using \eqref{dinf}, \ref{assump4} and \ref{assump6}, allow us to write

\begin{equation}\label{}
\begin{aligned}
\frac{1}{2}\int_{\Omega_{\mathcal{L}}^{\varepsilon}}|D_{L}\nabla v_{l,n}(0) ^{\varepsilon}\nabla v_{l,n}^{\varepsilon}(0)| dx&\leq C\int_{\Omega_{\mathcal{L}}^{\varepsilon}}|\nabla h_{b_{l},n}^{\varepsilon}|dx\\
&\leq C,
\end{aligned}
\end{equation}

\begin{equation}\label{}
\begin{aligned}
\frac{1}{2}\int_{\Omega_{\mathcal{R}}^{\varepsilon}}|D_{R}\nabla v_{r,n}(0) ^{\varepsilon}\nabla v_{r,n}^{\varepsilon}(0)| dx\leq C.
\end{aligned}
\end{equation}
Using the structure of $P_{\delta}(\cdot)$, there exist a $k\in \mathbb{R}$ such that $P_{\delta}(k)=0$. Now we use Mean Value Theorem, \eqref{a22}, we get
\begin{equation}\label{}
\frac{\varepsilon^{\beta}}{2}\int_{\Omega_{\mathcal{M}}^{\varepsilon}}D_{M}^{\varepsilon}\nabla v_{m,n}(0) ^{\varepsilon}\nabla v_{m,n}^{\varepsilon}(0) dx\leq C,
\end{equation}

\begin{equation}\label{}
\begin{aligned}
\int_{\Omega_{\mathcal{R}}^{\varepsilon}}|B_{R}P_{\delta}(v_{r,n}^{\varepsilon}(0)-u_{b}(0))\nabla v_{r,n}^{\varepsilon}(0)| dx&=\int_{\Omega_{\mathcal{R}}^{\varepsilon}}|B_{L}P_{\delta}(h_{b_{r}}^{\varepsilon}-u_{b}(0))\nabla h_{b_{r},n}^{\varepsilon}| dx\\
&\leq C\int_{\Omega_{\mathcal{R}}^{\varepsilon}}|h_{b_{r}}^{\varepsilon}-u_{b}(0)-k||\nabla h_{b_{r},n}^{\varepsilon}| dx\\
&\leq C,
\end{aligned}
\end{equation}
and
\begin{equation}\label{}
\begin{aligned}
\int_{\Omega_{\mathcal{L}}^{\varepsilon}}|B_{L}P_{\delta}(v_{l,n}^{\varepsilon}(0)-u_{b}(0))\nabla v_{l,n}^{\varepsilon}(0)| dx&\leq C,
\end{aligned}
\end{equation}

\begin{equation}\label{gel1}
\varepsilon^{\gamma}\int_{\Omega_{\mathcal{M}}^{\varepsilon}}|B_{M}P_{\delta}(v_{m,n}^{\varepsilon}(0)-u_{b}(0))\nabla v_{m,n}^{\varepsilon}(0)| dx\leq C.
\end{equation}
By using similar arguments of \eqref{ge21}, \eqref{ge22}, \eqref{ge23} and \eqref{ge24} together with (A5), we get
\begin{align}\label{}
\int_{\Gamma_{h}^{\varepsilon}\cap \partial\Omega_{\mathcal{L}}^{\varepsilon}}|g_{b_{l}}(0)v_{l,n}^{\varepsilon}(0)| d\sigma&\leq C,\\
\int_{\Gamma_{h}^{\varepsilon}\cap \partial\Omega_{\mathcal{R}}^{\varepsilon}}|g_{b_{r}}(0)v_{r,n}^{\varepsilon}(0) |d\sigma&\leq C,\\
\varepsilon^{\xi}\int_{\Gamma_{0}^{\varepsilon}}|g_{0}^{\varepsilon}(0) v_{m,n}^{\varepsilon}(0)| d\sigma&\leq C,\\
\varepsilon^{\beta}\int_{\Gamma_{0}^{\varepsilon}}|g_{b_{0}}(0) v_{m,n}^{\varepsilon} (0)|d\sigma&\leq C.\label{gen1}
\end{align}

Choosing $\eta>0$ small enough and  using ellipticity condition together with \eqref{ge10}-\eqref{gen1}, we obtain for $t=T$ the bound

\begin{multline}\label{gel2}
\|\partial_{t}v_{l,n}^{\varepsilon}\|_{L^{2}(0,T;L^{2}(\Omega_{\mathcal{L}}^{\varepsilon}))}^{2}+\|\partial_{t}v_{r,n}^{\varepsilon}\|_{L^{2}(0,T;L^{2}(\Omega_{\mathcal{R}}^{\varepsilon}))}^{2}+\varepsilon^{\alpha}\|\partial_{t}v_{m,n}^{\varepsilon}\|_{L^{2}(0,T;L^{2}(\Omega_{\mathcal{M}}^{\varepsilon}))}^{2}\\ 
\leq C+C(t)+C \left(	\|\nabla v_{l}^{\varepsilon}\|_{L^{2}(\Omega_{\mathcal{L}}^{\varepsilon})}^{2}+	\|\nabla v_{r,n}^{\varepsilon}\|_{L^{2}(\Omega_{\mathcal{R}}^{\varepsilon})}^{2}+\varepsilon^{\beta}\|\nabla v_{m,n}^{\varepsilon}\|_{L^{2}(\Omega_{\mathcal{M}}^{\varepsilon})}^{2}\right).
\end{multline}
From \eqref{gel2}, we can observe that, the time dependent constant $C(t)$ is a consequence of \eqref{ge22}, \eqref{ge23}, \eqref{ge31} and \eqref{ge32}. So, by using \ref{assump3} and \ref{assump4} we get $\int_{0}^{T}C(t)\leq C$.
Now, integrating \eqref{gel2} again from $0$ to $T$ with respect to $t$, and using \eqref{e2}, we get

\begin{equation}\label{gel3}
\|\partial_{t}v_{l,n}^{\varepsilon}\|_{L^{2}(0,T;L^{2}(\Omega_{\mathcal{L}}^{\varepsilon}))}^{2}+\|\partial_{t}v_{r,n}^{\varepsilon}\|_{L^{2}(0,T;L^{2}(\Omega_{\mathcal{R}}^{\varepsilon}))}^{2}+\varepsilon^{\alpha}\|\partial_{t}v_{m,n}^{\varepsilon}\|_{L^{2}(0,T;L^{2}(\Omega_{\mathcal{M}}^{\varepsilon}))}^{2}
\leq C.
\end{equation}

Now as an application of Aubin-Lions compactness lemma, we get 

\begin{equation}\label{gel4}
\|\partial_{t}v_{l}^{\varepsilon}\|_{L^{2}(0,T;L^{2}(\Omega_{\mathcal{L}}^{\varepsilon}))}^{2}+\|\partial_{t}v_{r}^{\varepsilon}\|_{L^{2}(0,T;L^{2}(\Omega_{\mathcal{R}}^{\varepsilon}))}^{2}+\varepsilon^{\alpha}\|\partial_{t}v_{m}^{\varepsilon}\|_{L^{2}(0,T;L^{2}(\Omega_{\mathcal{M}}^{\varepsilon}))}^{2}
\leq C.
\end{equation}

Hence we proved \eqref{e3}

\end{proof}

\subsection{Extension to fixed domain}\label{efd}
\begin{lemma}\label{L3}
	If $v_{m}^{\varepsilon}\in H^{1}(\Omega_{\mathcal{M}}^{\varepsilon})$, then there exists an extension of $v_{m}^{\varepsilon}$ to
	$H^{1}\left((-\varepsilon, \varepsilon)\times (0,h)\right)$ denote as $\tilde{v}_{m}^{\varepsilon}$ satisfying the following inequality
	\begin{equation}
	    \|\tilde{v}_{\varepsilon}^{m}\|_{H^{1}\left((-\varepsilon, \varepsilon)\times (0,h)\right)}\leq C\|v_{\varepsilon}^{m}\|_{H^{1}(\Omega_{\mathcal{M}}^{\varepsilon})}.
	\end{equation}
	\begin{proof}
		By using Theorem 9.7 of \cite{brezis2010functional} we can easly obtain the extension result for standard cell $Z$ with the inequality
		\begin{equation}\label{l2e1}
		\|\tilde{v}_{\varepsilon}^{m}\|_{H^{1}(Y)}\leq C \|v\|_{H^{1}(Z)} ,
		\end{equation}
		for some constant $C$.
	\\
	Now, using \eqref{l2e1}, we have
	\begin{align}\label{x1}
		\|\tilde{v}_{\varepsilon}^{m}\|_{H^{1}\left((-\varepsilon, \varepsilon)\times (0,h)\right)}&=\int_{(-\varepsilon, \varepsilon)\times (0,h)} |\tilde{v}_{\varepsilon}^{m}|^{2}+|\nabla\tilde{v}_{\varepsilon}^{m}|^{2} dx\\
		&=\sum_{0}^{h/\varepsilon}\int_{ke_{2}+\varepsilon Y} |\tilde{v}_{\varepsilon}^{m}|^{2}+|\nabla\tilde{v}_{\varepsilon}^{m}|^{2} dx\\
		&=\sum_{0}^{h/\varepsilon}\varepsilon^{2}\int_{ke_{2}+Y} |\tilde{v}_{\varepsilon}^{m}(\varepsilon x)|^{2}+\varepsilon^{2}|\nabla\tilde{v}_{\varepsilon}^{m}(\varepsilon x)|^{2} dx\\
		&\leq C \sum_{0}^{h/\varepsilon}\varepsilon^{2}\int_{ ke_{2}+Z} |{v}_{\varepsilon}^{m}(\varepsilon x)|^{2}+\varepsilon^{2}|\nabla{v}_{\varepsilon}^{m}(\varepsilon x)|^{2} dx\\
		&=C\sum_{0}^{h/\varepsilon}\int_{ ke_{2}+\varepsilon Z_{k}} |{v}_{\varepsilon}^{m}|^{2}+|\nabla{v}_{\varepsilon}^{m}|^{2} dx\\
		&=C\|v_{\varepsilon}^{m}\|_{H^{1}(\Omega_{\mathcal{M}}^{\varepsilon})}.
	\end{align}
	\end{proof}

\end{lemma}
To prove the above result we used a similar technique as used in \cite{ACERBI1992481}.
	\subsection{Two scale convergence for thin layer}\label{tsctl}
	Here we use Theorem \ref{T3} and obtain two scale limit of $v_{m}^{\varepsilon}$ as  $\varepsilon \rightarrow 0$  for the layer. We use two scale limit for layer definition similar to definition defined in  \cite{neuss2007effective} which is motivated from \cite{lukkassen2002two}. 
	\begin{definition}\label{D3}
		We define sequence of functions $v_{\varepsilon}^{m}\in L^{2}((0,T)\times \Omega_{\mathcal{M}}^{\varepsilon})$ two-scale converges to $v_{0}(t,\bar{x},y)\in L^{2}((0,T)\times \Sigma\times Z )$, If 
		\begin{equation}\label{}
		\lim_{\varepsilon \rightarrow 0}\frac{1}{\varepsilon}\int_{0}^{T}\int_{\Omega_{\mathcal{M}}^{\varepsilon}}v_{\varepsilon}(t,x)\psi (t,\bar{x},\frac{x}{\varepsilon})dxdt=\int_{0}^{T}\int_{\Sigma}\int_{Z}v_{0}(t,\bar{x},y)\psi (t,\bar{x},y)dydxdt,
		\end{equation}
		for all $\psi\in L^{2}((0,T)\times Z;C_{\#} (\overline{Z}))$,
		 where $\Sigma:=\{(0,x_{2})\in\Omega:x_{2}\in (0,h)\}$ and we denote the two-scale convergence of $v_{m}^{\varepsilon}$ to $v_{m}^{0}$ as $v_{m}^{\varepsilon}\overset{2-s}{\rightharpoonup}v_{m}^{0}$.
	\end{definition}
		\begin{definition}\label{D4}
		We define sequence of functions $v_{\varepsilon}^{m}\in L^{2}((0,T)\times \Gamma_{0}^{\varepsilon})$ two-scale converges to $v_{0}(t,\bar{x},y)\in L^{2}((0,T)\times \Sigma\times\partial Y_{0})$, If 
		\begin{equation}\label{}
		\lim_{\varepsilon \rightarrow 0}\frac{1}{\varepsilon}\int_{0}^{T}\int_{\Gamma_{0}^{\varepsilon}}v_{\varepsilon}(t,x)\psi (t,\bar{x},\frac{x}{\varepsilon})d\sigma_{x} dt=\int_{0}^{T}\int_{\Sigma}\int_{\partial{Y_{0}}}v_{0}(t,\bar{x},y)\psi (t,\bar{x},y)d\sigma_{y}d\bar{x}dt
		\end{equation}
			for all $\psi\in L^{2}((0,T)\times \Sigma;C_{\#} (\overline{\partial Y_{0}}))$.
	\end{definition}

\begin{theorem}\label{T4}
	For any sequence $v_{\varepsilon}^{m}\in  L^{2}((0,T)\times \Omega_{\mathcal{M}}^{\varepsilon})$ with the condition \begin{equation}\label{}
	\frac{1}{\varepsilon}\|v_{\varepsilon}^{m}\|_{L^{2}((0,T)\times \Omega_{\mathcal{M}}^{\varepsilon})}^{2}\leq C,
	\end{equation}
	for a constant $C$, we can find a subsequence, again denoted as $v_{\varepsilon}^{m}$, such that $v_{\varepsilon}^{m}$ two-scale converges to $v_{0}^{m}\in L^{2}((0,T)\times \Sigma\times Z)$.
\end{theorem}
\begin{theorem}\label{T5}
	For any sequence $v_{\varepsilon}^{m}\in  L^{2}(\Gamma_{0}^{\varepsilon}\times(0,T))$ with the condition \begin{equation}\label{}
\|v_{\varepsilon}^{m}\|_{L^{2}(\Gamma_{0}^{\varepsilon}\times(0,T))}^{2}\leq C,
	\end{equation}
	for a constant $C$, we can find a subsequence, again denoted as $v_{\varepsilon}^{m}$, such that $v_{\varepsilon}^{m}$ two-scale converges to $v_{0}^{m}\in L^{2}(\Sigma\times \partial Y_{0}\times (0,T))$.
\end{theorem}
\begin{proof}
	For proof of Theorem \ref{T4} and Theorem \ref{T5}, refer proof of Theorem 4.4 of \cite{effectiveapratim} and Proposition 4.2 of \cite{neuss2007effective}.
\end{proof}
\begin{theorem}\label{T6}
	Let $(v_{l}^{\varepsilon},v_{m}^{\varepsilon},v_{r}^{\varepsilon})$ be the weak solution of \eqref{wf}, then there exist\\ $(v_{l}^{0},v_{r}^{0})\in (L^{2}(0,T;H^{1}(\Omega_{\mathcal{L}})),L^{2}(0,T;H^{1}(\Omega_{\mathcal{R}})))$ such that
	\begin{align}
	\mathbbm{1}_{\Omega_{\mathcal{L}}^{\varepsilon}}v_{l}^{\varepsilon}\rightarrow v_{l}^{0}\hspace{2cm}\mbox{on}\hspace{1cm}L^{2}((0,T)\times\Omega_{\mathcal{L}}),\label{t51}\\
		\mathbbm{1}_{\Omega_{\mathcal{R}}^{\varepsilon}}v_{r}^{\varepsilon}\rightarrow v_{r}^{0}\hspace{2cm}\mbox{on}\hspace{1cm}L^{2}((0,T)\times\Omega_{\mathcal{R}}),\label{t52}\\
	\mathbbm{1}_{\Omega_{\mathcal{L}}^{\varepsilon}}v_{l}^{\varepsilon}(t,x_{1},-\varepsilon)\rightarrow v_{l}^{0}(t,x_{1},0)\hspace{2cm}\mbox{on}\hspace{1cm}L^{2}((0,T)\times\Omega_{\mathcal{L}}),\label{t53}\\
			\mathbbm{1}_{\Omega_{\mathcal{R}}^{\varepsilon}}v_{r}^{\varepsilon}(t,x_{1},\varepsilon)\rightarrow v_{r}^{0}(t,x_{1},0)\hspace{2cm}\mbox{on}\hspace{1cm}L^{2}((0,T)\times\Omega_{\mathcal{R}}),\label{t54}\\
			\mathbbm{1}_{\Omega_{\mathcal{L}}^{\varepsilon}}\nabla v_{l}^{\varepsilon}\overset{w}\rightharpoonup \nabla v_{l}^{0}\hspace{2cm}\mbox{on}\hspace{1cm}L^{2}((0,T)\times\Omega_{\mathcal{L}}),\label{t55}\\	
				\mathbbm{1}_{\Omega_{\mathcal{R}}^{\varepsilon}}\nabla v_{r}^{\varepsilon}\overset{w}\rightharpoonup \nabla v_{r}^{0}\hspace{2cm}\mbox{on}\hspace{1cm}L^{2}((0,T)\times\Omega_{\mathcal{R}}),\label{t56}\\
				\mathbbm{1}_{\Omega_{\mathcal{L}}^{\varepsilon}}\partial_{t}v_{l}^{\varepsilon}\overset{w}\rightharpoonup \partial_{t}v_{l}^{0}\hspace{2cm}\mbox{on}	\hspace{1cm}L^{2}(0,T;L^{2}(\Omega_{\mathcal{L}})),\label{t57}\\
					\mathbbm{1}_{\Omega_{\mathcal{R}}^{\varepsilon}}\partial_{t}v_{r}^{\varepsilon}\overset{w}\rightharpoonup \partial_{t}v_{r}^{0}\hspace{2cm}\mbox{on}	\hspace{1cm}L^{2}(0,T;L^{2}(\Omega_{\mathcal{R}})),\label{t58}\\
					\mathbbm{1}_{\Omega_{\mathcal{L}}^{\varepsilon}}P_{\delta}(v_{l}^{\varepsilon}-u_{b})\rightarrow P_{\delta}(v_{l}^{0}-u_{b}) \hspace{2cm}\mbox{on}\hspace{1cm}L^{2}((0,T)\times\Omega_{\mathcal{L}}),\label{t59}\\
						\mathbbm{1}_{\Omega_{\mathcal{R}}^{\varepsilon}}P_{\delta}(v_{r}^{\varepsilon}-u_{b})\rightarrow P_{\delta}(v_{r}^{0}-u_{b}) \hspace{2cm}\mbox{on}\hspace{1cm}L^{2}((0,T)\times\Omega_{\mathcal{R}})\label{t510},
	\end{align}
	as $\varepsilon\rightarrow0$.
	\end{theorem}
\begin{proof}
	Proof of convergence \eqref{t51}, \eqref{t52}, \eqref{t55})-\eqref{t58} is application of\\
	Lemma \ref{L2}, Lemma \ref{L1} Theorem \ref{T3} and Lions-Aubin's compactness lemma (see \cite{aubin1963analyse}). For details of the proof see Proposition 2.1 in \cite{neuss2007effective}.
To prove convergence result \eqref{t59} we use the following estimate 
	\begin{equation}\label{t59p}
	\begin{aligned}
\|\mathbbm{1}_{\Omega_{\mathcal{L}}^{\varepsilon}}P_{\delta}(v_{l}^{\varepsilon}-u_{b})-&P_{\delta}(v_{l}^{0}-u_{b})\|_{L^{2}(0,T;L^{2}(\Omega_{\mathcal{L}}))}\\
=&\|\mathbbm{1}_{\Omega_{\mathcal{L}}^{\varepsilon}}P_{\delta}(v_{l}^{\varepsilon}-u_{b})-\mathbbm{1}_{\Omega_{\mathcal{L}}^{\varepsilon}}P_{\delta}(v_{l}^{0}-u_{b})\\
&\hspace{2cm}-\mathbbm{1}_{\Omega_{\mathcal{L}}\backslash \Omega_{\mathcal{L}}^{\varepsilon}}P_{\delta}(v_{l}^{0}-u_{b})\|_{L^{2}(0,T;L^{2}(\Omega_{\mathcal{L}}))}\\
&\leq \|\mathbbm{1}_{\Omega_{\mathcal{L}}^{\varepsilon}}P_{\delta}(v_{l}^{\varepsilon}-u_{b})-\mathbbm{1}_{\Omega_{\mathcal{L}}^{\varepsilon}}P_{\delta}(v_{l}^{0}-u_{b})\|_{L^{2}(0,T;L^{2}(\Omega_{\mathcal{L}}))}\\
&\hspace{2cm}+\|\mathbbm{1}_{\Omega_{\mathcal{L}}\backslash \Omega_{\mathcal{L}}^{\varepsilon}}P_{\delta}(v_{l}^{0}-u_{b})\|_{L^{2}(0,T;L^{2}(\Omega_{\mathcal{L}}))}\\
&\leq C	\|\mathbbm{1}_{\Omega_{\mathcal{L}}^{\varepsilon}}(v_{l}^{\varepsilon}-v_{l}^{0})\|_{L^{2}(0,T;L^{2}(\Omega_{\mathcal{L}}))}
\\
&\hspace{2cm}+\|\mathbbm{1}_{\Omega_{\mathcal{L}}\backslash \Omega_{\mathcal{L}}^{\varepsilon}}P_{\delta}(v_{l}^{0}-u_{b})\|_{L^{2}(0,T;L^{2}(\Omega_{\mathcal{L}}))},
	\end{aligned}
	\end{equation}
		to get the  inequality \eqref{t59p} we used the structure of $P_{\delta}$ operator, Mean Value Theorem and Minkowski's inequality. As a consequence of Monotone Convergence Theorem, we have 
		\begin{equation}\label{mce}
		\|\mathbbm{1}_{\Omega_{\mathcal{L}}\backslash \Omega_{\mathcal{L}}^{\varepsilon}}P_{\delta}(v_{l}^{0}-u_{b})\|_{L^{2}(0,T;L^{2}(\Omega_{\mathcal{L}}))}\rightarrow 0
	\end{equation}
	as $\varepsilon\rightarrow 0$.
	Now, using  \eqref{t51}, \eqref{t59p} and \eqref{mce} as $\varepsilon\rightarrow0$ we can conclude $	\mathbbm{1}_{\Omega_{\mathcal{L}}^{\varepsilon}}P_{\delta}(v_{l}^{\varepsilon}-u_{b})\rightarrow P_{\delta}(v_{l}^{0}-u_{b})$ strongly in $L^{2}((0,T)\times\Omega_{\mathcal{L}})$ as $\varepsilon\rightarrow0$.\\
	Similarly we can prove \eqref{t510}.
\end{proof}
\begin{theorem}\label{T7}
	Let $(v_{l}^{\varepsilon},v_{m}^{\varepsilon},v_{r}^{\varepsilon})$ be a weak solution of \eqref{wf}. Then there exists $v_{m}^{0}\in L^{2}((0,T)\times \Sigma; H^{1}_{\#}(Z))$ such that upto a subsequence, it holds 
	\begin{align}
	v_{m}^{\varepsilon}&\overset{2-s}{\rightharpoonup}v_{m}^{0},\label{t61}\\
		\partial_{t}v_{m}^{\varepsilon}&\overset{2-s}{\rightharpoonup} \partial_{t}v_{m}^{0}, \label{tt61}\\
	\varepsilon \nabla v_{m}^{\varepsilon}&\overset{2-s}{\rightharpoonup} \nabla_{y}v_{m}^{0}, \label{t62}\\
	\varepsilon P_{\delta}(v_{m}^{\varepsilon}-u_{b})&\overset{2-s}{\rightharpoonup} 0, \label{t63}
	\end{align}
	as $\varepsilon\rightarrow 0.$
\end{theorem}
\begin{proof}
	To prove \eqref{t61} and \eqref{t62} we use Theorem \ref{T3} and Theorem \ref{T4}. For details we refer Proposition 2.1 and Proposition 2.2 of \cite{neuss2007effective}. Using Theorem \ref{T3} we have 
	\begin{equation}\label{t6e1}
	\frac{1}{\varepsilon}\|P_{\delta}(v_{m}^{\varepsilon}-u_{b})\|_{L^{2}(\Omega_{\mathcal{M}}^{\varepsilon})}^{2}\leq C.
	\end{equation}
	Now using Theorem \ref{T4} for \eqref{t6e1}, we have 
	\begin{equation}\label{}
	P_{\delta}(v_{m}^{\varepsilon}-u_{b})\overset{2-s}{\rightharpoonup} w ,
	\end{equation}
	where $w\in L^{2}((0,T)\times \Sigma\times Z)$. Consequently, we get $	\varepsilon P_{\delta}(v_{m}^{\varepsilon}-u_{b})\overset{2-s}{\rightharpoonup} 0$.
\end{proof}

	\section{Macroscopic model }\label{macro}
		In this section we derive upscaled equations and effective transmission conditions and coefficients  for  a variable selection  of scalings depending on the small parameter $\varepsilon$; see Table \ref{table1}.
	
	\begin{center}
	\begin{table}[ht]
	\caption{List of discussed scalings.}\label{table1}
		\begin{tabular}{ |p{3cm}|p{3cm}|  }
			\hline
			\multicolumn{2}{|c|}{Scaling options for infinitely thin layer}\\
			\hline
			Choice S1 & Choice S2 \\
			\hline
			$\alpha=-1$ & $\alpha=-1$   \\
			$\beta=1$&$\beta\in (0,1)$\\
			 $\gamma\geq 1$&$\gamma\geq \beta$\\
			  $\xi\geq \frac{1}{2}$& $\xi\geq \min\{\beta-\frac{1}{2},0\}$\\
					\hline
		\end{tabular}
		\label{tablen1}
		\quad
			\begin{tabular}{ |p{3cm}|p{4cm}|  }
			\hline
			\multicolumn{2}{|c|}{Scaling options for finitely thin layer}\\
			\hline
			Choice S3 &Choice S4 \\
			\hline
			$\alpha\in (-1,\infty)$ &$\alpha\in (-1,\infty)$\\
		$\beta-\alpha=2$&$\beta-\alpha\in (1,\infty)\backslash\{2\}$\\
		$\gamma-\alpha\geq1$&$\gamma-\alpha\geq1$\\
		 $\xi-\alpha\geq 1$& $\xi-\alpha \geq 1$\\
			
			\hline
		\end{tabular}
		\label{tablen2}
		\end{table}
		\end{center}
		In Fig. \ref{fig4} and Fig. \ref{fig5}, we sketch the basic thin layers geometries we are handling here.
		\begin{figure}[ht]
		\begin{center}
			\begin{tikzpicture}
			\draw (0,0) node [anchor=north] {{\scriptsize $-\ell/2$}} to (6,0) node [anchor=north] {{\scriptsize $+\ell/2$}}  to (6,3.75) to (0,3.75) to (0,0);
			\draw (3,0)node [anchor=north] {{\scriptsize $0$}}   to (3,3.75)  ;
			\draw [->](.75,3.5) node[anchor=west] {{\scriptsize $\Gamma_{\mathcal{L}} $}} to (0,2.75) ;
			\draw[->](0,3.75) to (0,5) node[anchor=west] {{\scriptsize $e_{2}$}};
			\draw[->](6,0) to (7,0) node[anchor=west] {{\scriptsize $e_{1}$}};
			\draw[->](.5,1) node[anchor=west] {{\scriptsize $\Gamma_{h}$}} to (1,0);
			\draw [<->] (6.1,0) to (6.1,3.75);
			\draw (6.1,2) node[anchor=west] {{\scriptsize $h$}};
			\draw (1.25,2) node[anchor=west] {{\scriptsize $\Omega_{\mathcal{L}}$}};
			\draw (4.75,2) node[anchor=west] {{\scriptsize $\Omega_{\mathcal{R}}$}};
			\draw[->](5.25,3.5) node[anchor=west] {{\scriptsize $\Gamma_{h}$}} to (5.00,3.75);
			\draw[->] (4.24,3) node[anchor=west] {{\scriptsize $\Sigma$}} to (3,3.25);
			\end{tikzpicture}
			\caption{Schematic representation of the macroscopic model for infinitely thin layer. }
			\label{fig4}
		\end{center}
	\end{figure}
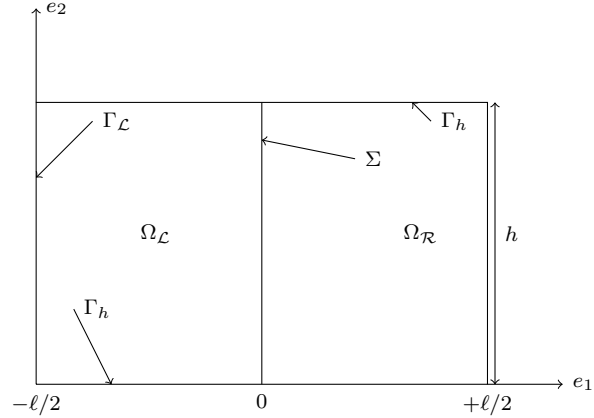
\quad
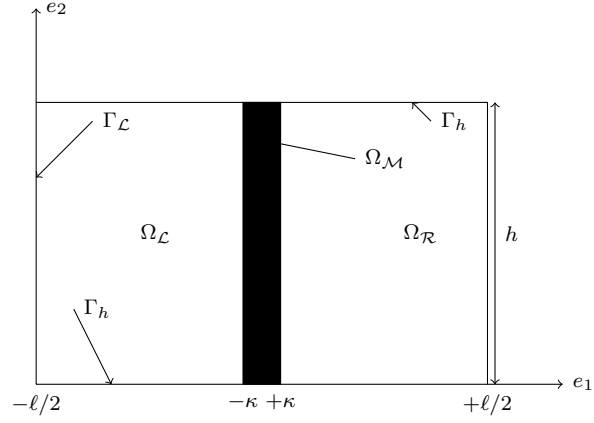
\begin{figure}[ht]
	\begin{center}
		\begin{tikzpicture}
		\draw (0,0) node [anchor=north] {{\scriptsize $-\ell/2$}} to (6,0) node [anchor=north] {{\scriptsize $+\ell/2$}}  to (6,3.75) to (0,3.75) to (0,0);
			\draw [fill](2.75,0) to (2.75,3.75) to(3.25,3.75) to (3.25,0) to (2.75,0);
		\draw [->](.75,3.5) node[anchor=west] {{\scriptsize $\Gamma_{\mathcal{L}} $}} to (0,2.75) ;
		\draw[->](0,3.75) to (0,5) node[anchor=west] {{\scriptsize $e_{2}$}};
		\draw[->](6,0) to (7,0) node[anchor=west] {{\scriptsize $e_{1}$}};
		\draw[->](.5,1) node[anchor=west] {{\scriptsize $\Gamma_{h}$}} to (1,0);
		\draw [<->] (6.1,0) to (6.1,3.75);
		\draw (2.75,0) node[anchor=north] {{\scriptsize $-\kappa$}};
		\draw (3.25,0) node[anchor=north] {{\scriptsize $+\kappa$}};
		\draw (6.1,2) node[anchor=west] {{\scriptsize $h$}};
		\draw (1.25,2) node[anchor=west] {{\scriptsize $\Omega_{\mathcal{L}}$}};
		\draw (4.75,2) node[anchor=west] {{\scriptsize $\Omega_{\mathcal{R}}$}};
		\draw[->](5.25,3.5) node[anchor=west] {{\scriptsize $\Gamma_{h}$}} to (5.00,3.75);
		\draw[->] (4.24,3) node[anchor=west] {{\scriptsize $\Omega_{\mathcal{M}}$}} to (3,3.25);
		\end{tikzpicture}
		\caption{Schematic representation of the macroscopic model for finitely thin layer. }
		\label{fig5}
	\end{center}
\end{figure}

\subsection{Macroscopic model for infinitely thin layer}\label{macmitl}
	\begin{theorem}\label{T8}
		Assume \ref{assump1}-\ref{assump7}. Then for scaling choice S1, $$(v_{l}^{\varepsilon},v_{m}^{\varepsilon},v_{r}^{\varepsilon})\in L^{2}(0,T;V_{\varepsilon})\cap H^{1}(0,T;L^{2}(\Omega_{\mathcal{L}}^{\varepsilon})\times L^{2}(\Omega_{\mathcal{M}}^{\varepsilon})\times L^{2}(\Omega_{\mathcal{R}}^{\varepsilon}) ),$$  satisfying $(P_{\varepsilon})$ in the sense of Definition 2.1 converges to $$(v_{l}^{0},v_{m}^{0},v_{r}^{0})\in  (L^{2}(0,T;H^{1}(\Omega_{\mathcal{L}})),L^{2}((0,T)\times \Sigma; H^{1}_{\#}(Z)),L^{2}(0,T;H^{1}(\Omega_{\mathcal{R}})))$$ which satisfies the identity
		\begin{multline}\label{t7}
			\int_{0}^{T}	\int_{\Omega_{\mathcal{L}}} \partial_{t} v_{l}^{0}\phi_{1}dxdt+\int_{0}^{T}\int_{\Omega_{\mathcal{L}}}D_{L}\nabla v_{l}^{0}\nabla \phi_{1}dxdt\\
			- \int_{0}^{T}\int_{\Omega_{\mathcal{L}}}D_{L}B_{L}P_{\delta}(v_{l}^{0}-u_{b})\nabla \phi_{1} dxdt\\
			+	\int_{0}^{T}	\int_{\Omega_{\mathcal{R}}} \partial_{t} v_{r}^{0}\phi_{3}dxdt+\int_{0}^{T}\int_{\Omega_{\mathcal{R}}}D_{R}\nabla v_{l}^{0}\nabla \phi_{1}dxdt\\
			- \int_{0}^{T}\int_{\Omega_{\mathcal{R}}}B_{R}P_{\delta}(v_{r}^{0}-u_{b}) \phi_{3} dxdt\\
			+\int_{0}^{T}\int_{\Sigma}\int_{ Z}\partial_{t}v_{m}^{0}(t,\bar{x},y)\phi_{2}(t,\bar{x},y)dyd\bar{x}dt\\
			+ \int_{0}^{T}\int_{\Sigma}\int_{ Z}D_{M}(y)\nabla_{y}v_{m}^{0}(t,\bar{x},y)\nabla_{y}\phi_{2}(t,\bar{x},y)dyd\bar{x}dt\\
			=\int_{0}^{t}\int_{\Omega_{\mathcal{L}}}f_{b_{l}}\phi_{1} dx dt-\int_{0}^{t}\int_{\Gamma_{h}\cap \partial\Omega_{\mathcal{L}}}g_{b_{l}}\phi_{1}d\sigma dt+\int_{0}^{t}\int_{\Omega_{\mathcal{R}}}f_{b_{r}}\phi_{3} dx dt\\
			-\int_{0}^{t}\int_{\Gamma_{h}\cap \partial\Omega_{\mathcal{R}}}g_{b_{r}}\phi_{3}d\sigma dt
			+\int_{0}^{T}\int_{\Sigma}\int_{ Z} f_{a_{0}}(t,\bar{x},y)\phi_{2}(t,\bar{x},y)dyd\bar{x}dt\\
			\int_{0}^{T}\int_{\Sigma}\int_{ Z}D_{L}\nabla_{\bar{x}} u_{b}(t,\bar{x},0)\cdot n_{l}\phi_{2}(t,\bar{x},\bar{y},-1)dyd\bar{x} dt
			\\
			- \int_{0}^{T}\int_{\Sigma}\int_{ Z}D_{R}\nabla_{\bar{x}} u_{b}(t,\bar{x},0)\cdot n_{l}\phi_{2}(t,\bar{x},\bar{y},+1)dyd\bar{x} dt
		\end{multline}
		for all $(\phi_{1},\phi_{3} )\in L^{2}((0,T);H^{1}(\Omega_{\mathcal{L}};\Gamma_{\mathcal{L}}))\times  L^{2}((0,T);H^{1}(\Omega_{\mathcal{R}};\Gamma_{\mathcal{R}}))$ and $\phi_{2} \in L^{2}((0,T)\times \Sigma\times Z),$ along with the initial condition
		 \begin{equation}\label{t7ic}
		\begin{aligned}
		v_{l}^{0}(0,x)&=h_{b_{{l}}}^{0}(x)  \mbox{      for all       } x\in {\overline{\Omega}_{\mathcal{L}}^{\varepsilon}},\\
		v_{r}^{0}(0,x)&=h_{b_{r}}^{0}(x)  \mbox{      for all       } x\in {\overline{\Omega}_{\mathcal{R}}^{\varepsilon}},\\
		v_{m}^{0}(0,\bar{x},y)&=h_{b_{m}}^{0}(\bar{x},y)  \mbox{      for all       } (\bar{x},y)\in \Sigma\times \overline{Z},
		\end{aligned}
		\end{equation}
		where the limit function $(v_{l}^{0},v_{m}^{0},v_{r}^{0})$ are given in Theorem \ref{T6} and Theorem \ref{T7}.
	\end{theorem}
\begin{proof}
We integrate the weak formulation  \eqref{wf} from $0$ to $T$ and choose  $$(\phi_{1},\phi_{3})\in  L^{2}((0,T);H^{1}(\Omega_{\mathcal{L}}^{\varepsilon};\Gamma_{\mathcal{L}}))\times  L^{2}((0,T);H^{1}(\Omega_{\mathcal{R}}^{\varepsilon};\Gamma_{\mathcal{R}}))$$ and \\
$\phi_{2}=\psi^{\varepsilon}(t,x):=\psi(t,\bar{x},\frac{x}{\varepsilon})\in L^{2}\left((0,T)\times \Sigma ;C_{\#}^{\infty}(Z)\right)$ with $\phi_{1}=\psi^{\varepsilon}   \mbox{  on  } \mathcal{B_{L}^{\varepsilon}},\;\phi_{3}=\psi^{\varepsilon}\mbox{  on  } \mathcal{B_{R}^{\varepsilon}}$, we get
\begin{multline}
\int_{0}^{T}\int_{\Omega_{\mathcal{L}}^{\varepsilon}} \partial_{t} v_{l}^{\varepsilon}\phi_{1}dxdt
+\int_{0}^{T}\int_{\Omega_{\mathcal{L}}^{\varepsilon}}D_{L}\nabla v_{l} ^{\varepsilon}\nabla\phi_{1} dxdt
-\int_{0}^{T}\int_{\Omega_{\mathcal{L}}^{\varepsilon}}B_{L}P_{\delta}(v_{l}^{\varepsilon}-u_{b})\nabla \phi_{1} dxdt\\
+\int_{0}^{T}\int_{\Omega_{\mathcal{R}}^{\varepsilon}} \partial_{t} v_{r}^{\varepsilon}\phi_{3}dxdt
+\int_{0}^{T}\int_{\Omega_{\mathcal{R}}^{\varepsilon}}D_{R}\nabla v_{r} ^{\varepsilon}\nabla\phi_{3} dxdt
-\int_{0}^{T}\int_{\Omega_{\mathcal{R}}^{\varepsilon}}B_{R}P_{\delta}(v_{r}^{\varepsilon}-u_{b})\nabla \phi_{3} dxdt\\
+\varepsilon^{\alpha}  { \int_{0}^{T}\int_{\Omega_{\mathcal{M}}^{\varepsilon}} \partial_{t} v_{m}^{\varepsilon}\psi(t,\bar{x},\frac{x}{\varepsilon})dxdt}
\\
+ \varepsilon^{\beta}\int_{0}^{T}\int_{\Omega_{\mathcal{M}}^{\varepsilon}}D_{M}^{\varepsilon}(\frac{x}{\varepsilon})\nabla v_{m} ^{\varepsilon}\left( \nabla_{x}\psi(t,\bar{x},\frac{x}{\varepsilon}) +\frac{1}{\varepsilon}\nabla_{y}\psi(t,\bar{x},\frac{x}{\varepsilon})\right) dxdt\\
-\varepsilon^{\gamma}\int_{0}^{T}\int_{\Omega_{\mathcal{M}}^{\varepsilon}}B_{M}^{\varepsilon}(\frac{x}{\varepsilon})P_{\delta}(v_{m}^{\varepsilon}-u_{b})\left( \nabla_{x}\psi(t,\bar{x},\frac{x}{\varepsilon}) +\frac{1}{\varepsilon}\nabla_{y}\psi(t,\bar{x},\frac{x}{\varepsilon})\right)dxdt\\
=\int_{0}^{t} \int_{\Omega_{\mathcal{L}}^{\varepsilon}}f_{b_{l}}\phi_{1} dx dt
-\int_{0}^{T}\int_{\Gamma_{h}\cap \partial\Omega_{\mathcal{L}}^{\varepsilon}}g_{b_{l}} \phi_{1} d\sigma dt
+\int_{0}^{T} \int_{\Omega_{\mathcal{R}}^{\varepsilon}}f_{b_{r}}\phi_{3} dx dt\\
-\int_{0}^{T}\int_{\Gamma_{h}\cap \partial\Omega_{\mathcal{R}}^{\varepsilon}}g_{b_{r}} \phi_{3} d\sigma dt
\\
+\varepsilon^{\alpha}\int_{0}^{T}\int_{\Omega_{\mathcal{M}}^{\varepsilon}}f_{a_{m}}^{\varepsilon}\psi(t,\bar{x},\frac{x}{\varepsilon}) dxdt +\varepsilon^{\beta}\int_{0}^{T}\int_{\Omega_{\mathcal{M}}^{\varepsilon}}f_{b_{m}}\psi(t,\bar{x},\frac{x}{\varepsilon}) dxdt\\
-\varepsilon^{\xi}\int_{0}^{T}\int_{\Gamma_{0}^{\varepsilon}}g_{0}^{\varepsilon} \psi(t,\bar{x},\frac{x}{\varepsilon}) d\sigma dt-\varepsilon^{\beta}\int_{0}^{T}\int_{\Gamma_{0}^{\varepsilon}}g_{b_{0}}\psi(t,\bar{x},\frac{x}{\varepsilon}) d\sigma dt\\
+	\int_{0}^{T}\int_{\mathcal{B_{L}^{\varepsilon}}}(D_{L}-\varepsilon^{\beta}D_{M})\nabla u_{b}\cdot n_{l}\phi_{1}d\sigma dt+\int_{0}^{T}\int_{\mathcal{B_{R}^{\varepsilon}}}(D_{R}-\varepsilon^{\beta}D_{M})\nabla u_{b}\cdot n_{r}\phi_{3}d\sigma dt.
\end{multline}
Now, using \eqref{t51}, \eqref{t55}, \eqref{t57} and \eqref{t59} for $\varepsilon\rightarrow 0 $, we obtain 
\begin{multline}\label{t7e1}
\int_{0}^{T}\int_{\Omega_{\mathcal{L}}^{\varepsilon}} \partial_{t} v_{l}^{\varepsilon}\phi_{1}dxdt
+\int_{0}^{T}\int_{\Omega_{\mathcal{L}}^{\varepsilon}}D_{L}\nabla v_{l} ^{\varepsilon}\nabla\phi_{1} dxdt
\\
-\int_{0}^{T}\int_{\Omega_{\mathcal{L}}^{\varepsilon}}B_{L}P_{\delta}(v_{l}^{\varepsilon}-u_{b})\nabla \phi_{1} dxdt
-\int_{0}^{T} \int_{\Omega_{\mathcal{L}}^{\varepsilon}}f_{b_{l}}\phi_{1} dx dt
+\int_{0}^{T}\int_{\Gamma_{h}\cap \partial\Omega_{\mathcal{L}}^{\varepsilon}}g_{b_{l}} \phi_{1} d\sigma dt\\
=\int_{0}^{T}	\int_{\Omega_{\mathcal{L}}} \partial_{t} v_{l}^{0}\phi_{1}dxdt+\int_{0}^{T}\int_{\Omega_{\mathcal{L}}}D_{L}\nabla v_{l}^{0}\nabla \phi_{1}dxdt\\
- \int_{0}^{T}\int_{\Omega_{\mathcal{L}}}D_{L}B_{L}P_{\delta}(v_{l}^{0}-u_{b})\nabla \phi_{1} dxdt\\
-\int_{0}^{T}\int_{\Omega_{\mathcal{L}}}f_{b_{l}}\phi_{1} dx dt+\int_{0}^{T}\int_{\Gamma_{h}\cap \partial\Omega_{\mathcal{L}}}g_{b_{l}}\phi_{1}d\sigma dt.
\end{multline}
Similarly, using \eqref{t52}, \eqref{t56}, \eqref{t58} and \eqref{t510} for $\varepsilon\rightarrow 0$, we get 
\begin{multline}\label{t7e2}
\int_{0}^{T}\int_{\Omega_{\mathcal{R}}^{\varepsilon}} \partial_{t} v_{r}^{\varepsilon}\phi_{3}dxdt
+\int_{0}^{T}\int_{\Omega_{\mathcal{R}}^{\varepsilon}}D_{R}\nabla v_{r} ^{\varepsilon}\nabla\phi_{3} dxdt
-\int_{0}^{T}\int_{\Omega_{\mathcal{R}}^{\varepsilon}}B_{R}P_{\delta}(v_{r}^{\varepsilon}-u_{b})\nabla \phi_{r} dxdt\\
-\int_{0}^{T} \int_{\Omega_{\mathcal{R}}^{\varepsilon}}f_{b_{r}}\phi_{3} dx dt
+\int_{0}^{T}\int_{\Gamma_{h}\cap \partial\Omega_{\mathcal{R}}^{\varepsilon}}g_{b_{r}} \phi_{3} d\sigma dt\\
=\int_{0}^{T}	\int_{\Omega_{\mathcal{R}}} \partial_{t} v_{r}^{0}\phi_{3}dxdt+\int_{0}^{T}\int_{\Omega_{\mathcal{R}}}D_{R}\nabla v_{r}^{0}\nabla \phi_{3}dxdt\\
- \int_{0}^{T}\int_{\Omega_{\mathcal{L}}}D_{R}B_{R}P_{\delta}(v_{r}^{0}-u_{b})\nabla \phi_{3} dxdt\\
-\int_{0}^{T}\int_{\Omega_{\mathcal{R}}}f_{b_{l}}\phi_{3} dx dt+\int_{0}^{T}\int_{\Gamma_{h}\cap \partial\Omega_{\mathcal{R}}}g_{b_{l}}\phi_{3}d\sigma dt.
\end{multline}
Now for $\alpha=-1$, $\beta=1$, $\gamma\geq 1$, $\xi\geq \frac{1}{2}$ and $\varepsilon\rightarrow 0$, we use Theorem 6 and obtain
\begin{equation}\label{t7e3}
\frac{1}{\varepsilon}  \int_{0}^{T}\int_{\Omega_{\mathcal{M}}^{\varepsilon}} \partial_{t} v_{m}^{\varepsilon}\psi(t,\bar{x},\frac{x}{\varepsilon})dxdt\rightarrow \int_{0}^{T}\int_{\Sigma}\int_{ Z}\partial_{t}v_{m}^{0}(t,\bar{x},y)\psi(t,\bar{x},y)dyd\bar{x}dt,
\end{equation}
\begin{equation}\label{t7e4}
\frac{1}{\varepsilon}\int_{0}^{T}\int_{\Omega_{\mathcal{M}}^{\varepsilon}}D_{M}^{\varepsilon}(\frac{x}{\varepsilon})\varepsilon^{2}\nabla v_{m} ^{\varepsilon} \nabla_{x}\psi(t,\bar{x},\frac{x}{\varepsilon})  dxdt\rightarrow 0,
\end{equation}

\begin{multline}\label{t7e5}
\frac{1}{\varepsilon}\int_{0}^{T}\int_{\Omega_{\mathcal{M}}^{\varepsilon}}D_{M}^{\varepsilon}(\frac{x}{\varepsilon})\varepsilon\nabla v_{m} ^{\varepsilon}\nabla_{y}\psi(t,\bar{x},\frac{x}{\varepsilon})dxdt\\
\rightarrow  \int_{0}^{T}\int_{\Sigma}\int_{ Z}D_{M}(y)\nabla_{y}v_{m}^{0}(t,\bar{x},y)\nabla_{y}\psi(t,\bar{x},y)dyd\bar{x}dt,
\end{multline}

\begin{equation}\label{t7e6}
\varepsilon^{1+\gamma}\int_{0}^{T}\int_{\Omega_{\mathcal{M}}^{\varepsilon}}B_{M}^{\varepsilon}(\frac{x}{\varepsilon})P_{\delta}(v_{m}^{\varepsilon}-u_{b}) \nabla_{x}\psi(t,\bar{x},\frac{x}{\varepsilon}) dxdt\rightarrow 0,
\end{equation}
\begin{equation}\label{t7e7}
\varepsilon^{\gamma}\int_{0}^{T}\int_{\Omega_{\mathcal{M}}^{\varepsilon}}B_{M}^{\varepsilon}(\frac{x}{\varepsilon})P_{\delta}(v_{m}^{\varepsilon}-u_{b})\nabla_{y}\psi(t,\bar{x},\frac{x}{\varepsilon})dxdt\rightarrow 0.
\end{equation}
Using \eqref{tsf}, we have
\begin{equation}\label{t7e8}
\frac{1}{\varepsilon}\int_{0}^{T}\int_{\Omega_{\mathcal{M}}^{\varepsilon}}f_{a_{m}}^{\varepsilon}\psi(t,\bar{x},\frac{x}{\varepsilon}) dxdt\rightarrow  \int_{0}^{T}\int_{\Sigma}\int_{ Z} f_{a_{0}}(t,\bar{x},y)\psi(t,\bar{x},y)dyd\bar{x}dt,
\end{equation}
\begin{equation}\label{t7e9}
\begin{aligned}
\varepsilon \int_{0}^{T}\int_{\Omega_{\mathcal{M}}^{\varepsilon}}f_{b_{m}}&\psi(t,\bar{x},\frac{x}{\varepsilon}) dxdt
=\varepsilon \int_{0}^{T}\int_{\Omega_{\mathcal{M}}^{\varepsilon}}-\mathrm{div}(D_{M}(\frac{x}{\varepsilon})\nabla u_{b})\psi(t,\bar{x},\frac{x}{\varepsilon}) dxdt\\
&=\varepsilon \int_{0}^{T}\int_{\Omega_{\mathcal{M}}^{\varepsilon}}D_{M}(\frac{x}{\varepsilon})\nabla u_{b}\left( \nabla_{x}\psi(t,\bar{x},\frac{x}{\varepsilon}) +\frac{1}{\varepsilon}\nabla_{y}\psi(t,\bar{x},\frac{x}{\varepsilon})\right) dxdt\\
&=\varepsilon \int_{0}^{T}\int_{\Omega_{\mathcal{M}}^{\varepsilon}}\nabla u_{b}\left(D_{M}(\frac{x}{\varepsilon})\right)^{t}\left( \nabla_{x}\psi(t,\bar{x},\frac{x}{\varepsilon}) +\frac{1}{\varepsilon}\nabla_{y}\psi(t,\bar{x},\frac{x}{\varepsilon})\right) dxdt\\
&\rightarrow  0.
\end{aligned}
\end{equation}
Using \eqref{tsg}, we obtain
\begin{equation}\label{t7e10}
\varepsilon^{\xi}\int_{0}^{T}\int_{\Gamma_{0}^{\varepsilon}}g_{0}^{\varepsilon} \psi(t,\bar{x},\frac{x}{\varepsilon}) d\sigma dt \rightarrow 0,
\end{equation}

\begin{equation}\label{t7e11}
\begin{aligned}
\varepsilon\int_{0}^{T}\int_{\Gamma_{0}^{\varepsilon}}g_{b_{0}}\psi(t,\bar{x},\frac{x}{\varepsilon}) d\sigma dt&=\varepsilon\int_{0}^{T}\int_{\Gamma_{0}^{\varepsilon}}-\nabla u_{b}\left(D_{M}^{\varepsilon}(\frac{x}{\varepsilon})\right)^{t}\psi(t,\bar{x},\frac{x}{\varepsilon}) d\sigma dt\\
&\rightarrow 0.
\end{aligned}
\end{equation}
\begin{multline}\label{t7e12}
\int_{\mathcal{B_{L}^{\varepsilon}}}(D_{L}-\varepsilon D_{M})\nabla u_{b}\cdot n_{l}\phi_{1}d\sigma+\int_{\mathcal{B_{R}^{\varepsilon}}}(D_{R}-\varepsilon D_{M})\nabla u_{b}\cdot n_{r}\phi_{3}d\sigma\\
=\int_{\mathcal{B_{L}^{\varepsilon}}}(D_{L}-\varepsilon D_{M})\nabla u_{b}\cdot n_{l}\psi\sigma-\int_{\mathcal{B_{R}^{\varepsilon}}}(D_{R}-\varepsilon^{\beta}D_{M})\nabla u_{b}\cdot n_{l}\psi d\sigma.
\end{multline}

\begin{equation}\label{t7e13}
\begin{aligned}
\int_{0}^{T}\int_{\mathcal{B_{L}^{\varepsilon}}}(D_{L}-\varepsilon D_{M}(\frac{x}{\varepsilon}))&\nabla u_{b}\cdot n_{l}\psi(t,\bar{x},\frac{x}{\varepsilon})d\sigma dt\\
&-\int_{0}^{T}\int_{\mathcal{B_{R}^{\varepsilon}}}(D_{R}-\varepsilon D_{M}(\frac{x}{\varepsilon}))\nabla u_{b}\cdot n_{l}\psi(t,\bar{x},\frac{x}{\varepsilon})d\sigma dt\\
&\rightarrow \int_{0}^{T}\int_{\Sigma}\int_{ Z}D_{L}\nabla_{\bar{x}} u_{b}(t,\bar{x},0)\cdot n_{l}\psi(t,\bar{x},\bar{y},-1)d\bar{x} dt\\
&\hspace{1cm}- \int_{0}^{T}\int_{\Sigma}\int_{ Z}D_{R}\nabla_{\bar{x}} u_{b}(t,\bar{x},0)\cdot n_{l}\psi(t,\bar{x},\bar{y},+1)d\bar{x} dt.\\
\end{aligned}
\end{equation}
Combining \eqref{t7e2}-\eqref{t7e12} yields the desired result \eqref{t7}.
\par For deriving initial conditions, first we choose $\phi_{1}\in C_{c}^{\infty}(\Omega_{\mathcal{L}})$ and \\
$\Theta(t)\in C^{\infty}([0,T])$ with $\Theta(T)=0$, then
\begin{equation}\label{t7e14}
\begin{aligned}
\int_{\Omega_{\mathcal{L}}}v_{l}^{0}(0,x)\phi_{1}(x)\Theta(0) dx&=-\int_{0}^{T}\int_{\Omega_{\mathcal{L}}}\partial_{t} v_{l}^{0}(t,x)\phi_{1}(x)\Theta(t) dxdt\\
&\hspace{1cm}-\int_{0}^{T}\int_{\Omega_{\mathcal{L}}} v_{l}^{0}(t,x)\phi_{1}(x)\partial_{t}\Theta(t) dxdt\\
&\hspace{2cm}+\int_{\Omega_{\mathcal{L}}}v_{l}^{0}(T,x)\phi_{1}(x)\Theta(T) dx\\
&=-\lim_{\varepsilon\rightarrow0}\int_{0}^{T}\int_{\Omega_{\mathcal{L}}^{\varepsilon}}\partial_{t} v_{l}^{\varepsilon}(t,x)\phi_{1}(x)\Theta(t) dxdt\\
&\hspace{2cm}-\lim_{\varepsilon\rightarrow0}\int_{0}^{T}\int_{\Omega_{\mathcal{L}}^{\varepsilon}} v_{l}^{\varepsilon}(t,x)\phi_{1}(x)\partial_{t}\Theta(t) dxdt\\
&=\lim_{\varepsilon\rightarrow0}\int_{\Omega_{\mathcal{L}}^{\varepsilon}}h_{b_{l}}^{\varepsilon}(x)\phi_{1}(x)\Theta(0)dx\\
&=\int_{\Omega_{\mathcal{L}}}h_{b_{l}}^{0}(x)\phi_{1}(x)\Theta(0) dx,
\end{aligned}
\end{equation}
here we used the assumption \eqref{A51}.\\
Similarly, for $\phi_{3}\in C_{c}^{\infty}(\Omega_{\mathcal{R}})$ and using \eqref{A52}, we get
\begin{equation}\label{t7e15}
\int_{\Omega_{\mathcal{R}}}v_{r}^{0}(0,x)\phi_{3}(x)\Theta(0) dx=\int_{\Omega_{\mathcal{R}}}h_{b_{r}}^{0}(x)\phi_{3}(x)\Theta(0) dx,
\end{equation}
and for $\psi\in C_{c}^{\infty}(\Sigma;C_{\#}^{\infty}(Z))$ and using \eqref{A53}, we get
\begin{equation}\label{t7e16}
\int_{ \Sigma}\int_{Z}v_{m}^{0}(0,\bar{x},y)\psi(\bar{x},y)\Theta(0) d\bar{x}y=\int_{ \Sigma}\int_{Z}h_{b_{m}}^{0}(\bar{x},y)\psi(\bar{x},y)\Theta(0) d\bar{x}y.
\end{equation}
From \eqref{t7e14}, \eqref{t7e15} and \eqref{t7e16} we get the desired result \eqref{t7ic}.
\end{proof}
\begin{theorem}\label{T9}
	Assume \ref{assump1}-\ref{assump7}. Then for scaling choice S1, the limit functions \\ $(v_{l}^{0},v_{m}^{0},v_{r}^{0})$ which are given in Theorem \ref{T6} and Theorem \ref{T7} satisfies the following boundary conditions
	\begin{align}
	v_{l}^{0}(t,\bar{x},0)=v_{m}^{0}(t,\bar{x},y) \hspace{1cm}\mbox{for a.e }(t,\bar{x},y)\in (0,T)\times \Sigma \times Z_{L}\label{T81}\\
		v_{r}^{0}(t,\bar{x},0)=v_{m}^{0}(t,\bar{x},y) \hspace{1cm}\mbox{for a.e }(t,\bar{x},y)\in (0,T)\times \Sigma \times Z_{R}\label{T82}.
	\end{align}
\end{theorem}
\begin{proof}
	To prove Theorem \ref{T9}, we use  same technique of Theorem 4.2 of \cite{gahn2020singular}. To prove \eqref{T81}, we choose $\psi\in C^{\infty}\left((0,T)\times\Sigma \times C_{\#}(\bar{Z})\right)$ such that $\psi(t,x,\cdot)$ has compact support in $Z_{L}\cup Z$
	Now, using integration by parts, Theorem 6 and \eqref{tc2}, we have
	\begin{equation}\label{}
	\begin{aligned}
\int_{0}^{T}\int_{ \Sigma}\int_{ Z}\nabla_{y}v_{m}^{0}&\psi dy d\bar{x}dt=\lim_{\varepsilon \rightarrow 0}\frac{1}{\varepsilon}\int_{0}^{T}\int_{\Omega_{\mathcal{M}}^{\varepsilon}}\varepsilon  \nabla v_{m}^{\varepsilon}\psi(t,x,\frac{x}{\varepsilon})dxdt\\
&=\lim_{\varepsilon \rightarrow 0}\left(\frac{-1}{\varepsilon}\int_{0}^{T}\int_{\Omega_{\mathcal{M}}^{\varepsilon}}v_{m}^{\varepsilon}\left(\nabla_{y}\psi(t,\bar{x},\frac{x}{\varepsilon})+\varepsilon \nabla_{x}\phi(t,x,\frac{x}{\varepsilon})\right)dxdt\right.\\
&\hspace{1.5cm}\left.+\int_{0}^{T}\int_{\mathcal{B_{L}^{\varepsilon}}}v_{m}^{\varepsilon}\psi(t,x,\frac{x}{\varepsilon})\cdot nd\sigma dt\right)\\
&=-\int_{0}^{T}\int_{ \Sigma}\int_{ Z}v_{m}^{0}\nabla_{y}\psi dy d\bar{x}dt\\
&\hspace{1cm}+\lim_{\varepsilon \rightarrow 0}\int_{0}^{T}\int_{\mathcal{B_{L}^{\varepsilon}}}v_{l}^{\varepsilon}\psi(t,x,\frac{x}{\varepsilon})\cdot nd\sigma dt\\
&=-\int_{0}^{T}\int_{ \Sigma}\int_{ Z}v_{m}^{0}\nabla_{y}\psi dy d\bar{x}dt\\
&\hspace{1cm}+\int_{0}^{T}\int_{ \Sigma}\int_{ Z_{L}}v_{l}^{0}\psi(t,\bar{x},y)\cdot n d\sigma d\bar{x}dt\\
&=\int_{0}^{T}\int_{ \Sigma}\int_{ Z}\nabla_{y}v_{m}^{0}\psi dy d\bar{x}dt-\int_{0}^{T}\int_{ \Sigma}\int_{ Z_{L}}v_{m}^{0}\psi(t,\bar{x},y)\cdot n d\sigma d\bar{x}dt\\
&\hspace{1cm}+\int_{0}^{T}\int_{ \Sigma}\int_{ Z_{L}}v_{l}^{0}\psi(t,\bar{x},y)\cdot n d\sigma d\bar{x}dt.\\
	\end{aligned}
	\end{equation}
	So, we obtain
	\begin{equation}\label{}
	\int_{0}^{T}\int_{ \Sigma}\int_{ Z_{L}}v_{m}^{0}\psi(t,\bar{x},y)\cdot n d\sigma d\bar{x}dt=\int_{0}^{T}\int_{ \Sigma}\int_{ Z_{L}}v_{l}^{0}\psi(t,\bar{x},y)\cdot n d\sigma d\bar{x}dt.
	\end{equation}
	which is equivalent to \eqref{T81}. Similarly by choosing test function from \\
	$C^{\infty}\left((0,T)\times \Sigma\times C_{\#}(\bar{Z})\right)$ such that $\psi(t,x,\cdot)$ has compact support in $Z_{R}\cup Z$ gives \eqref{T82}.
\end{proof}
\begin{theorem}\label{T11}
Assume \ref{assump1}-\ref{assump7}. Then for scaling choice S1, the limit function\\ $(v_{l}^{0},v_{m}^{0},v_{r}^{0})$ given in Theorem \ref{T6} and Theorem \ref{T7} is the weak solution of the following problem:
\begin{align}
v_{l}^{0}&\in L^{2}((0,T);H^{1}(\Omega_{\mathcal{L}}))\cap H^{1}((0,T);L^{2}(\Omega_{\mathcal{L}})) \\
v_{m}^{0}&\in L^{2}((0,T)\times\Sigma;H_{\#}^{1}({Z}))\cap H^{1}((0,T)\times \Sigma;L_{\#}^{2}(Z)) \\
v_{r}^{0}&\in L^{2}((0,T);H^{1}(\Omega_{\mathcal{R}}))\cap H^{1}((0,T);L^{2}(\Omega_{\mathcal{R}})) 
\end{align}
satisfying 
	\begin{equation}\label{meq}
\begin{aligned}
\frac{\partial v_{l}^{0}}{\partial t} +\mathrm{div}(-D_{L}\nabla v_{l}^{0}+ B_{L}P_{\delta}(v_{l}^{0}-u_{b}))&=f_{b_{l}}  &\mbox{on}  \ (0,T)\times\Omega_{\mathcal{L}} ,\\
\frac{\partial v_{r}^{0}}{\partial t} +\mathrm{div}(-D_{R}\nabla v_{r}^{0}+ B_{R}P_{\delta}(v_{r}^{0}-u_{b}))&=f_{b_{r}}  &\mbox{on}  \ \ (0,T)\times\Omega_{\mathcal{R}} ,\\
\end{aligned}
\end{equation}
	\begin{equation}\label{mbc1}
\begin{aligned}
v_{l}^{0}&=0   \mbox{       on      } (0,T)\times\Gamma_{\mathcal{L}}\\
v_{r}^{0}&=0  \mbox{       on      }(0,T)\times \Gamma_{\mathcal{R}}\\
\end{aligned}
\end{equation} 
\begin{equation}\label{mbc2}
\begin{aligned}
v_{l}^{0}(t,\bar{x},0)&=v_{m}^{0}(t,\bar{x},y)  \mbox{      for a.e     } (t,\bar{x},y)\in (0,T)\times \Sigma\times Z_{L},\\
	v_{r}^{0}(t,\bar{x},0)&=v_{m}^{0}(t,\bar{x},y) \mbox{     for a.e    }(t,\bar{x},y)\in (0,T)\times \Sigma \times Z_{R},\\
\end{aligned}
\end{equation} 
\begin{equation}\label{mbc3}
\begin{aligned}
(-D_{L}\nabla v_{l}^{0}+ B_{L}P_{\delta}(v_{l}^{0}-u_{b}))\cdot n_{l}&=  g_{b{_{l}}} \mbox{       on      } \left(\Gamma_{h}\cap \partial\Omega_{\mathcal{L}}\right)\times (0,T),\\
(-D_{R}\nabla v_{r}^{0}+ B_{R}P_{\delta}(v_{r}^{0}-u_{b}))\cdot n_{r}&=  g_{b_{{r}}} \mbox{       on      } \left(\Gamma_{h}\cap \partial\Omega_{\mathcal{R}}\right)\times (0,T),\\
\end{aligned}
\end{equation}
\begin{equation}\label{}
\begin{aligned}
v_{l}^{0}(0,x)&=h_{b_{l}}^{0}  \mbox{       on      } \overline{\Omega_{\mathcal{L}}}\\
v_{r}^{0}(0,x)&=h_{b_{r}}^{0}  \mbox{       on      } \overline{\Omega_{\mathcal{R}}}\\
\end{aligned}
\end{equation}
\begin{multline}\label{mbc4}
(-D_{L}\nabla v_{l}^{0}+ B_{L}P_{\delta}(v_{l}^{0}-u_{b})+D_{R}\nabla v_{r}^{0}- B_{R}P_{\delta}(v_{r}^{0}-u_{b}))\cdot n_{l}\\= \int_{ Z_{L}}D_{M}\nabla_{y}v_{m}^{0}\cdot n_{l}+D_{L}\nabla_{\bar{x}} u_{b}(t,\bar{x},0)\cdot n_{l}\\
\hspace{.05cm}- \int_{ Z_{R}}D_{M}\nabla_{y}v_{m}^{0}\cdot n_{l}+D_{R}\nabla_{\bar{x}} u_{b}(t,\bar{x},0)\cdot n_{l}\\
\mbox{on  }\hspace{2pt} (0,T) \times Z
\end{multline} 
and $v_{0}^{m}$ solves the following cell problem
\begin{equation}\label{}
\begin{aligned}
\frac{\partial v_{m}^{0}}{\partial t} +\mathrm{div_{y}}(-D_{M}\nabla_{y} v_{m}^{0})&=f_{a_{0}}  &\mbox{on}  \ &(0,T)\times\Sigma \times Z ,\\
(-D_{M}\nabla_{y} v_{m}^{0})\cdot n&=0 \;&\mbox{on}  &\ (0,T)\times\Sigma \times\left(\partial Z\backslash(Z_{L}\cup Z_{R})\right)\\
v_{l}^{0}(0,x)&=h_{b_{l}}^{0} & \mbox{ on }& {\Sigma} \times Z.
\end{aligned}
\end{equation}
\end{theorem}
\begin{proof}
The proof follows directly  from Theorem \ref{T8} and Theorem \ref{T9}.
\end{proof}
\begin{theorem}
	Assume \ref{assump1}-\ref{assump7}. Then for scaling choice S2, the macroscopic equation for $P_{\varepsilon}$ problem is:
	\begin{align}
	v_{l}^{0}&\in L^{2}((0,T);H^{1}(\Omega_{\mathcal{L}}))\cap H^{1}((0,T);L^{2}(\Omega_{\mathcal{L}})) \\
	v_{m}^{0}&\in L^{2}((0,T);\Sigma))\cap H^{1}((0,T); \Sigma) \\
	v_{r}^{0}&\in L^{2}((0,T);H^{1}(\Omega_{\mathcal{R}}))\cap H^{1}((0,T);L^{2}(\Omega_{\mathcal{R}})) 
	\end{align}
	satisfying 
	\begin{equation}\label{}
	\begin{aligned}
	\frac{\partial v_{l}^{0}}{\partial t} +\mathrm{div}(-D_{L}\nabla v_{l}^{0}+ B_{L}P_{\delta}(v_{l}^{0}-u_{b}))&=f_{b_{l}}  &\mbox{on}  \ (0,T)\times\Omega_{\mathcal{L}} ,\\
	\frac{\partial v_{r}^{0}}{\partial t} +\mathrm{div}(-D_{R}\nabla v_{r}^{0}+ B_{R}P_{\delta}(v_{r}^{0}-u_{b}))&=f_{b_{r}}  &\mbox{on}  \ \ (0,T)\times\Omega_{\mathcal{R}} ,\\
	\end{aligned}
	\end{equation}
	\begin{equation}\label{}
	\begin{aligned}
	v_{l}^{0}&=0   \mbox{       on      } (0,T)\times\Gamma_{\mathcal{L}}\\
	v_{r}^{0}&=0  \mbox{       on      }(0,T)\times \Gamma_{\mathcal{R}}\\
	\end{aligned}
	\end{equation} 
	\begin{equation}\label{}
	\begin{aligned}
	v_{l}^{0}(t,\bar{x},0)&=v_{m}^{0}(t,\bar{x})  \mbox{      for a.e     } (t,\bar{x})\in (0,T)\times \Sigma,\\
	v_{r}^{0}(t,\bar{x},0)&=v_{m}^{0}(t,\bar{x}) \mbox{     for a.e    }(t,\bar{x})\in (0,T)\times \Sigma,\\
	\end{aligned}
	\end{equation} 
	\begin{equation}\label{}
	\begin{aligned}
	(-D_{L}\nabla v_{l}^{0}+ B_{L}P_{\delta}(v_{l}^{0}-u_{b}))\cdot n_{l}&=  g_{b{_{l}}} \mbox{       on      } \left(\Gamma_{h}\cap \partial\Omega_{\mathcal{L}}\right)\times (0,T),\\
	(-D_{R}\nabla v_{r}^{0}+ B_{R}P_{\delta}(v_{r}^{0}-u_{b}))\cdot n_{r}&=  g_{b_{{r}}} \mbox{       on      } \left(\Gamma_{h}\cap \partial\Omega_{\mathcal{R}}\right)\times (0,T),\\
	\end{aligned}
	\end{equation}
	\begin{multline}\label{}
		(-D_{L}\nabla v_{l}^{0}+ B_{L}P_{\delta}(v_{l}^{0}-u_{b})+D_{R}\nabla v_{r}^{0}- B_{R}P_{\delta}(v_{r}^{0}-u_{b}))\cdot n_{l}\\
		=-|Z|\frac{\partial v_{0}^{m}}{\partial t}+\int_{ Z}f_{a_{0}}dy+ D_{L}\nabla_{\bar{x}} u_{b}(t,\bar{x},0)\cdot n_{l}
-D_{R}\nabla_{\bar{x}} u_{b}(t,\bar{x},0)\cdot n_{l},\\
\mbox{on}\hspace{.2cm} (0,T)\times \Sigma	\end{multline} 

	\begin{align}
	v_{l}^{0}(0,x)&=h_{b_{l}}^{0}  \mbox{       on      } \overline{\Omega}_{\mathcal{L}}\\
	v_{r}^{0}(0,x)&=h_{b_{r}}^{0}  \mbox{       on      } \overline{\Omega}_{\mathcal{R}}\\
	v_{m}^{0}(0,x)&=\frac{1}{|Z|}\int_{ Z}h_{b_{m}}^{0}(\bar{x},y)dy  \mbox{       on      } \overline{\Sigma}.
	\end{align}

\end{theorem}
\begin{proof}
    Proof is application of Theorem \ref{T6} and Theorem \ref{T7} and follows via similar technique of proof of Theorem \ref{T11} and convergence results from  \cite{effectiveapratim}.
\end{proof}
\subsection{Macroscopic equation for finitely thin layer}\label{macroeftl}
To derive macroscopic equation for finitely thin layer we use the following assumption \ref{assumptionb1}, \ref{assumptionb2}, and \ref{assumptionb3} instead of \ref{assump3}, \ref{assump4} and \ref{assump5}
 \begin{enumerate}[label=({B}{{\arabic*}})]

 	 \item \label{assumptionb1}
For the reaction rate, we assume 
 	 $f_{b_{l}},\partial_{t}f_{b_{l}}\in L^{2}(0,T;L^{2}(\Omega_{\mathcal{L}}))$, \\
 	 $f_{b_{r}},\partial_{t}f_{b_{r}}\in L^
{2}(0,T;L^{2}(\Omega_{\mathcal{R}}))$,  $f_{b_{m}}^{\varepsilon},\partial_{t}f_{b_{m}}^{\varepsilon}\in L^{2}(0,T;L^{2}(\Omega_{\mathcal{M}}^{\varepsilon}))$ and \begin{equation}\label{}
 	 \varepsilon^{\alpha}\|f_{a_{m}}^{\varepsilon}\|_{L^{2}(0,T;L^{2}(\Omega_{\mathcal{M}}^{\varepsilon})}\leq C,
 	 \end{equation}
 	 for $a.e.$ $t\in (0,T).$Together we assume there exist $f_{a_{0}}\in L^{2}((0,T)\times \Omega_{\mathcal{M}}\times Z)$ such that
 	 	\begin{equation}\label{}
 	 	f_{a_{m}}^{\varepsilon}\overset{2-s}{\rightharpoonup}f_{a_{0}}.
 	 	\end{equation}
 	 \item \label{assumptionb2}
 	 $g_{b_{l}},\partial_{t}g_{b_{l}} \in L^{\infty}(0,T;L^{2}(\Gamma_{h}\cap \partial\Omega_{\mathcal{L}}))$,  $g_{b_{r}},\partial_{t}g_{b_{r}} \in L^{\infty}(0,T;L^{2}(\Gamma_{h}\cap \partial\Omega_{\mathcal{R}}))$, \\ $g_{b_{
0}}^{\varepsilon},\partial_{t}g_{b_{0}}^{\varepsilon} \in L^{\infty}(0,T;L^{2}(\Gamma_{0}^{\varepsilon}))$, $g_{{0}}^{\varepsilon},\partial_{t}g_{{0}}^{\varepsilon} \in L^{\infty}(0,T;L^{2}(\Gamma_{0}^{\varepsilon}))$ and 
 	 
 	 \begin{equation}\label{}
 	 \varepsilon^{\xi-\frac{1}{2}}\|g_{0}^{\varepsilon}\|_{L^{2}(\Gamma_{0}^{\varepsilon})}^{2}\leq C,
 	 \end{equation}
 	  \begin{equation}\label{}
 	 \varepsilon^{\beta-\frac{1}{2}}\|g_{b_{0}}^{\varepsilon}\|_{L^{2}(\Gamma_{0}^{\varepsilon})}^{2}\leq C,
  \end{equation}
 	 
 	 for $a.e.$ $t\in (0,T)$. Together we assume there exist $g_{0}\in L^{2}((0,T)\times \Omega_{\mathcal{M}} \times \partial Y_{0})$ such that
 	 \begin{equation}\label{}
 	 g_{0}^{\varepsilon}\overset{2-s}{\rightharpoonup}g_{0}.
 	 \end{equation}
 	 \item \label{assumptionb3}
 	 For initial conditions, we assume  $h_{b_{l}}^{\varepsilon}\in H^{1}(\Omega_{\mathcal{L}})$, $h_{b_{r}}^{\varepsilon}\in H^{1}(\Omega_{\mathcal{R}})$, $h_{b_{m}}^{\varepsilon}\in H^{1}(\Omega_{\mathcal{M}}^{\varepsilon})$  with
 	 \begin{equation}\label{}
 	 \|h_{b_{l}}^{\varepsilon}\|_{L^{2}(\Omega_{\mathcal{L}}^{\varepsilon})}^{2}+\|h_{b_{r}}^{\varepsilon}\|_{L^{2}(\Omega_{\mathcal{R}}^{\varepsilon})}^{2}+\varepsilon^{\alpha}\|h_{b_{m}}^{\varepsilon}\|_{L^{2}(\Omega_{\mathcal{M}}^{\varepsilon})}^{2}\leq C,
 	 \end{equation} \\
 	and 
 		\begin{align}
 	\mathbbm{1}_{\Omega_{\mathcal{L}}^{\varepsilon}}h_{b_{l}}^{\varepsilon}&\rightarrow h_{b_{l}}^{0}\hspace{2cm}\mbox{on}\hspace{1cm}L^{2}((0,T)\times\Omega_{\mathcal{L}}),\label{}\\
 	\mathbbm{1}_{\Omega_{\mathcal{R}}^{\varepsilon}}h_{b_{r}}^{\varepsilon}&\rightarrow h_{b_{r}}^{0}\hspace{2cm}\mbox{on}\hspace{1cm}L^{2}((0,T)\times\Omega_{\mathcal{R}}),\label{}\\
 	h_{b_{m}}^{\varepsilon}&\overset{2-s}{\rightharpoonup} h_{b_{r}}^{0}.\label{}
 	\end{align}
 
 \end{enumerate}
 On assumption \ref{assumptionb1}, \ref{assumptionb2} and \ref{assumptionb3} we use two scale convergence definition from \cite{lukkassen2002two}.
\begin{theorem}\label{T12}
	Assume \ref{assump1}, \ref{assump2}, \ref{assump6}, \ref{assump7} and \ref{assumptionb1}-\ref{assumptionb3}. Then for scaling choice S3, the macroscopic equation for $(P_{\varepsilon})$ problem is:
		\begin{align}
	v_{l}^{0}&\in L^{2}((0,T);H^{1}(\Omega_{\mathcal{L}}))\cap H^{1}((0,T);L^{2}(\Omega_{\mathcal{L}})) \\
	v_{m}^{0}&\in L^{2}((0,T)\times \Omega_{\mathcal{M}};H^{1}_{\#}(Z))\cap H^{1} ((0,T)\times \Omega_{\mathcal{M}};L_{\#}^{2}(Z))\\
	v_{r}^{0}&\in L^{2}((0,T);H^{1}(\Omega_{\mathcal{R}}))\cap H^{1}((0,T);L^{2}(\Omega_{\mathcal{R}})) 
	\end{align}
	satisfying
		\begin{equation}\label{}
	\begin{aligned}
	\frac{\partial v_{l}^{0}}{\partial t} +\mathrm{div}(-D_{L}\nabla v_{l}^{0}+ B_{L}P_{\delta}(v_{l}^{0}-u_{b}))&=f_{b_{l}}  &\mbox{on}  \ (0,T)\times\Omega_{\mathcal{L}} ,\\
	\frac{\partial v_{r}^{0}}{\partial t} +\mathrm{div}(-D_{R}\nabla v_{r}^{0}+ B_{R}P_{\delta}(v_{r}^{0}-u_{b}))&=f_{b_{r}}  &\mbox{on}  \ \ (0,T)\times\Omega_{\mathcal{R}} ,\\
	\end{aligned}
	\end{equation}
		\begin{equation}\label{}
	\begin{aligned}
	v_{l}^{0}&=0   \mbox{       on      } (0,T)\times\Gamma_{\mathcal{L}}\\
	v_{r}^{0}&=0  \mbox{       on      }(0,T)\times \Gamma_{\mathcal{R}}\\
	\end{aligned}
	\end{equation} 
	\begin{equation}\label{}
	\begin{aligned}
	(-D_{L}\nabla v_{l}^{0}+ B_{L}P_{\delta}(v_{l}^{0}-u_{b}))\cdot n_{l}&=  g_{b{_{l}}} \mbox{       on      } \left(\Gamma_{h}\cap \partial\Omega_{\mathcal{L}}\right)\times (0,T),\\
	(-D_{R}\nabla v_{r}^{0}+ B_{R}P_{\delta}(v_{r}^{0}-u_{b}))\cdot n_{r}&=  g_{b_{{r}}} \mbox{       on      } \left(\Gamma_{h}\cap \partial\Omega_{\mathcal{R}}\right)\times (0,T),\\
	\end{aligned}
	\end{equation}
	\begin{equation}\label{}
		\frac{\partial v_{m}^{0}}{\partial t} +\mathrm{div}_{y_{2}}(-\lambda_{1}D_{M}\nabla_{y_{2}} v_{m}^{0}+ \lambda_{2} B_{M}P_{\delta}(v_{m}^{0}-u_{b}))=f_{a_{m}}\hspace{1cm} \mbox{on}\;\;  \ (0,T)\times\Omega_{\mathcal{M}}\times
		Z ,
	\end{equation}
	\begin{equation}\label{}
	(-\lambda_{1}D_{M}\nabla_{y_{2}} v_{m}^{0}+ \lambda_{2} B_{M}P_{\delta}(v_{m}^{0}-u_{b}))\cdot n_{m}=g_{0}\hspace{1cm} \mbox{on}\;\;  \ (0,T)\times\Omega_{\mathcal{M}}\times
	\partial Y_{0}.
	\end{equation}
		\begin{align}
	v_{l}^{0}(t,\bar{x},0)=v_{m}^{0}(t,\bar{x},y) \hspace{1cm}\mbox{for a.e }(t,\bar{x})\in (0,T)\times \mathcal{B_{L}}\times Z_{L} \label{nn}\\
		v_{r}^{0}(t,\bar{x},0)=v_{m}^{0}(t,\bar{x},y) \hspace{1cm}\mbox{for a.e }(t,\bar{x})\in (0,T)\times \mathcal{B_{R}}\times Z_{R}\label{nnn}.
	\end{align}
	\begin{multline}\label{}
		(-D_{L}\nabla v_{l}^{0}+ B_{L}P_{\delta}(v_{l}^{0}-u_{b}))\cdot n_{l}\\=\int_{ Z_{L}}(-\lambda_{1}D_{M}\nabla_{y_{2}}v_{m}^{0}+\lambda_{2} B_{M}P_{\delta}(v_{m}^{0}-u_{b}))\cdot n_{l}dy_{2}-D_{L}\nabla u_{b}\cdot n_{l}\hspace{0.5cm} \mbox{on}\;  \ (0,T)\times\mathcal{B_{L}}.
	\end{multline}
		\begin{multline}\label{}
	(-D_{R}\nabla v_{l}^{0}+ B_{R}P_{\delta}(v_{l}^{0}-u_{b}))\cdot n_{r}\\=\int_{ Z_{R}}(-D_{M}\nabla_{y_{2}}v_{m}^{0}+B_{M}P_{\delta}(v_{m}^{0}-u_{b}))\cdot n_{r}dy_{2}-D_{R}\nabla u_{b}\cdot n_{r}\\ \mbox{on}\;  \ (0,T)\times\mathcal{B_{R}},
	\end{multline}
\begin{align}
v_{l}^{0}(0,x)&=h_{b_{l}}^{0}(x)  \mbox{       on      } \overline{\Omega_{\mathcal{L}}}\\
v_{r}^{0}(0,x)&=h_{b_{r}}^{0}(x)  \mbox{       on      } \overline{\Omega_{\mathcal{R}}}\\
v_{m}^{0}(0,x,y)&=h_{b_{m}}^{0}(x,y)  \mbox{       on      } \overline{\Omega}_{\mathcal{M}}\times \overline{Z},
\end{align}
	where $\lambda_{1}=1$, and $\lambda_{2}=1$ if $\gamma-\alpha =1$ and  $\lambda_{2} =0$ if $\gamma-\alpha >1$.
	\begin{proof}
	    The proof follows from Theorem 2 and the two scale compactness result from \cite{lukkassen2002two}.
	\end{proof}
\end{theorem}
\begin{remark}\label{choise s4}
The working strategy to obtain macroscopic equation for choice S4 is similar to that used to obtain the macroscopic equation for choice S3. The only difference in  macroscopic equations for the choice S3 and choice S4  is the value of $\lambda_{1}$.  We obtain $\lambda_{1}=1$ for the choice S3 while we obtain $\lambda_{1}=0$ for the choice S4.
\end{remark}
\section{Approximation of non-regularized problem}\label{anrp}

In this section, we propose a strategy that allows the vanishing of the parameter $\delta$ arising in our regularized nonlinear reaction-diffusion-convection problem, i.e. we replace the nonlinear operator $P_{\delta}(\cdot) $ cf.  \eqref{rd1}  by $P(\cdot)$ as defined in \eqref{pd2} and comment on what possibilities are available to handle a fully nonlinear oscillating drift. 
\par Within this section, we refer to the jointly $\varepsilon$- and $\delta$-dependent problem  \eqref{wf} and \eqref{ic}  as problem  $\left({P_{\varepsilon}^{\delta}}\right)$. Similarly, the  $\delta$ independent problem where $P_{\delta}(\cdot) $ replaced by $P(\cdot)$ in \eqref{wf} and \eqref{ic} is referred to as the  $\left({P_{\varepsilon}^{0}}\right)$ problem. What concerns the macroscopic equation \eqref{t7} with initial condition \eqref{t7ic}, we call it the $\left({P_{0}^{\delta}}\right)$ problem. Finally, we denote the  $\varepsilon$ and $\delta$ independent macroscopic equation   as the $\left({P_{0}^{0}}\right)$ problem. It appears  anytime $P_{\delta}(\cdot)$ is replaced in \eqref{t7} with \eqref{t7ic} by $P(\cdot)$. 

The hypotheses on data and parameters needed for the solvability of problems $\left({P_{\varepsilon}^{\delta}}\right)$, $\left({P_{\varepsilon}^{0}}\right)$, $\left({P_{0}^{\delta}}\right)$, and $\left({P_{0}^{0}}\right)$ are assumed to hold. In such case, the following approximation results hold:

\begin{theorem}\label{T13}
    If $(v_{l}^{\varepsilon,\delta},v_{m}^{\varepsilon,\delta},v_{r}^{\varepsilon,\delta})$ is the weak solution of $\left({P_{\varepsilon}^{\delta}}\right)$ and $(v_{l}^{\varepsilon,0},v_{m}^{\varepsilon,0},v_{r}^{\varepsilon,0})$ is the weak solution of $\left({P_{\varepsilon}^{0}}\right)$ problem, then as $\delta \rightarrow 0$,  $(v_{l}^{\varepsilon,\delta},v_{m}^{\varepsilon,\delta},v_{r}^{\varepsilon,\delta})$ $\rightarrow$ $(v_{l}^{\varepsilon,0},v_{m}^{\varepsilon,0},v_{r}^{\varepsilon,0})$ weakly in $L^{2}((0,T);H(\Omega_{\mathcal{L}}^{\varepsilon};\Gamma_{\mathcal{L}}))\times L^{2}((0,T);H^{1}(\Omega_{\mathcal{M}}^{\varepsilon}))\times H(\Omega_{\mathcal{R}}^{\varepsilon};\Gamma_{\mathcal{R}})$.
\end{theorem}
\begin{proof}
   To prove this result we rely on the basic working ideas from \cite{Pokorny}. The proof follows via a direct application of the convolution property (see Theorem 4.22 from \cite{brezis2010functional}). We take $\delta \rightarrow 0 $ in $\left({P_{\varepsilon}^{\delta}}\right)$ and apply the property of convolution which is  $P_{\delta}(r) \rightarrow {P}(r)$ in $L^{2}(\mathbb{R})$ strongly as $\delta \rightarrow 0$. See \cite{Pokorny} for related arguments.
\end{proof}

\begin{theorem}\label{T14}
    If $(v_{l}^{0,\delta},v_{m}^{0,\delta},v_{r}^{0,\delta})$ is the weak solution of $\left({P_{0}^{\delta}}\right)$ and $(v_{l}^{0,0},v_{m}^{0,0},v_{r}^{0,0})$ is the weak solution of $\left({P_{0}^{0}}\right)$ problem, then as $\delta \rightarrow 0$,  $(v_{l}^{0,\delta},v_{m}^{0,\delta},v_{r}^{0,\delta})$ $\rightarrow$ $(v_{l}^{0,0},v_{m}^{0,0},v_{r}^{0,0})$ weakly in $L^{2}((0,T);H^{1}(\Omega_{\mathcal{L}}))\times L^{2}((0,T)\times\Sigma;H_{\#}^{1}({Z}))\times L^{2}((0,T);H^{1}(\Omega_{\mathcal{R}}))$.
\end{theorem}

\begin{proof}
    The proof follows similar lines as when proving  Theorem \ref{T13}.
\end{proof}

Combining Theorem \ref{T13} and Theorem \ref{T14} , we  conclude that the weak solution to $\left({P_{\varepsilon}^{0}}\right)$ can be approximated in terms of the weak  solution to $\left({P_{0}^{0}}\right)$. We indicate this fact in  flowchart shown in  Fig \ref{flowchart}.

\tikzstyle{io}=[rectangle,rounded corners, minimum width=5em, minimum height=3em, text centered, draw=black]
\tikzstyle{arrow}=[thick, ->, >=stealth]

\begin{figure}[ht]
		\begin{center}
			\begin{tikzpicture}{node distance=3em, auto}
			\node (ed) [io] {$\left({P_{\varepsilon}^{0}}\right)$};
			\node (e0) [io, below of=ed, yshift=10em] {$\left({P_{\varepsilon}^{\delta}}\right)$};
			\node (0d) [io, right of=ed, xshift=10em] {$\left({P_{0}^{0}}\right)$};
			\node (00) [io, below of=0d, yshift=10em] {$\left({P_{0}^{\delta}}\right)$};
			\draw[arrow] (e0) -- node [anchor=south]{$\varepsilon \rightarrow 0$}(00);
			\draw[arrow] (e0) -- node [anchor=east]{$\delta \rightarrow 0$}(ed);
			\draw[arrow] (00) -- node [anchor=east]{$\delta \rightarrow 0$}(0d);
			\end{tikzpicture}
			\caption{Flowchart showing links between our parameter-dependent problems. }
			\label{flowchart}
		\end{center}
	\end{figure}
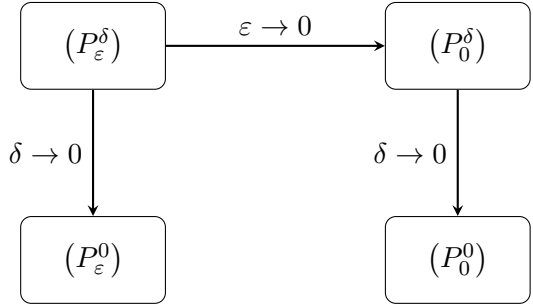

\section{Conclusion}\label{conclusion}

Starting off from a setting involving reaction-diffusion with nonlinear drift crossing a periodically perforated layer,  we derived upscaled equations, some of them reduced dimensionally, as well as effective transmission conditions for different choice of scalings in terms of a small heterogeneity parameter called $\varepsilon$ for the diffusion and drift transport terms as well as for the microscopic surface reaction rates. To pass to the homogenization limit $\varepsilon\to 0$, we used both the classical concept of two-scale convergence (see e.g. \cite{nguetseng1989general}, \cite{allaire1992homogen}) as well as  the concept of two-scale layer convergence (see \cite{neuss2007effective}), depending on the used parametric scaling. The second type of convergence is able to handle simultaneously periodic homogenization and dimension reduction limits.

\par A number of distinct limit upscaled model equations have been obtained in this framework. It is worth noting that our list is not exhaustive. Some more cases can be added. However, we believe that these options are potentially the most relevant ones if one has in mind the physical problem. At this moment, we are unable to classify, in the spirit of Occam’s razor, which of these models is best. A robust multiscale numerical approach as well as access to flux measurements for a given flat membrane with controlled regular internal structure are ingredients needed to make such comparisons. This is yet to be done. 


\par We studied here only the 2D case. We did  that because the derivation of the original problem has been done for an interacting particle system in 2D. Our convergence results  extend to higher dimensions without additional mathematical difficulties. However, we expect that eventual numerical approximations of the proposed problems are harder in 3D compared to 2D. Notice also the fact that our rectangular microstructures can be replaced in theory by any other type of inclusion having Lipschitz boundary and satisfying the restriction $\partial Y \cap  Y_{0}=\emptyset$. 

\par We expect that the diagram shown in Figure \ref{flowchart} is commutative. However, more mathematical results still need to be obtain to support such statement. The main issue is that, currently, we do not control in a parameter independent way the non-regularized nonlinear drift.  We believe that the way of working proposed in \cite{Marui2000TwoscaleCF} will turn to be useful to clarify this matter. It is also worth to study the corrector estimates of our problem since it can give idea about how good our approximation is. We expect that the method proposed e.g. in \cite{VoAnh1282488} and in \cite{gahn2021correctors} can be used to derive corrector estimates for our problem.


\section*{Acknowledgements} The work of V.R. and A.M. is partly supported by the project  "Homogenization and dimension reduction of thin heterogeneous layers", grant nr. VR 2018-03648 of the Swedish Research Council.
\par A.M. thanks Maria Neuss-Radu (Erlangen, Germany) and Willi J\"ager (Heidelberg, Germany) for many discussions on this and related topics during the last years. This work has been finalized during the friendly and inspiring atmosphere of the Institut Mittag Leffler workshop "New trends in numerical multiscale methods and beyond", Stockholm, July 12 -- 16, 2021. 

\bibliographystyle{amsplain}
	\bibliography{thinlayerhomogenization}

\end{document}